\documentclass{article}
\usepackage{amsmath, amsthm, amssymb,amsfonts}
\usepackage{tikz-cd}
\usepackage[english]{babel}
\usepackage{float}
\usepackage[a4paper,top=3cm,bottom=3cm,left=3cm,right=3cm,marginparwidth=1.75cm]{geometry}
\usepackage{bbm}
\usepackage{setspace}
\usepackage{quiver}

\usepackage[
colorlinks=true,
urlcolor=purple,
linkcolor=blue,
anchorcolor=blue,
citecolor=cyan,
menucolor=green,
urlcolor=blue,
linkcolor=purple!87!black,
]{hyperref}

\usepackage[capitalise,nameinlink]{cleveref}
\crefdefaultlabelformat{\bfseries#2#1#3}
\newtheorem{theorem}{Theorem}[section]
\newtheorem{lemma}[theorem]{Lemma}
\newtheorem{proposition}[theorem]{Proposition}
\theoremstyle{definition}
\newtheorem{definition}[theorem]{Definition}

\newtheorem{example}[theorem]{Example}
\newtheorem{notation}[theorem]{Notation}
\newtheorem{construction}[theorem]{Construction}
\newtheorem{corollary}[theorem]{Corollary}
\theoremstyle{remark}
\newtheorem{remark}[theorem]{Remark}
\DeclareMathOperator{\colim}{colim}
\definecolor{DefColor}{rgb}{0.6,0.15,0.25}

\newcommand{\Fun}{\mathrm{Fun}}

\newcommand{\Mod}{\mathrm{Mod}}
\newcommand{\Map}{\mathrm{Map}}
\newcommand{\Hom}{\mathrm{Hom}}
\newcommand{\Pure}{\mathrm{Pure}}
\newcommand{\Rex}{\mathrm{Rex}}
\newcommand{\Cat}{\mathrm{Cat}}
\newcommand{\Ob}{\mathrm{Ob}}
\newcommand{\Mor}{\mathrm{Mor}}

\newcommand{\Ind}{\mathrm{Ind}}
\newcommand{\CAlg}{\mathrm{CAlg}}
\newcommand{\ev}{\mathrm{ev}}
\newcommand{\op}{\mathrm{op}}
\newcommand{\Ex}{\mathrm{Ex}}
\newcommand{\Ret}{\mathrm{Ret}}
\newcommand{\wRet}{\mathrm{wRet}}
\newcommand{\Id}{\mathrm{Id}}
\newcommand{\cC}{\mathcal{C}}
\newcommand{\cD}{\mathcal{D}}
\newcommand{\cE}{\mathcal{E}}
\newcommand{\cT}{\mathcal{T}}
\newcommand{\cP}{\mathcal{P}}
\newcommand{\sset}{\mathrm{Set}^{\Delta^{\mathrm{op}}}}
\newcommand{\seset}{\mathrm{Set}^{(\Delta^+)^{\mathrm{op}}}}

\begin{document}
\title{The $\infty$-Categorical Reflection Theorem and Applications}
\date{}
\author{Shaul Ragimov \and Tomer M. Schlank}
\maketitle 
\begin{abstract}
   We prove an $\infty$-categorical version of the reflection theorem of \cite{AdamekRosicky89}. Namely, that a full subcategory of a presentable $\infty$-category which is closed under limits and $\kappa$-filtered colimits is a presentable $\infty$-category. We then use this theorem in order to classify subcategories of a symmetric monoidal $\infty$-category which are equivalent to a category of modules over an idempotent algebra.
\end{abstract}
\tableofcontents
\hypersetup{linkcolor=blue}
\clearpage

\section{Introduction}

Presentable $\infty$-categories provide a natural setting for modern homotopy theory and higher algebra. These are large $\infty$-categories that admit all small colimits and are generated under $\kappa$-filtered colimits by a small subcategory of $\kappa$-compact objects. Familiar examples include the $\infty$-categories of spaces, spectra, chain complexes, and sheaves of spaces on a site. These categories play a role in higher category theory similar to that of locally presentable categories in ordinary category theory: they serve as ambient categories for doing mathematics, where constructions like colimits and adjunctions are available and well-behaved.

In fact, in the higher-categorical world, presentability plays an even more crucial role. In this setting, where direct, “by-hand” definitions and constructions are often substantially more intricate, many arguments rely on abstract categorical principles. These principles frequently depend on some form of presentability to ensure that constructions are well-defined and behave as expected. As a result, presentability becomes not just a convenience but a necessity for working effectively in higher category theory.

A fundamental feature of presentable $\infty$-categories is that they support powerful adjoint functor theorems~\cite[Corollary 5.5.2.9]{lurie2009higher}. In this context, an important class of subcategories are the \emph{reflective subcategories}, that is, full subcategories $\cD \subset \cC$ such that the inclusion admits a left adjoint. Examples include the sheafification functor for subcategories of presheaves and Bousfield localizations in the category of spectra.

This leads to a natural and important question:
\begin{quote}
\emph{Given a full subcategory $\cD$ of a presentable $\infty$-category $\cC$, under what conditions is $\cD$ also presentable, and when does the inclusion $\cD \hookrightarrow \cC$ admit a left adjoint?}
\end{quote}

In the $1$-categorical setting, this question was settled by Adámek and Rosický~\cite{AdamekRosicky89}, who proved that if $\cD$ is closed under limits and $\kappa$-filtered colimits (for some regular cardinal $\kappa$), then $\cD$ is presentable and reflective in $\cC$. This result, known as the \emph{reflection theorem}, has become a fundamental tool in the theory of locally presentable categories.

The main goal of this paper is to establish an $\infty$-categorical generalization of this theorem:

\begin{theorem}[\cref{theorem: reflaction}]\label{theorem: thm1}
Let $\cC$ be a presentable $\infty$-category and $\cD \subset \cC$ a full subcategory. Then the following are equivalent:
\begin{enumerate}
    \item $\cD$ is closed in $\cC$ under all limits and $\kappa$-filtered colimits for some regular cardinal $\kappa$.
    \item $\cD$ is presentable and the inclusion $\cD \hookrightarrow \cC$ admits a left adjoint.
\end{enumerate}
\end{theorem}

The implication $(2) \Rightarrow (1)$ follows formally from~\cite[5.5.2.9]{lurie2009higher}. However, the converse direction $(1) \Rightarrow (2)$ is more subtle and requires a genuinely higher-categorical approach. While many aspects of the $1$-categorical proof of~\cite{AdamekRosicky94} carry over, key steps require new ideas. The central technical challenge lies in \Cref{proposition: prop4}, where we prove that subcategories closed under limits and filtered colimits are also closed under a class of morphisms called \emph{pure morphisms} (see \cref{definition: def1}). We note that, unlike other presentability criteria such as~\cite[Lemma 5.5.4.14]{lurie2009higher}, which assume a ``small generation'' condition on the subcategory, our result derives the necessary smallness purely from closure properties.

The $1$-categorical version of \cref{theorem: thm1} has a notable history. A special case of this result appeared already in~\cite{makkai1987some}, where the main theorem corresponds precisely to \cref{theorem: thm1} in the setting of compactly generated ordinary categories, under the additional assumption that the subcategory $\cD$ is closed under filtered colimits. Later, Adámek and Rosický generalized and strengthened this result in~\cite{AdamekRosicky89}, establishing the reflection theorem in full generality for ordinary presentable categories. However, their proof does not straightforwardly extend to the $\infty$-categorical context.

There is also a close connection between reflection results and a large cardinal axiom known as the \emph{Vopěnka principle}. This principle, which  implies the existence of measurable cardinals, moreover ensures, by work of Rosický and Tholen~\cite{rosicky2003left} in the $1$-categorical setting and by~\cite{https://doi.org/10.48550/arxiv.2105.04251} in the $\infty$-categorical setting, that every full subcategory of a presentable category that is closed under limits is also closed under sufficiently filtered colimits\footnote{We have been informed by members of the community that the proof in~\cite{https://doi.org/10.48550/arxiv.2105.04251} may contain a gap. As we do not rely on this result, we have not verified the details of the argument. However, we believe the conclusion should nonetheless hold.}. In fact, these papers show that the preceding statement is equivalent to the Vopěnka principle.

Combining this with \cref{theorem: thm1} shows that, under the Vop\v{e}nka principle, every limit-closed full subcategory of a presentable category is reflective. In particular, this means that \cref{theorem: thm1} is, in a precise sense, optimal: any stronger statement would require additional large-cardinal assumptions.

\subsection{Applications of Our Theorem in the Literature}

Since the appearance of the initial preprint, several works have relied on \cref{theorem: thm1} to establish presentability and the existence of sheafification functors in contexts ranging from algebraic geometry to abstract homotopy theory. We now briefly survey some of these applications, which illustrate the utility and reach of the theorem in practice.

\paragraph{Segalification.} 
Algebraic objects in an $\infty$-category $\cC$ are often defined as functors into $\cC$ satisfying Segal-type conditions. More precisely, one considers an indexing category $I$ and defines a full subcategory $\cD \subset \Fun(I, \cC)$ consisting of functors that satisfy suitable conditions encoding the desired algebraic structure. Many familiar examples fit this pattern, including monoids, commutative monoids, category objects, and higher commutative monoids.

In such constructions, the indexing category $I$ typically has a distinguished object, often denoted $\mathrm{pt}$, such that evaluation at $\mathrm{pt}$ corresponds to ``forgetting the algebraic structure.'' The associated free object construction then factors as a composite of left adjoints:
\[
\cC \xrightarrow{\mathrm{pt}_!} \Fun(I, \cC) \xrightarrow{L} \cD,
\]
where the first functor is the left adjoint to evaluation at $\mathrm{pt}$, and the second is the left adjoint to the inclusion. In other words, the subcategory $\cD$ must be reflective.

This framework appears, for instance, in the theory of higher semiadditivity (also known as ambidexterity), where it underlies the definition of $m$-commutative monoids. The existence of the corresponding localization is established using \cref{theorem: thm1}; see~\cite[Proposition 4.6]{ben2024higher}.\footnote{A similar result already appears in~\cite[Lemma 5.17]{harpaz2020ambidexterity}, where the author implicitly assumes the conclusion of our theorem.}

\paragraph{Complete modules.} In many situations, one defines a highly structured algebraic object $R$, and the definition automatically gives rise to a notion of $R$-modules. However, the resulting category is often too broad: not all such modules are ``complete'' in the sense appropriate to the context. One is therefore led to define, by hand, a full subcategory of \emph{complete} $R$-modules.

In such settings, many constructions are automatically functorial with respect to the uncompleted category of $R$-modules. The existence of a suitable completion operation then ensures functoriality with respect to the category of complete modules. This completion is realized as a left adjoint to the inclusion of the subcategory of complete modules into the ambient category of $R$-modules.

A prominent example arises in the theory of analytic stacks, where one needs to isolate a full subcategory of completed modules. The existence of the corresponding completion functor in this context is established using \cref{theorem: thm1}; see~\cite[Minute 0:50–8:40]{youtube}\footnote{We cite the YouTube lecture as written notes have not yet been released.}.

\paragraph{Further examples (briefly).} In~\cite{gleason2023meromorphic}, meromorphic vector bundles on the Fargues--Fontaine curve are studied via sheaves valued in rigid categories. The existence of a sheafification functor in this setting is established using \cref{theorem: thm1}; see~\cite[Theorem A.10]{gleason2023meromorphic}.

The main theorem of~\cite{haine2024exodromy}, which asserts that their version of the exit-path category is well-behaved, relies on the fact that a full subcategory of an atomically generated $\infty$-category that is closed under limits and colimits is also atomically generated. To establish this the author use \cref{theorem: thm1}; see~\cite[Proposition 1.1.13]{haine2024exodromy}.

In~\cite{porta2025stoksderived}, Stokes-filtered local systems are modeled by functors satisfying certain conditions. The geometric input ensures that, in favorable cases, these functors form a subcategory closed under limits and colimits inside the full functor category; see~\cite[Theorem 1.7]{porta2025stoksderived}. The authors then apply \cref{theorem: thm1} to conclude that the category of Stokes-filtered local systems is presentable; see~\cite[Theorem 1.8]{porta2025stoksderived}.

\cref{theorem: thm1} has been used in additional contexts not detailed here. For example, see~\cite[Lemma 4.20]{burklund2022chromatic}, \cite[Definition A.1]{hamann2023torsion}, and~\cite[Proposition 4.18]{barthel2024chromatic}, among others.

\subsection{Applications: Recognizing Properties Classified by Idempotent Algebras}

A natural application of the reflection theorem is to the problem of classifying reflective subcategories in symmetric monoidal $\infty$-categories in terms of algebraic structure. Specifically, one is often interested in the following question:
\begin{quote}
\emph{When is a full subcategory of a symmetric monoidal $\infty$-category equivalent to the category of modules over a commutative algebra object $A$?}\footnote{Such that under the equivalence, the forgetful functor is identified with the inclusion.}
\end{quote}

This question arises frequently in categorical and algebraic contexts, where one seeks to determine whether a given property can be captured by a universal algebra object. When such an algebra exists, it is called an \emph{idempotent algebra}. A classical example is Lurie's construction of the symmetric monoidal structure on spectra, by recognizing the $\infty$-category of spectra as the category of modules over the idempotent algebra classifying stability~\cite[Proposition 4.8.2.18]{Lurie11}.

One can also define idempotent algebras intrinsically: Let $\cC$ be a symmetric monoidal $\infty$-category and let $A \in \cC$ be an object. A map $u \colon \mathbbm{1}_{\cC} \to A$ exhibits $A$ as an \emph{idempotent algebra} if the map
\[
A \simeq A \otimes \mathbbm{1} \xrightarrow{1 \otimes u} A \otimes A
\]
is an equivalence. By~\cite[4.8.2.9]{Lurie11}, in this case $A$ admits a unique commutative algebra structure for which $u$ is the unit. Like we claimed above a fundamental feature of idempotent algebras is that the forgetful functor $\Mod_A(\cC) \to \cC$ is fully faithful. In particular, being an $A$-module is a \emph{property} of objects in~$\cC$.

Given a property of objects in $\cC$, it is natural to ask whether it is classified by an idempotent algebra—that is, whether the full subcategory of objects satisfying the property is equivalent to $\Mod_A(\cC)$ for some idempotent $A$. In the final section of this paper, we apply Theorem~1.1 to answer this question in two important cases. The first is when $\cC \in \CAlg(\Pr^L)$.

\begin{theorem}\cref{theorem: smashing in C}
Let $\cC \in \CAlg(\Pr^L)$ and let $\iota \colon \cD \hookrightarrow \cC$ be the inclusion of a full subcategory. The following are equivalent:
\begin{enumerate}
    \item There exists an object $A \in \CAlg^{\mathrm{idem}}(\cC)$ such that $\cD \simeq \Mod_A(\cC)$, and under this equivalence, $\iota$ identifies with the forgetful functor $\Mod_A(\cC) \to \cC$;
    \item The inclusion $\iota$ admits a left adjoint $L$, and $L$ is a smashing localization;
    \item The subcategory $\cD$ is closed under limits and colimits in $\cC$, and for all $d \in \cD$ and $c \in \cC$, both $d \otimes c$ and $\Hom^{\cC}(c, d)$ lie in $\cD$.
\end{enumerate}
\end{theorem}

The second case is where $\cC = \Pr^L$. Idempotent algebras in this category have been studied under the name \emph{modes} in~\cite{carmeli2021ambidexterity}. We obtain a similar characterization to Theorem~1.2, though since $\Pr^L$ is not itself presentable, additional set-theoretic assumptions are required.

\begin{theorem}\cref{theorem: smashing in PrL}
Let $\cP \hookrightarrow \Pr^L$ be a full subcategory. The following are equivalent:
\begin{enumerate}
    \item $\cP$ is equivalent to the category of modules over an idempotent algebra in $\Pr^L$;
    \item 
    \begin{enumerate}
        \item $\cP$ is closed under colimits in $\Pr^L$;
        \item If $\cD \in \cP$, then $\cD^{\Delta^1} := \Fun(\Delta^1, \cD)$ also lies in $\cP$;
        \item There exists a regular cardinal $\kappa$ such that for all $\kappa \leq \pi$, and for all diagrams $p \colon I \to \cP_\pi$, we have: for all $\kappa \leq \mu \leq \pi$, the category $\Ind_\mu(\lim p(i)^\pi)$ lies in $\cP$.
    \end{enumerate}
\end{enumerate}
\end{theorem}

\subsection{Outline of the Proof of the Reflection Theorem}

Before beginning the outline, we note that by the adjoint functor theorem~\cite[3.2.5]{NguyenRaptisScharadn18}, to prove the reflection theorem it suffices to show that a full subcategory of a presentable $\infty$-category that is closed under limits and sufficiently filtered colimits is accessibly embedded. The proof is devoted to verifying this.

We now outline the proof.\\

\textbf{Special class of morphisms:} In section 2 we define a special class of morphisms that we call $\kappa$-pure morphisms (see \cref{definition: def1}), depending on a cardinal $\kappa$. We then show that every $\kappa$-pure morphism is a $\kappa$-filtered colimit of split morphisms with a fixed domain, where split morphisms are morphisms that have a left inverse (see \cref{proposition: prop1}). Therefore, we get that a $\kappa$-pure morphism $f:A \to B$ gives rise to a functor $F:I^{\lhd} \to \cC$ where $I$ is $\kappa$-filtered and all the maps $F(-\infty \to i)$ admit a retraction (here $-\infty$ is the cone point). We call such functors quasi-split cones on $I$.\\

\textbf{Cone with retracts:} Let $F:I^{\lhd} \to \cC$ be a quasi-split cone. By choosing a left inverse for each map $F(-\infty \to i)$, $F$ gives rise to a functor from the pushout (see \cref{construction: cons1})
\[
L_I:= I^{\lhd} \underset{\Ob(I) \times \Delta^1}\coprod (\Ob(I) \times \Ret) \to \cC
\]
where $\Ret$ is the universal category with a retraction (see \cref{definition: def4}). We call the $\infty$-category $L_I$ the cone with retracts on $I$. In section $3$ we show that if $I$ is equivalent to a $1$-category then $L_I$ is equivalent to a $1$-category. Furthermore, when $I$ is a poset we give an explicit description of $L_I$ (see \cref{corollary: cor2}). The idea is that gluing a copy of $\Ret$ along the morphism $-\infty \to i$ is equivalent to gluing a free arrow from $i$, i.e. taking the pushout
\[
\begin{tikzcd}
    {\Delta^0} & I \\
    {\Delta^1} & {C_I}
    \arrow["s"', from=1-1, to=2-1]
    \arrow[from=1-1, to=1-2]
    \arrow[from=2-1, to=2-2]
    \arrow[from=1-2, to=2-2]
\end{tikzcd}
\]
and then inverting the unique arrow from $-\infty$ to the ``new" object (i.e. the object in $C_I$ and not in $C$) in $C_I^{\lhd}$. Therefore, it suffices to show that an $\infty$-category obtained by gluing a free arrow to a $1$-category is equivalent to a $1$-category (see \cref{proposition: prop2}), and that the $\infty$-category obtained by inverting the arrow from $-\infty$ to the new object is equivalent to a $1$-category (see \cref{proposition: prop3}).\\

\textbf{Closedness under pure morphisms:}
Let $\cC$ be a presentable $\infty$-category and $\cD \subset \cC$ be a full subcategory. We say that $\cD$ is closed under $\kappa$-pure morphisms if for any $B \in \cD$ and a $\kappa$-pure morphism $f:A \to B$ the object $A$ is in $\cD$ as well. In section $4$ we show that a full subcategory closed under limits and $\kappa$-filtered colimits in a category of presheaves is closed under $\mu$-pure morphisms for $\mu$ large enough (see \cref{proposition: prop4}). This is the heart of the argument. The general approach is to show that if $f:A \to B$ is $\kappa$-pure then $A$ can be obtained from $B$ by limits and $\kappa$-filtered colimits. To build the relevant diagram, we use the explicit description of $L_I$ from \cref{corollary: cor2}. This is also the only place where the fact that $L_I$ is a 1-category is used, but it is used in an essential way, as it allows for a combinatorial description of the diagram. \\

\textbf{Factoring through pure morphism:}
Let $\cC$ be a category of presheaves. In section $5$ we show that for any large enough cardinal $\kappa$, there exists a cardinal $\gamma \gg \kappa$ such that for any morphism $f:A \to B$ in $\cC$ with $A$ $\gamma$-compact there exists a factorization:
\[
\begin{tikzcd}
    A & B \\
    {A'}
    \arrow[from=1-1, to=2-1]
    \arrow["f", from=1-1, to=1-2]
    \arrow["{f'}"', from=2-1, to=1-2]
\end{tikzcd}
\]
where $f'$ is $\kappa$-pure and $A'$ is $\gamma$-compact.\\

\textbf{Finishing the proof:}
Let $\cC$ be a presentable $\infty$-category and $\cD \subset \cC$ a full subcategory closed under limits and $\kappa$-filtered colimits. Since $\cD$ is closed under $\kappa$-filtered colimits it is closed under $\mu$-filtered colimits for all $\mu \geq \kappa$. Therefore, from the result of section $4$, by enlarging $\kappa$, we may assume that $\cD$ is closed under $\kappa$-pure morphisms. Fix $B \in \cD$. Let $\gamma \gg \kappa$ as above and denote $\Pure_\kappa(\cC^\gamma)_{/B}$ the subcategory of $\cC^\gamma_{/B}$ spanned by the pure morphisms to $B$. Using the results of section $5$, we get that the inclusion
\[
\Pure_\kappa(\cC^\gamma)_{/B} \hookrightarrow (\mathcal{C}^\gamma)_{/B}
\]
is cofinal. This means that every object of $\cD$ is a filtered colimit of objects in $\cC^\gamma \cap \cD$, thus $\cD$ is an accessible category. From our assumption on $\cD$ we get that the inclusion $i: \cD \to \cC$ is an accessible functor between accessible categories. The adjoint functor theorem \cite[3.2.5]{NguyenRaptisScharadn18} then says that $i$ admits an accessible left adjoint. Hence $\cD$ is presentable (see \cref{theorem: reflaction}).\\

One should note that the general structure of the proof is analogous to the structure of the proof of the $1$-categorical case as presented in \cite{AdamekRosicky94}. The main difference is the proof of \cref{proposition: prop4} which is much more straight-forward in the $1$-categorical case.

\subsection{Conventions}

We shall generally follow \cite{lurie2009higher} in notation and terminology regarding $\infty$-categories. Having said that, a category with no more adjectives will always mean for us an $\infty$-category, unlike \cite{lurie2009higher}. However, since the $1$-categories in this paper usually play the role of "indexing" categories, we will distinguish between $1$-categories and $\infty$-categories notationally, as indicated below.

\begin{notation}$\empty$
\begin{itemize}
\item $\infty$-categories will be denoted by calligraphic font e.g. $\cC, \ \cD, \ \cE$.
\item $1$-categories will be denoted by capital letters e.g. $C, \ D, \ E$.
\item For a $1$-category $C$, we denote the discrete category on the objects of $C$ by $C_\delta$ \footnote{One should note that the operation $C \mapsto C_\delta$ is not invariant under equivalence of categories. However, we will only use it when $C$ is a poset and then it is an invariant notion.}. We will use the same notion for the set of vertices of a simplicial set $K$.
\item For a category $\cC$  we let $\cC^\kappa$ be the full subcategory on $\kappa$-compact objects.
\item For a symmetric monoidal category $\cC$  we let $\CAlg(\cC)$ be the the category of commutative algebra in $\cC$.
\item For a symmetric monoidal category $\cC$ and $R\in \CAlg(\cC)$ we let $\Mod_R(\cC)$ be the category of $R$ modules in $\cC$.
\item $\mathrm{Set}$ will denote the category of small sets.
\item $\Cat_1$ will denote the category of small  $1$-categories.
\item $\Cat_\infty$ will denote the category of small $\infty$-categories.
\item $\mathcal{S}$ will denote the category of small spaces.
\item $\Delta$ will denote the category of finite non-empty sets and monotone increasing maps.
\item $\Delta^+$ will denote the category of finite non-empty sets and strictly monotone increasing maps. 
\item $\sset$ will denote the category of simplicial sets i.e. the category $\Fun(\Delta^{\op}, \mathrm{Set})$, and $\Delta^n$, $\Lambda^n_i$, $\partial \Delta^n \in \sset$ are defined as usual.
\item $\Pr^L$  will denote the category of presentable categories and left adjoints.
\item $\Pr^R$  will denote the category of presentable categories and accessible right adjoints.
\end{itemize}
\end{notation}

\subsection{Acknowledgments}

We would like to thank the entire Seminarak group, especially Shay Ben Moshe, Shaul Barkan, Segev Cohen, Shai Keidar, and Lior Yanovski, for useful discussions on various aspects of this project and for their valuable comments on an earlier draft of the paper. We also thank the anonymous referees for their valuable suggestions and comments, which greatly improved this work.

\section{From Pure Morphisms to Quasi-Split Cones}

In this section we define a class of morphisms, dependent on a cardinal $\kappa$, called $\kappa$-pure morphisms (see \cref{definition: def1}), and show that every $\kappa$-pure morphism is a $\kappa$-filtered colimit of split morphisms, (see  \cref{proposition: prop1}).

We will need the following lemma:

\begin{lemma} \label{lemma: lem1}
Let $\kappa$ be a regular cardinal, $\mathcal{C}$ be a $\kappa$-presentable category and $K$ be a $\kappa$-small simplicial set. Then there exists a regular cardinal $\pi \geq \kappa$ such that for every cardinal $\mu\geq \pi$, and any $f\in \Fun(K,\mathcal{C})$, $f$ is a $\mu$-compact object (in the category $\Fun(K,\mathcal{C})$) if and only if for each vertex $x\in K_0$, $f(x)$ is a $\mu$-compact object in $\mathcal{C}$.  
\end{lemma}

\begin{proof}
First, the assumptions on $K$ and $\mathcal{C}$ implies that $\Fun(K,\mathcal{C})$ is presentable. \\
Since colimits in a functor category are computed point-wise, the evaluation functor $\ev_x:\Fun(K,\mathcal{C}) \to \mathcal{C}$ is a colimit preserving functor between presentable categories for any $x\in K_0$. It follows that there exists a cardinal $\mu_x$ such that for every $\mu\geq \mu_x$, the functor, $\ev_x$ sends $\mu$-compact objects to $\mu$-compact objects. Hence, by taking the supremum $\pi:=\sup_{x\in K_0}\mu_x$ we have that if $f$ is $\mu$-compact for $\mu\geq \pi$ then $f(x)\in \Ob(\mathcal{C})$ is also $\mu$-compact.\\
For the other implication see \cite[5.3.4.13]{lurie2009higher}.
\end{proof}

\begin{definition}\label{definition: def1}
Let $\mathcal{C}$ be a presentable category and let $A,B \in \Ob(\mathcal{C})$. We will say that a morphism $f:A \to B $ in $\cC$ is $\kappa$-pure if for any two $\kappa$-compact objects $A',B'\in \Ob(\cC)$ and a commutative diagram:
\[
\begin{tikzcd}
A' \arrow[r, "f'"] \arrow[d, "u"'] & B' \arrow[d, "v"] \\
A \arrow[r, "f"']                  & B                
\end{tikzcd}
\]
in $\mathcal{C}$, there exists a $2$ simplex filling the following diagram:
\[
\begin{tikzcd}
A' \arrow[d, "u"'] \arrow[r, "f'"] & B' \arrow[ld, "\bar{u}"] \\
A                                  &                         
\end{tikzcd}
\]
 for some $\bar{u}:B'\to A$. In other words, there exists a factorization $u\simeq \bar{u}\circ f'$.
\end{definition}
 
 \begin{definition} \label{definition: def2}
Let $\mathcal{C}$ be a category, following \cite[4.4.5.2]{lurie2009higher} we will say that a commutative diagram of the form:
\[
\begin{tikzcd}
                                  & X \arrow[d, "r"] \\
Y \arrow[r, "\Id"] \arrow[ru, "i"] & Y               
\end{tikzcd}
\]
is a weak retraction.\\
For any such weak retraction we will call $i$ a split morphism and $r$ a retraction.  
 \end{definition}
 
\begin{example}
Note that a split morphism is $\kappa$-pure for all $\kappa$.
\end{example}
 
 \begin{proposition} \label{proposition: prop1}
Let $\mathcal{C}$ be a $\kappa$-presentable category. We claim that there exists a regular cardinal $\pi \geq \kappa$, such that for all $\mu \geq \pi$, every $\mu$-pure morphism $f \in \Fun(\Delta^1,\mathcal{C})$ is a $\mu$-filtered colimit of split morphisms with common domain (as in the Outline in the introduction).
 \end{proposition}
 
\begin{proof}
By \cref{lemma: lem1} there exists $\pi \geq \kappa$ such that for all $\mu \geq \pi$,  $\Fun(\Delta^1,\mathcal{C})$ is $\mu$-presentable and $A \xrightarrow{g} B \in \Fun(\Delta^1,\mathcal{C})$ is $\mu$-compact if and only if $A$ and $B$ are $\mu$-compact in $\mathcal{C}$. Let $f:A \to B$ be a $\mu$-pure morphism, since  $\Fun(\Delta^1,\mathcal{C})$ is $\mu$-presentable, by definition there exists a functor
\[
F:J \to \Fun(\Delta^1,\mathcal{C})^\mu \simeq \Fun(\Delta^1,\mathcal{C}^\mu)
\]
where $J$ is $\mu$-filtered and $\colim F= f$. For each $i \in \Ob(J)$ we denote  $f_i:=F(i):A_i \to B_i$ and by $u_i$ the canonical map
\[
u_i:f_i \to \underset{J}{\colim} \ f_j = f.
\]                             
Let $\overline{B}_i $ be the pushout:
\[
\begin{tikzcd}
A_i \arrow[d, "\ev_0(u_i)"'] \arrow[r, "f_i"] & B_i \arrow[d] \\
A \arrow[r, "f_i'"']              & \overline{B_i}                 
\end{tikzcd}
\]
By the functionality of pushouts we get a $\mu$-filtered diagram $F':J \to \Fun(\Delta^1,\cC)$ with $F'(i)=f_i'$. We will now show that $f_i'$ are split and that their colimit is $f$ i.e. that $\colim\ F' =f$.\\

We begin by showing that $\colim F' =f$: Let $\mathcal{D}:=\Fun(\Delta^1,\mathcal{C}) \times_{\mathcal{C}} \mathcal{C}_{{/}A}$ be the category of morphisms over $A$. Informally $\mathcal{D}$ is spanned by the objects :
\[
\begin{tikzcd}
X \arrow[d, "g"'] \arrow[r, "f'"] & Y \\
A                                 &  
\end{tikzcd}
\]
with morphisms:
\[
\begin{tikzcd}
X' \arrow[r] \arrow[dd, bend right] & Y'          \\
X \arrow[d] \arrow[r] \arrow[u]     & Y \arrow[u] \\
A                                   &            
\end{tikzcd}
\]
In other words it is the category of arrows in $\cC$ together with a morphism from the source to $A$.\\
Since $\cD$ is a pullback, the functors $F:J \to \Fun(\Delta^1,\mathcal{C}) $ and $\ev_0(u_\bullet):J \to \cC_{/A}$ induce a functor $F'':J \to \mathcal{D}$, adding the fact that the forgetful $\mathcal{D} \to \Fun(\Delta^1,\mathcal{C})$ commutes with colimits we have that $\colim F''=f$ (here, we abuse notation by identifying $f$ and the object $``(f,\Id)"\in \mathcal{D}$). As the pushout functor $\mathcal{D}\to\mathcal{C}_{A/}$ also commutes with colimits, we get that $\colim \ F' =f$.\\
\\
We now turn to show that $f_i'$ are split: By construction we have a commutative diagram:
\[
\begin{tikzcd}
A_i \arrow[d, "\ev_0(u_i)"'] \arrow[r, "f_i"] & B_i \arrow[d,"\ev_1(u_i)"] \\
A \arrow[r, "f"]              & B                 
\end{tikzcd}
\]
where $A_i$ and $B_i$ are $\mu$-compact. Therefore, since $f$ is $\mu$-pure, we have a $2$-simplex of the form
\[
\begin{tikzcd}
A_i \arrow[r, "f_i"] \arrow[d, "\ev_0(u_i)"'] & B_i \arrow[ld, "g_i"] \\
A                                            &                      
\end{tikzcd}
\]
Thus, by the universal property of pushouts we get
\[
\begin{tikzcd}
    {A_i} & {B_i} \\
    A & {\bar{B}_i} \\
    && A
    \arrow["{\ev_0(u_i)}"', from=1-1, to=2-1]
    \arrow["{f_i'}", from=2-1, to=2-2]
    \arrow[from=1-2, to=2-2]
    \arrow[from=1-1, to=1-2]
    \arrow["{g_i}", from=1-2, to=3-3]
    \arrow["\Id"', from=2-1, to=3-3]
    \arrow[from=2-2, to=3-3]
\end{tikzcd}
\]
and the claim follows.
\end{proof}

\section{Quasi-Split Cones}

We have just seen that a pure morphism gives rise to a cone on a filtered category where all maps from the initial object have retracts. We will call such cones quasi-split (see \cref{definition: def3}). \\
This section is devoted to the study of quasi-split cones and the categories derived from them.

\begin{definition} \label{definition: def3}
Let $K \in \sset$ and $F:K^{\lhd} \to \mathcal{C}$ a cone on $K$. We say that $F$ is a quasi-split cone if
\[
F(-\infty \to k) \ \textrm{is split} \ \forall k \in K.
\]   
\end{definition}

\begin{definition} \cite[4.4.5.2]{lurie2009higher}\label{definition: def4}
Let $\Ret$ be the $1$-category defined as follows
\begin{itemize}
    \item The set of objects is given by $\{X,Y\}$.
    \item The sets of morphisms are given by
\[
\Map(X,X) = \{\Id, e \}, \quad
\Map(X,Y) = \{r\}, \quad \Map(Y,X) = \{i\}, \quad  \Map(Y, Y) = \{\Id\}.
\]
\end{itemize}
The composition law is determined by $r \circ i = \Id$ and $i \circ r = e$.
\end{definition}

\begin{definition}
In the category $\Ret$ there exists a commutative diagram
\[
\begin{tikzcd}
    & X \\
    Y && Y
    \arrow["\Id"', from=2-1, to=2-3]
    \arrow["i", from=2-1, to=1-2]
    \arrow["r", from=1-2, to=2-3]
\end{tikzcd}
\]
Which is determined by a map of simplicial sets $\sigma : \Delta^2 \to \Ret.$ We let $\wRet$ denote the image of $\sigma$ in $\Ret$.
\end{definition}

\begin{construction} \label{construction: cons1}
Given a quasi-split cone $F':K^{\lhd} \to \mathcal{C}$, by choosing a retract $g_k:F'(k)\to F'(-\infty) $ for each $h_k:=F'(-\infty \to k)$ we get a functor from the pushout in $\sset$
 \[
 F:K^{\lhd} \underset{K_\delta \times \Delta^1}\coprod ( K_\delta \times \Ret)= \colon  \widetilde{L}_K  \to \mathcal{C}.
 \]
We denote by $L_K$ a fibrant replacement of $\widetilde{L}_K$ in the Joyal model structure i.e. a $L_K$ is a quasi-category equivalent to $\widetilde{L}_K$. We shall call $L_K$ \textbf{the cone with retracts on} $K$.
\end{construction}

The main goal of this section is to prove the following:

\begin{proposition} \label{proposition: goal L_I}
If $\cC$ is equivalent to the nerve of a $1$-category $N(C)$ then, $L_C$, the cone with retracts on $C$ is equivalent to a nerve of a $1$-category. Furthermore, if $C=I$ is a poset then $L_I$ admits an explicit description as in \cref{corollary: cor2}.
\end{proposition}

The outline of the proof of \cref{proposition: goal L_I} is as follows:

\begin{itemize}
\item In subsection $3.1$ we will show that the category obtained by gluing a ``free arrow" to a $1$-category is still a $1$-category. \\
More precisely we show that the pushout in $\Cat_\infty$
\[
C \underset{\Delta^0}\coprod \Delta^1
\]
is equivalent to a $1$-category - \cref{proposition: prop2}.
\item Denote by $D_C$ the category obtained by iterating the previous construction for each object in $C$. \\
In subsection $3.2$ we will show that inverting the arrows from the initial object in $D_C^{\lhd}$ to the objects ``which are not in $C$", gives rise to a $1$-category.\\
More precisely the category $D_C$ is defined by the pushout:
\[
D_C:= C \underset{C_\delta \times \Delta^0}\coprod (C_\delta \times \Delta^1)
\]
and we will show that inverting all morphisms in $D_C^{\lhd}$ of the form $-\infty\to x$ where $x$ is not an object of $C$ (i.e. not in the image of the natural inclusion $C\to D_C^{\lhd}$) is still a $1$-category - \cref{proposition: prop3}.
\item In subsection $3.3$ we will show that the previous construction gives the same result as gluing $\Ret$ along the morphism from the initial object to the source of the free arrow - \cref{lemma: lem6}.
\end{itemize} 

\subsection{Gluing a Free Arrow}

This subsection is devoted to showing that the pushout
\[
\begin{tikzcd}
\Delta^0 \arrow[d, "s"'] \arrow[r] & C \arrow[d] \\
\Delta^1 \arrow[r]                 & P        
\end{tikzcd}
\]
in $\Cat_\infty$ is a $1$-category, when $C$ is a $1$-category.\\

Our first step toward proving the above claim will be to prove the following (rather technical) lemmas:

\begin{lemma}\label{lemma : no same degn}
Let $X \in \sset$ and let $\sigma_1 \in X_{k_1}$ and $\sigma_2 \in X_{k_2}$ be two non-degenerate simplices. Assume that for some $n\geq 0$ there are two surjective order preserving maps
\[
\phi_l \colon [n] \to  [k_l], \quad l=1,2
\]
such that $\phi_1^*(\sigma_1) = \phi_2^*(\sigma_2)$, then $k_1 = k_2$, $\sigma_1 = \sigma_2$ and $\phi_1=\phi_2$.
\end{lemma}

\begin{proof}
Without loss of generality assume that $k_1 \leq k_2$. Let
\[
f: \Delta^{k_1} \to X
\]
be the map classified by $\sigma_1$. Denote the unique non degenerate simplex in $\Delta^{k_1}_{k_1}$ by $\tau$. Let $\psi\colon [k_2] \to [n]$ be a right inverse to $\phi_2$ we have that
\[
* \quad \sigma_2=\psi^*(\phi_2^*(\sigma_2)) =\psi^*(\phi_1^*(\sigma_1))=\psi^*(\phi_1^*(f(\tau)))=f(\psi^*(\phi_1^*(\tau)).
\]
In particular the image of $f$ contains $\sigma_2$. Now, since $\Delta^{k_1}$ does not contain any non-degenerate simplices of dimension larger than $k_1$ we get that $k_1  = k_2$.
The other two implications are immediate, as the image of $f$ can contain at most one non-degenerate simplex of dimension $k_1$, and it contains both $\sigma_1$ and $\sigma_2$ so $\sigma_1 = \sigma_2$.\\

We now show that $\phi_1=\phi_2$. By the above $\sigma_1=\sigma_2$, and thus by the chain of equivalences $*$ we have that $\phi_1\circ \psi =\Id$. We conclude that every right inverse of $\phi_2$ is a right inverse of $\phi_1$ which implies that  $\phi_1=\phi_2$.
\end{proof}

\begin{notation}
Let $f:X \to Y$ be a map of simplicial sets, we denote by
\[
f^\mathrm{nd}_k:X_k^\mathrm{nd} \to Y_k
\]
the induced map we get by restricting $f_k$ to the non degenerate simplices in $X_k$.
\end{notation}

\begin{lemma} \label{lemma : seset injective map}
Let $A$ be a set and $0\leq i \leq n$ be natural numbers. Assume we have a commutative diagram in $\sset$ of the form:
\[
\begin{tikzcd}
    {A\times \Lambda^n_i} & X \\
    {A \times \Delta^n} & Y
    \arrow[from=1-1, to=1-2]
    \arrow[from=1-1, to=2-1]
    \arrow[from=2-1, "j" to=2-2]
    \arrow[hook',"f", from=1-2, to=2-2]
\end{tikzcd}
\]
and let \[
 g:A \times \Delta^{n-1} \xrightarrow{\Id\times\delta_i} A \times \Delta^n \to Y
 \]
where $\delta_i$ embeds $\Delta^{n-1}$ as the face in front of the $i$-th vertex.\\
Assume further that:
\begin{enumerate}
    \item The map $f:X \hookrightarrow Y$ is a level-wise injection.
    \item $g$ induces an injection    
\[
g_{n-1}^\mathrm{nd}:(A \times \Delta^{n-1})_{n-1}^\mathrm{nd} \to  Y_{n-1}.
\]
\item $\mathrm{Im}(g_{n-1}^\mathrm{nd}) \cap \mathrm{Im}(f_{n-1})=\mathrm{Im}(j_{n}^\mathrm{nd}) \cap \mathrm{Im}(f_{n})=\emptyset$.
\item Both the maps $j_n$ and $g_{n-1}$ send non degenerate simplices to non degenerate simplices.  
\end{enumerate}
Then the induced natural map from the pushout in $\sset$
\[
h:P:=(A \times \Delta^n) \underset{A \times \Lambda^n_i}\coprod X \to Y
\]
is a level-wise injection.
\end{lemma}

\begin{proof}
Let $\sigma_1, \sigma_2 \in P_k$, such that $h_k(\sigma_1)=h_k(\sigma_2)$. We shall show that $\sigma_1=\sigma_2$ by analyzing each possible case individually. \begin{itemize}
    \item Assume that $\sigma_1$ and $\sigma_2$ are in the image of the natural map $X\to P$.\\
    In this case the claim follows from (1).
    \item Assume that $\sigma_1$ and $\sigma_2$ are in the image of the natural map $A\times\Delta^n \to P$ and not in the image of the map $X\to P$.\\
    Let $\sigma_l' \in (A \times \Delta^n)_{k_l}$ for $l = 1,2$ be two non degenerate simplices that degenerate to the $\sigma_l$-s, i.e. there exists two order preserving and surjective maps $\phi_l:[k] \to [k_l]$ such that $\phi_l^*(\sigma_l')= \sigma_l$. If one of the $\sigma_l'$ is in the image of the natural map $A\times \Lambda_i^n\to P$ then it is also in the image of the natural map $X\to P$ which gives us that $\sigma_l$ is in the image of the natural map $X\to P$, contradicting our assumption.\\
    By the above, we conclude that both of the $\sigma_l'$-s are not in the image of the natural map $A\times \Lambda_i^n\to P$, so either $k_l = n-1$ and $h_{n-1}(\sigma_l')=g_{n-1}(\sigma_l')$ or $k_l = n$ and $h_n(\sigma_l')=j_n(\sigma_l')$. Since both $j_n$ and $g_{n-1}$ send non degenerate simplices to non degenerate simplices, we get by \cref{lemma : no same degn} that $h_{k_1}(\sigma_1')=h_{k_2}(\sigma_2')$ and $\phi_1=\phi_2$. All in all, since $g_{n-1}^\mathrm{nd}$ and $j_n^\mathrm{nd}$ are one-to-one, we also have that $\sigma_1'=\sigma_2'$ and thus $\sigma_1=\sigma_2$.
    
    \item Assume that $\sigma_1$ is in the image of the natural map $X\to P$ and that $\sigma_2$ is in the image of the natural map $A \times \Delta^n\to P$. \\
    Let $\sigma_2' \in (A \times \Delta^n)_{m}$ be a non degenerate simplex that degenerates to $\sigma_2$ i.e. $\phi^*(\sigma_2')=\sigma_2$. \\
    If $\sigma_2'$ is in the image of the natural map $A\times \Lambda_i^n\to P$, we are back to the first case. \\
    If $\sigma_2'$ is not in the image of the natural map $A\times \Lambda_i^n\to P$, then either $m = n$ and $h_n(\sigma_2')=j_n(\sigma_2')$ or $m=n-1$ and $h_{n-1}(\sigma_2')=g_{n-1}(\sigma_2')$. Assume that $m=n$ and choose a right inverse $\psi$ for $\phi$. Since maps of simplicial sets are neutral transformations we get 
    \[
    f_k(\sigma_1)=\phi^*(j_n(\sigma_2')) \implies f_n(\psi^*(\sigma_1))=j_n(\sigma_2').
    \]
    Noting that $j_n$ sends non degenerate simplices to non degenerate simplices, we get a contradiction to our assumption that $\mathrm{Im}(j_{n}^\mathrm{nd}) \cap \mathrm{Im}(f_{n})=\emptyset$. Similar argument shows that if $m=n-1$, then we get a contradiction to our assumption that $\mathrm{Im}(g_{n-1}^\mathrm{nd}) \cap \mathrm{Im}(f_{n-1})=\emptyset$.
\end{itemize}
\end{proof}

Recall the main objective of this subsection: We want to show that for any $1$-category $C$ the pushout
\[
\begin{tikzcd}
\Delta^0 \arrow[d, "s"'] \arrow[r] & C \arrow[d] \\
\Delta^1 \arrow[r]                 & P        
\end{tikzcd}
\]
in $\Cat_\infty$ is a $1$-category. We shall do so by explicitly constructing its fibrant replacement in the Joyal model structure on $\sset$.\\

Let us fix for the remainder of this section a quasi-category $\cC \in \sset$ and an object $x \in \cC_0$. For each $n\geq 0$ denote by $\cC_n^x \subset \cC_n$ the subset of simplices with "last vertex $x$". We consider $\cC_n^x$ as a pointed set pointed by the degenerate simplex on $x$.\\

We now define our candidate for the fibrant replacement of $P$.

\begin{definition}
Let $\cD^{\infty} \in \sset$ be the following simplicial set:
\begin{itemize}
    
    \item the simplices are given by:
    \[
    \cD^{\infty}_n  = \cC_{n} \sqcup \coprod_{i=0}^{n} \cC_{i}^{x}.
    \]
   
    \item The face maps are defined as follows:\\
    If $\sigma \in \cC_{n+1} \subset \cD^{\infty}_{n+1}$, then
    \[
    \partial^k_{\cD}(\sigma)=\partial^k_{\cC}(\sigma) \in \cC_{n} \subset \cD^{\infty}_{n} .
    \]
    If $\sigma \in \cC_{n+1}^x \subset \cD^{\infty}_{n+1}$, then
    \[
    \partial^k_{\cD}(\sigma)=\begin{cases} \partial^k_{\cC}(\sigma) \in \cC_{n}^x  \subset \cD^{\infty}_{n}, \quad k \leq {n} \\
    \partial^k_{\cC}(\sigma) \in \cC_{n} \subset \cD^{\infty}_{n},  \quad k = {n+1}
    \end{cases}.
    \]
    If $\sigma \in \cC_{l}^x \subset \cD^{\infty}_{n+1}$ and $l \not=n+1$, then
    \[
    \partial^k_{\cD}(\sigma)=\begin{cases} \partial^k_{\cC}(\sigma) \in \cC_{l-1}^x  \subset \cD^{\infty}_{n}, \quad k \leq {l-1} \\
    \sigma \in \cC_{l}^x \subset \cD^{\infty}_{n},  \quad\quad\quad\ \  k \geq l
    \end{cases}.
    \]
    
    \item The degeneracies are defined as follows:\\
    If $\sigma \in \cC_{n} \subset \cD^{\infty}_{n}$, then
    \[
    s^k_{\cD}(\sigma)=s^k_{\cC}(\sigma) \in \cC_{n+1} \subset \cD^{\infty}_{n+1} .
    \]
    If $\sigma \in \cC_{l}^x \subset \cD^{\infty}_{n+1}$, then
    \[
    s^k_{\cD}(\sigma)= \begin{cases} s^k_{\cC}(\sigma) \in \cC_{l+1}^x  \subset \cD^{\infty}_{n+1}, \quad k \leq {l-1} \\
    \sigma \in \cC_{l}^x \subset \cD^{\infty}_{n+1},  \quad\quad\quad\ \  k \geq l
    \end{cases}.
    \]
    One can check directly that this defines a simplicial set.
\end{itemize}
\end{definition}

We proceed to show that $\cD^{\infty}$ is a fibrant replacement for the pushout  
\[
\begin{tikzcd}
\Delta^0 \arrow[d, "s"'] \arrow[r,"x"] & \mathcal{C} \arrow[d] \\
\Delta^1 \arrow[r]                 & P        
\end{tikzcd}
\]
in $\sset$, when $\cC$ is the nerve of a $1$-category.

\begin{remark}
We believe that whenever $\cC$ is a quasi-category, $\cD^{\infty}$ is the fibrant replacement of the pushout. Since we are only interested in the case where $\cC$ is the nerve of the $1$-category, we shall not prove this more general case.
\end{remark}

\begin{lemma}\label{lemma: lem3}
Assume that $\cC$ is the nerve of a $1$-category $C$ and $x\in\Ob(C)$. Let $D$ be the $1$-category described below:
\begin{itemize}
    \item  The objects are given by:\\
    \[
    \Ob(D) = \Ob(C) \sqcup \{x'\}.
    \]
   
    \item The morphisms sets are defined as follows:\\
    for all $a,b \in \Ob(D)$
    \[
    \hom_{D}(a,b) = \begin{cases}
    \hom_{C}(a,b),\quad a,b \in \Ob(C)\\
    \hom_{C}(a,x),\quad a \in \Ob(C),\, b = x' \\
    \{*\},\quad\quad\quad\quad \  a = b = x' \\
    \emptyset,\quad\quad \quad\quad\quad \  a = x',\, b  \in \Ob(C) \\
    \end{cases}.
    \]
\end{itemize}
Then $\cD^{\infty} \simeq N(D) $.
\end{lemma}

\begin{proof}
Note that we have a canonical bijection between $\hom_{D}(a,x')$ and $\hom_{D}(a,x)$ for $a\not=x'$, which we will denote by $f \mapsto \bar{f}$. We adopt the following notation - given a simplex, $\sigma \in N(D)_n$, denote by $i_\sigma$ the number of times that $x'$ appears as a vertex of $\sigma$. Note that $x'$ can only appear as the last $i_\sigma$ vertices of $\sigma$ and that the ``maps" from $x'$ to itself can only be identities. Furthermore, we note that if $i_\sigma=1$ then, by the above bijection, there exists a canonical simplex $\sigma_x\in N(D)_n$ in which we change the last vertex from $x'$, to $x$, i.e.
\[
\textrm{if} \ \ \sigma=[y_0\overset{f_0}\to... \overset{f_n}\to y_n= x']  \ \  \textrm{then} \  \ \sigma_x=[y_0\overset{f_0}\to... \overset{\bar{f_n}}\to x].
\]
We proceed to define an isomorphism $f:N(D)\to \cD^\infty$ as follows:
\[
f_n: N(D)_n \to \cD^{\infty}_n, \quad f(\sigma)=\begin{cases} \sigma, \quad\quad\quad\quad\quad\quad\quad\quad\quad\quad\quad\quad \  i_\sigma =0 \\
(\sigma_{|\{0,...n-i_\sigma+1\}})_x \in \cC^x_{n-i_\sigma+1}, \quad i_{\sigma}\not=0.
 \end{cases}.
\]
The above functions are obviously bijective, and one can also check that the different $f_n$ assemble to a map $f:N(D) \to \cD^{\infty}$. We conclude that $\cD^{\infty} \simeq N(D)$.
\end{proof}

We will now define a sequence of simplicial sets ``converging" to $\cD^\infty$ i.e.
\[
\underset{n}\colim (\cD^0\subset\cD^1\subset \cdots)=\cD^\infty
\]
where the colimit is taken in $\sset$. Let $({\cC}^x_n)^{k}$ be the union of the images of the maps $\cC_{k} \to {\cC}_n$ intersected with $\cC_n^x$. For example, for $k\geq n$ we have that $({\cC}^x_n)^{k}={\cC}^x_n$ and for $k=n-1$ we have that $({\cC}^x_n)^{n-1}$ is the set of degenerate simplices in $\cC_n^x$ (where we think of ${\cC}^x_n \subset \cC_n$). We define a sub-simplicial set ${\cD}^{m} \subset \cD^{\infty}$:
\[
 {\cD}^{m}_n =\cC_{n} \sqcup \underset{k=0}{\overset{n}\coprod} (\cC_k^x)^{m}= \cC_{n} \sqcup \underset{k=m+1}{\overset{n}\coprod} (\cC_k^x)^{m} \sqcup \coprod_{i=0}^{\min\{m,n\}}  \cC_{i}^{x}. 
\]
Informally the difference between $\cD^m$ and $\cD^{m+1}$ is the $m+1$ composable arrows missing in $\cD^m$ (a formal version of this statement is \cref{lemma: D^m pushout}).\\

We also denote by $\cC_{k}^{x,x}\subset  \cC_{k}^{x}$ the subset of simplices whose last two coordinates are $x$ and the ``map" from $x$ to $x$ is the ``identity". Note that if $k \leq n-2$, then all the simplices in $(\cC_n^x)^{k}$ are degenerate, when viewed as simplices of $\cD^m$. And if $k =n-1$, then the non degenerate simplices in $(\cC_n^x)^{n-1}$ are contained in $\cC_{n}^{x,x}$.

\begin{lemma} \label{lemma: lem4}
There exists a pushout diagram in $\sset$ of the form
\[
\begin{tikzcd}
\Delta^0 \arrow[d, "s"'] \arrow[r, "x"] & \mathcal{C} \arrow[d] \\
\Delta^1 \arrow[r]                 & \mathcal{D}^0        
\end{tikzcd}
\]
where $x$ is the morphism that chooses $x$ and $s$ is the morphism that chooses $0$.
\end{lemma}

\begin{proof}
Define $\cC \to \cD^0$ as the obvious map
\[
\cC_n \hookrightarrow \cC_n \sqcup \coprod_{k=0}^n (\cC_k^x)^0=\cD^0_n
\]
and $\Delta^1 \to \cD^0$ as the map that chooses the unique simplex in $(\cC^x_1)^0$. These two maps together form a map from the pushout of the diagram in the lemma $\cC \underset{\Delta^0}\coprod \Delta^1 \to \cD^0$. Using the fact that $\cD^0_n=\cC_n \underset{k=0}{\overset{n}\coprod} *$, one can verify that this map is level-wise bijective.  
\end{proof}

Note that for each $m\geq 0$ we can define $m+1$ maps $r_k:\cC_m^x \to \cD^{m-1}_m$ for $0\leq k\leq m$:
\\For $k< m$ the maps are given by
\[
 r_k:\cC_m^x \overset{\partial^k_\cD}\to \cC_{m-1}^x \xrightarrow{s^{m-1}_\cC} \cC_m^{x,x}\subset \cD^{m-1}_m, \quad k \leq m
\]
and $r_m$ is the forgetful map $\cC_m^x \to \cC_m\subset \cD^{m-1}_m$ which forgets that $x$ is ``marked". We observe that given a $m$ simplex $\sigma \in  \cC_m^x$ we can get a $m+1$ simplex $\sigma'\in \cD^\infty_{m+1}$ via the map $\cC_{m}^x \xrightarrow{s^{m}_\cC} \cC_{m+1}^{x,x}\subset \cD^\infty_{m+1}$ and that the $r_k(\sigma)$-s are all the faces of $\sigma'$ different then the face in front of the $m$-th vertex. It follows that the above maps assemble into a map of simplicial sets:
\[
\Lambda^{m+1}_{m} \times \mathcal{C}_{m}^{x} \to \mathcal{D}^{m-1}.
\]
\begin{remark} \label{remark : composing with Id}
One can think about the map
\[
\cC_{m-1}^x \xrightarrow{s^{m-1}_\cC} \cC_m^{x,x}\subset \cD^{m-1}_m
\]
as taking elements of $\cC_{m-1}^x$ thought of as $m-1$ composable arrows in $\cC$ ending in $x$, and composing them with the identity morphism on $x$ e.g.
\[
\begin{tikzcd}
y \arrow[d] & z \arrow[l] \arrow[ld] \\
x           &                       
\end{tikzcd}
\mapsto
\begin{tikzcd}
y \arrow[d] \arrow[rd] & z \arrow[ld] \arrow[l] \arrow[d] \\
x \arrow[r]            & x                               
\end{tikzcd}.
\]
\end{remark}

\begin{lemma} \label{lemma: D^m pushout}
Let $(\mathcal{C}_{m}^{x})^{\mathrm{nd}} \subset \mathcal{C}_{m}^{x}$ be the subset of non-degenerate simplices. Here, when we say non-degenerate we think of the elements of $\mathcal{C}_{m}^{x}$ as simplices in $\mathcal{C}_{m}$. We claim that the map described above defines a pushout diagram in $\sset$
\[
\begin{tikzcd}
\Lambda^{m+1}_{m} \times (\mathcal{C}_{m}^{x})^{\mathrm{nd}} \arrow[r] \arrow[d] & \mathcal{D}^{m-1} \arrow[d] \\
\Delta^{m+1} \times (\mathcal{C}_{m}^{x})^{\mathrm{nd}} \arrow[r]                  & \mathcal{D}^{m}      
\end{tikzcd}
\]
In particular the inclusion $\cD^m \to \cD^{m+1}$ is a categorical equivalence.
\end{lemma}

\begin{proof}
Denote the pushout by $P$. Let $i: \cD^{m-1} \hookrightarrow  \cD^m$ be the natural inclusion. Define
\[
f:\Delta^{m+1} \times (\mathcal{C}_{m}^{x})^{\mathrm{nd}} \to \cD^m
\]
by sending $\sigma \in (\mathcal{C}_{m}^{x})^{\mathrm{nd}}$ to the simplex in $\cC^{x,x}_{m+1}$ we get by composing with $\Id:x \to x$ see \cref{remark : composing with Id} e.g.
\[
\begin{tikzcd}
y \arrow[d] & z \arrow[l] \arrow[ld] \\
x           &                       
\end{tikzcd}
\mapsto
\begin{tikzcd}
y \arrow[d] \arrow[rd] & z \arrow[ld] \arrow[l] \arrow[d] \\
x \arrow[r]            & x                               
\end{tikzcd}
\]
By the universal property of pushout we get a map $g:P \to \cD_m$. We will show that $g$ is an isomorphism by verifying that $g_n:P_n \to \cD_n^m$ is bijective.
\begin{itemize}
    \item \textbf{$g_n$ is surjective:}\\
    Note that since $g$ is a natural transformation it suffices to show surjectivity on the non-degenerate simplices of $\cD^m$. Let $\sigma \in \cD^m_n$ be non-degenerate. If
    \[
    \sigma \in    \cC_{n} \sqcup \underset{k=0}{\overset{n}\coprod} (\cC_k^x)^{m-1} = \cD^{m-1}_n
    \]
    then $\sigma$ is obviously in the image $i$. Therefore, we may assume $n \in \{m,m+1\}$. If $n=m$, then by our analysis of the non degenerate simplices in $(\cC_k^x)^{m}$ we may also assume that $\sigma$ is contained in one of the following two subsets of $\cC_m^x$ -
    \[
    \sigma \in (\mathcal{C}_{m}^{x})^{\mathrm{nd}} \quad \textrm{or} \quad \sigma \in \mathcal{C}_{m}^{x,x}.
    \]
    We already handled the case of $\sigma \in \mathcal{C}_{m}^{x,x}\subset \cD^{m-1}_m$, so we assume that $\sigma \in (\mathcal{C}_{m}^{x})^{\mathrm{nd}}$. By assumption we have we have that
    \[
    \partial^m f_{m+1}(\sigma) =\sigma  \in \cD^m_m.
    \]
    Finally, if $n=m+1$, then we may assume that $\sigma \in \cC_{m+1}^{x,x}$ and we have
    \[
    f_{m+1}(\partial^m(\sigma))= \sigma \in \cD^m_{m+1}.
    \]
    We conclude that $g_n$ is surjective.
    \item \textbf{$g_n$ is injective:} We shall check that the diagram in the lemma satisfies the conditions of \cref{lemma : seset injective map}. Indeed, by assumption $i$ is level-wise injective. Note that the $g_{m}$ from \cref{lemma : seset injective map} that corresponds to $f$ is represented by the obvious map 
    \[
    (\mathcal{C}_{m}^{x})^{\mathrm{nd}} \hookrightarrow \mathcal{C}_{m}^{x} \subset \cD^m_m.
    \]
    Hence, $g_m$ sends non degenerate simplices to non degenerate simplices and by the definition of $\cD^m$,   $\mathrm{Im}(g_{m}^\mathrm{nd}) \cap \mathrm{Im}(i_{m})=\emptyset$. From the definition of $\cD^m$ and $\cD^{m-1}$ one can see that $\mathrm{Im}(f_{m+1}^\mathrm{nd}) \cap \mathrm{Im}(i_{m+1})=\emptyset$ and that $f_{m+1}$ sends non degenerate simplices to non degenerate simplices. 
    \end{itemize}

\end{proof}

Based on all of the above we get:  

\begin{proposition}\label{proposition: prop2}
There is a pushout diagram in $\Cat_{\infty}$
\[
\begin{tikzcd}
\Delta^0 \arrow[d, "s"'] \arrow[r] & \mathcal{C} \arrow[d] \\
\Delta^1 \arrow[r]                 & \mathcal{D}^\infty
\end{tikzcd}
\]
where $\cC$ is the nerve of a $1$-category. Furthermore, $\mathcal{D}^\infty \simeq N(D)$  where $D$ is as in \cref{lemma: lem3}.\footnote{This proposition is true for a general quasi-category $\cC$ (i.e. not necessarily the nerve of a $1$-category), but as we will only need the lemma for the case of a 1-category the general proof is omitted from this paper.}
\end{proposition}

\begin{proof}
By \cref{lemma: lem4} $\cD^0$ is the pushout in $\sset$. Due to the fact that a filtered colimit of weak equivalences remains a weak equivalence (in the Joyal model structure), it follows from \cref{lemma: D^m pushout} that $\cD^{0} \to \cD^{\infty}$ is a categorical equivalence, proving the claim.
\end{proof}

\subsection{Inverting the Morphism}

Let $C$ be a $1$-category and let $D ^{\lhd}_C$ be the $1$-category one gets by gluing arrows to all the objects in $C$, see \cref{construction: cons2}, and then taking a cone. In this subsection we show that inverting the arrows from the initial object in $D_C^{\lhd}$ to the objects ``which are not in $C$", gives rise to a $1$-category.\\
As a first step, we will present a criterion for when an object is initial in terms of ``mapping to" property.   

\begin{lemma} \label{lemma : initial as mapping to prop}
Let $\cC \in \Cat_\infty$. Then, $x \in \cC$ is initial if and only if for every $\cT \in \Cat_\infty$ and $t \in \cT$, the forgetful map $\cT_{t/}\to \cT$ induces an equivalence
\[
\Fun^{x \to (\Id:t\to t)}(\cC,\cT_{t/}) \to \Fun^{x \to t}(\cC,\cT)
\]
where $\Fun^{x \to \bullet}(\cC,(-))$ is the pullback
\[\begin{tikzcd}
	{\Fun^{x \to \bullet}(\cC,(-))} & {\Fun(\cC,(-))} \\
	\mathrm{pt} & (-)
	\arrow[from=1-1, to=1-2]
	\arrow[from=1-1, to=2-1]
	\arrow["{\mathrm{ev}_x}", from=1-2, to=2-2]
	\arrow["\bullet"', from=2-1, to=2-2]
\end{tikzcd}\]
\end{lemma}

\begin{proof}
Let $\cT \in \Cat_\infty$ and $t \in \cT$. We denote by $p_\cT:\cT_{t/} \to \cT$ the forgetful functor and by $i: \cC \to \cC^{\lhd}$ the natural inclusion. We have a commutative diagram
\[
\begin{tikzcd}
 * \quad\quad   {\Fun^{x \to (\Id:t\to t)}(\cC,\cT_{t/})} && {\Fun^{(-\infty\to x) \to (\Id:t\to t)}(\cC^{\lhd},\cT)} \\
    \\
    && {\Fun^{x \to t}(\cC,\cT)}
    \arrow["\sim", from=1-1, to=1-3]
    \arrow["{ (-)\circ i}", from=1-3, to=3-3]
    \arrow["{p_\cT \circ(-)}"', from=1-1, to=3-3]
\end{tikzcd} 
\]
where the upper arrow in the diagram
\[
\Fun^{x \to (\Id:t\to t)}(\cC,\cT_{t/}) \to \Fun^{(-\infty\to x) \to (\Id:t\to t)}(\cC^{\lhd},\cT)
\]
is the equivalence arising from the universal property of the under category. \\

Assume that $x$ is initial. In this case, any $\bar{F} \in \Fun^{(-\infty\to x) \to (\Id:t\to t)}(\cC^{\lhd}, \cT)$ is the right Kan extension of its restriction to $\cC$. We conclude that composition with $i$ induces an equivalence:
\[
\Fun^{(-\infty\to x) \to (\Id:t\to t)}(\cC^{\lhd}, \cT) \to \Fun^{x \to t} (\cC, \cT).
\]
By the diagram $*$ we have that $p_\cT$ induces an equivalence:
\[
\Fun^{x \to (\Id:t\to t)} (\cC, \cT_{t/}) \to \Fun^{x \to t} (\cC, \cT).
\]
For the other implication: Assume that for all $\cT \in \Cat_\infty$ and $t \in \cT$ composition with $p_\mathcal{T}$ induces an equivalence
\[
\Fun^{x \to (\Id:t\to t)} (\cC, \cT_{t/}) \to \Fun^{x \to t} (\cC, \cT).
\]
Choosing $\cT=\cC$ and $t=x$ we have that $p_\cC$ has a right inverse $q$. Since $q \in \Fun^{x \to (x,\Id)} (\cC, \cC_{x/})$, by definition $q(x)=(x,\Id)$. Furthermore, for all $y \in \cC$ the composition
\[
\Map_\cC(x,y) \overset{q(-)}\to \Map((x,\Id),q(y)) \overset{p(-)}\to \Map(x,y)
\]
is an equivalence. We conclude that $\Map_\cC(x,y)$ is a retract of a contractible space and thus contractible.
\end{proof}

\begin{corollary} \label{corollary: initial in loc}
Let $C$ be a $1$-category and assume that it has an initial object $\emptyset$. Let $W$ be a set of morphisms. Then, the image of $\emptyset$ in $C[W^{-1}]$ is an initial object.
\end{corollary}

\begin{proof}
Let $\cT \in \Cat_\infty$ and $t \in \cT$. By the universal property of localization and \cref{lemma : initial as mapping to prop} we have the following chain of equivalences:
\[
\Fun^{\emptyset \to t} (C[W^{-1}],\cT) \simeq \Fun^{W \to \mathrm{Eq}, \ \emptyset \to t}(C, \cT) \simeq \Fun^{W \to \mathrm{Eq}, \ \emptyset \to (\Id:t\to t)}(C, \cT_{t/}) \simeq \Fun^{\emptyset \to (\Id:t\to t)} (C[W^{-1}],\cT_{t/}).
\]
Therefore, the claim follows from \cref{lemma : initial as mapping to prop}.  
\end{proof}

\begin{construction} \label{construction: cons2}
Let $C$ be a $1$-category. Choose a well-ordering on $C_\delta$ i.e. an isomorphism $C_\delta \simeq \mu$ for some ordinal $\mu$.
We shall define categories $D_j \in \Cat_1$ for $j \leq \mu$ by induction. \\
Let $D_0 =C$. \\
For $j$ successor, we define $D_j$ to be the following pushout in $\Cat_\infty$:
\[
\begin{tikzcd}
    {\Delta^0} & {D_{j-1}} \\
    {\Delta^1} & {D_j}
    \arrow["s"', from=1-1, to=2-1]
    \arrow["{x_{j-1}}", from=1-1, to=1-2]
    \arrow[from=1-2, to=2-2]
    \arrow[from=2-1, to=2-2]
\end{tikzcd}
\]
where $s$ chooses $0$ and $x_{j-1}$ chooses $x_{j-1} \in C$. \\
For $j$ a limit ordinal we let $D_j$ be the colimit in $\Cat_\infty$:
\[
D_j:=\underset{k < j}\colim \ D_k.
\]
According to the previous section, $D_C:=D_\mu$ is also a $1$-category. We say that $D_C$ is the category obtained from $C$ by gluing arrows to all the objects. We denote the objects of $D_C$ which are not in $C$ by $\{x'\}_{x \in C}$.
\end{construction}

Recall that our goal is to show that the category obtained by inverting the morphisms $-\infty\to x_j'$ in $D_C^\lhd$ is a $1$-category. We will do so by using the hammock localization construction. For the convenience of the reader we shall repeat the construction here:

\begin{construction}\cite[2.1]{DwyerKan79}  \label{construction: cons3}
Let $C$ be a $1$-category and $W \subset C$ be a wide-subcategory i.e. one which contains all the objects. The hammock localization of $C$ with respect to $W$ is a simplicial category $L^H(C,W)$ defined as follows: for every two objects $x,y \in C$ the $k$-simplices of $\mathrm{Hom}_{L^H(C,W)}(x,y)$ will be the ``reduced hammocks of depth $k$ and any length" between $x$ and $y$ i.e. commutative diagram of the form:
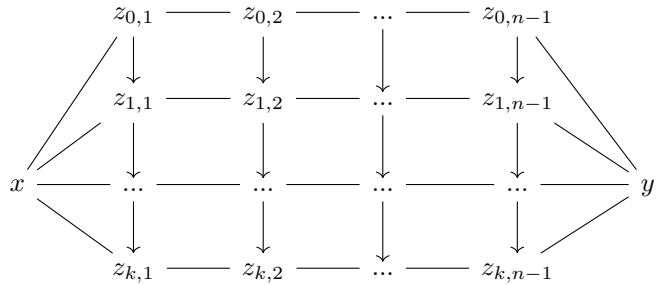
\[
\begin{tikzcd}
    & {z_{0,1}} & {z_{0,2}} & {...} & {z_{0,n-1}} \\
    & {z_{1,1}} & {z_{1,2}} & {...} & {z_{1,n-1}} \\
    x & {...} & {...} & {...} & {...} & y \\
    & {z_{k,1}} & {z_{k,2}} & {...} & {z_{k,n-1}}
    \arrow[from=1-2, to=2-2]
    \arrow[from=1-3, to=2-3]
    \arrow[from=1-5, to=2-5]
    \arrow[no head, from=1-2, to=1-3]
    \arrow[no head, from=2-2, to=2-3]
    \arrow[no head, from=2-3, to=2-4]
    \arrow[no head, from=2-4, to=2-5]
    \arrow[no head, from=1-3, to=1-4]
    \arrow[no head, from=1-4, to=1-5]
    \arrow[from=1-4, to=2-4]
    \arrow[no head, from=3-1, to=1-2]
    \arrow[no head, from=3-1, to=2-2]
    \arrow[no head, from=3-1, to=3-2]
    \arrow[from=2-2, to=3-2]
    \arrow[from=2-3, to=3-3]
    \arrow[from=2-4, to=3-4]
    \arrow[from=2-5, to=3-5]
    \arrow[no head, from=3-2, to=3-3]
    \arrow[no head, from=3-3, to=3-4]
    \arrow[no head, from=3-4, to=3-5]
    \arrow[from=3-2, to=4-2]
    \arrow[no head, from=4-2, to=4-3]
    \arrow[from=3-3, to=4-3]
    \arrow[no head, from=4-3, to=4-4]
    \arrow[no head, from=4-4, to=4-5]
    \arrow[from=3-5, to=4-5]
    \arrow[from=3-4, to=4-4]
    \arrow[no head, from=3-1, to=4-2]
    \arrow[no head, from=3-5, to=3-6]
    \arrow[no head, from=4-5, to=3-6]
    \arrow[no head, from=3-6, to=2-5]
    \arrow[no head, from=1-5, to=3-6]
\end{tikzcd}
\]
in which:
\begin{enumerate}
\item $n$ is an integer larger than $0$.
\item All the vertical maps are in $W$.
\item In each column, all maps point in the same direction; if they point to the left, then they are in $W$.
\item The maps in adjacent columns point in different directions.
\item No column contains only the identity map.
\end{enumerate}

Faces and degeneracies are defined by omitting or repeating rows. If the resulting hammock is not reduced, i.e. does not satisfy condition $4$ or $5$, then we make it reduced, by composing adjacent columns whenever their maps point in the same direction and omitting columns which contain only the identity map. This gives a model for the localization of $C$ with respect to $W$ by \cite[17 section 5]{Stevenson15}.
\end{construction}

\begin{proposition} \label{proposition: prop3}
Let $C$ be a $1$-category and denote by $D:=D_C$ the category obtained from $C$ by gluing arrows to all the objects. Let $W' \subset D^{\lhd}$ be the full subcategory on $\{- \infty\}\cup \{x'\}_{x\in C}$ where $-\infty$ is the cone point, and denote by $W = W' \cup D_\delta$ the wide-subcategory generated by $W'$. Then, the hammock localization $E:=D^{\lhd}[W^{-1}]=L^H(D^\lhd,W)$ is equivalent to a $1$-category.
\end{proposition}

\begin{proof}
We will show that $\Map_E(x,y)$ is equivalent to a discrete space for any $x,y \in E$. We will do so by case by case analyses.
\begin{itemize}
    \item Assume that $x=-\infty$.\\
    Since $-\infty$ is initial in $D^{\lhd}$, by \cref{corollary: initial in loc}, we have that $\Map_E(-\infty,y)$ is contractible and in particular equivalent to a discrete space for all $y \in E$.
    
    \item Assume that $x$ and $y$ are in the image of $C$ in $E$.\\
    We will show that under the above assumption, one can build an explicit isomorphism of simplicial sets between $\Map_E(x,y)$ and a discrete simplicial set.\\
    Assume that we have a reduced hammock of length bigger than $0$ 
    
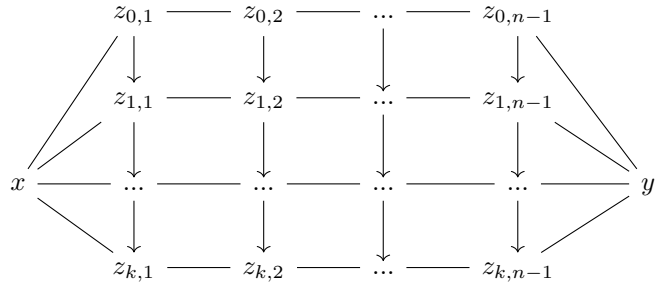
\[
    \begin{tikzcd}
    & {z_{0,1}} & {z_{0,2}} & {...} & {z_{0,n-1}} \\
    & {z_{1,1}} & {z_{1,2}} & {...} & {z_{1,n-1}} \\
    x & {...} & {...} & {...} & {...} & y \\
    & {z_{k,1}} & {z_{k,2}} & {...} & {z_{k,n-1}}
    \arrow[from=1-2, to=2-2]
    \arrow[from=1-3, to=2-3]
    \arrow[from=1-5, to=2-5]
    \arrow[no head, from=1-2, to=1-3]
    \arrow[no head, from=2-2, to=2-3]
    \arrow[no head, from=2-3, to=2-4]
    \arrow[no head, from=2-4, to=2-5]
    \arrow[no head, from=1-3, to=1-4]
    \arrow[no head, from=1-4, to=1-5]
    \arrow[from=1-4, to=2-4]
    \arrow[no head, from=3-1, to=1-2]
    \arrow[no head, from=3-1, to=2-2]
    \arrow[no head, from=3-1, to=3-2]
    \arrow[from=2-2, to=3-2]
    \arrow[from=2-3, to=3-3]
    \arrow[from=2-4, to=3-4]
    \arrow[from=2-5, to=3-5]
    \arrow[no head, from=3-2, to=3-3]
    \arrow[no head, from=3-3, to=3-4]
    \arrow[no head, from=3-4, to=3-5]
    \arrow[from=3-2, to=4-2]
    \arrow[no head, from=4-2, to=4-3]
    \arrow[from=3-3, to=4-3]
    \arrow[no head, from=4-3, to=4-4]
    \arrow[no head, from=4-4, to=4-5]
    \arrow[from=3-5, to=4-5]
    \arrow[from=3-4, to=4-4]
    \arrow[no head, from=3-1, to=4-2]
    \arrow[no head, from=3-5, to=3-6]
    \arrow[no head, from=4-5, to=3-6]
    \arrow[no head, from=3-6, to=2-5]
    \arrow[no head, from=1-5, to=3-6]
    \end{tikzcd}
    \]
    which represent a simplex in $\Map_E(x,y)$. One may observe that for all $0\leq m \leq k$, $z_{m,1}=z'$ for some $z'\in D$ (which by definition, is not in $C$) and $z_{m,n-1}=-\infty$. Furthermore, all the other $z_{i,j}$-s are in $\{-\infty\}\cup \{z'\}_{z\in C}$. Let
    \[
    |(C_\delta)^{\lhd}|=L^H((C_\delta)^{\lhd},(C_\delta)^{\lhd})
    \]
    be the hammock localization of $(C_\delta)^{\lhd}$ in which we invert all morphisms. Note that $|(C_\delta)^{\lhd}|$ is equivalent to the terminal category.\\
    We now construct a map of simplicial sets
    \[  
    \underset{g\in \mathrm{Hom}_C(x,y)}{\coprod} \Delta^0 \sqcup \underset{z\in C}\coprod\underset{f\in \mathrm{Hom}_D(x,z')}{\coprod} \mathrm{Hom}_{L^H((C_\delta)^{\lhd},(C_\delta)^{\lhd})}(z,-\infty) \to \mathrm{Hom}_E(x,y)
    \]
    as follows: \\
    Given  
   \[
    \Delta^0\in \underset{g\in \mathrm{Hom}_C(x,y)}{\coprod} \Delta^0
   \]  
    indexed on some $g:x\to y$, we will send it to the length zero hammock that corresponds to that map.\\
    Given 
    \[
    \sigma_{z,f} \in (\underset{z\in C}\coprod\underset{f\in \mathrm{Hom}_D(x,z)}{\coprod} \mathrm{Hom}_{L^H((C_\delta)^{\lhd},(C_\delta)^{\lhd})}(z,-\infty))_k
    \]
    we define its image by expanding its source and target with identities, i.e.
    \[
    \sigma_{z,f}= \begin{tikzcd}
        & {...} & {...} \\
        {z} & {...} & {...} & {-\infty} \\
        & {...} & {...}
        \arrow[no head, from=2-1, to=1-2]
        \arrow[no head, from=2-1, to=3-2]
        \arrow[from=1-2, to=2-2]
        \arrow[from=2-2, to=3-2]
        \arrow[no head, from=1-2, to=1-3]
        \arrow[no head, from=2-2, to=2-3]
        \arrow[from=1-3, to=2-3]
        \arrow[from=2-3, to=3-3]
        \arrow[no head, from=3-2, to=3-3]
        \arrow[no head, from=3-3, to=2-4]
        \arrow[no head, from=2-1, to=2-2]
        \arrow[no head, from=2-3, to=2-4]
        \arrow[no head, from=1-3, to=2-4]
    \end{tikzcd}
    \mapsto
    \begin{tikzcd}
        & {z'} & {...} & {-\infty} \\
        x & {...} & {...} & {...} & y \\
        & {z'} & {...} & {-\infty}
        \arrow["\Id"', from=1-2, to=2-2]
        \arrow["\Id"', from=2-2, to=3-2]
        \arrow["f", from=2-1, to=1-2]
        \arrow["f"', from=2-1, to=3-2]
        \arrow["f"{description}, from=2-1, to=2-2]
        \arrow[no head, from=1-2, to=1-3]
        \arrow[no head, from=1-3, to=1-4]
        \arrow[no head, from=2-2, to=2-3]
        \arrow[from=1-3, to=2-3]
        \arrow[from=2-3, to=3-3]
        \arrow[no head, from=3-2, to=3-3]
        \arrow[no head, from=3-3, to=3-4]
        \arrow[no head, from=2-3, to=2-4]
        \arrow["\Id", from=1-4, to=2-4]
        \arrow["\Id", from=2-4, to=3-4]
        \arrow[from=1-4, to=2-5]
        \arrow[from=3-4, to=2-5]
        \arrow[from=2-4, to=2-5]
    \end{tikzcd}.
    \]
    By our analysis above this defines an isomorphism of simplicial sets, concluding that $\Map_E(x,y)$ is equivalent to a discrete space.
    
    \item Assume that $y=-\infty$.\\
    We will use a similar argument to the previous case.\\
    As before one can check that the reduced hammocks all begin with $z'\in E$ (again, those $z'$ are not in $C$) and all other vertices are in $\{-\infty\}\cup \{z'\}_{z\in C}$.\\
    We now contract a map of simplicial sets
    \[
    \underset{z\in C}\coprod\underset{f\in \mathrm{Hom}_D(x,z')}{\coprod} \mathrm{Hom}_{L^H((C_\delta)^{\lhd},(C_\delta)^{\lhd})}(z,-\infty) \to \mathrm{Hom}_E(x,y)
    \]
    by expanding the source of a given $\sigma_{z,f}\in (\underset{z\in C}\coprod\underset{f\in \mathrm{Hom}_D(x,z')}{\coprod} \mathrm{Hom}_{L^H(C_\delta)^{\lhd},(C_\delta)^{\lhd})}(z,-\infty))_k$ with identities i.e.
    \[
    \sigma_{z,f}= \begin{tikzcd}
        & {...} & {...} \\
        {z} & {...} & {...} & {-\infty} \\
        & {...} & {...}
        \arrow[no head, from=2-1, to=1-2]
        \arrow[no head, from=2-1, to=3-2]
        \arrow[from=1-2, to=2-2]
        \arrow[from=2-2, to=3-2]
        \arrow[no head, from=1-2, to=1-3]
        \arrow[no head, from=2-2, to=2-3]
        \arrow[from=1-3, to=2-3]
        \arrow[from=2-3, to=3-3]
        \arrow[no head, from=3-2, to=3-3]
        \arrow[no head, from=3-3, to=2-4]
        \arrow[no head, from=2-1, to=2-2]
        \arrow[no head, from=2-3, to=2-4]
        \arrow[no head, from=1-3, to=2-4]
    \end{tikzcd}
    \mapsto
    \begin{tikzcd}
        & {z'} & {...} \\
        x & {...} & {...} & {-\infty} \\
        & {z'} & {...}
        \arrow["\Id"', from=1-2, to=2-2]
        \arrow["\Id"', from=2-2, to=3-2]
        \arrow[no head, from=1-3, to=1-2]
        \arrow[no head, from=2-3, to=2-2]
        \arrow[from=1-3, to=2-3]
        \arrow[no head, from=3-3, to=3-2]
        \arrow[from=2-3, to=3-3]
        \arrow[no head, from=2-4, to=1-3]
        \arrow[no head, from=3-3, to=2-4]
        \arrow[no head, from=2-3, to=2-4]
        \arrow["f", from=2-1, to=1-2]
        \arrow["f"', from=2-1, to=3-2]
        \arrow["f"{description}, from=2-1, to=2-2]
    \end{tikzcd}
    \]
    As before, our analysis shows that this is an isomorphism of simplicial sets, concluding that $\Map_E(x,-\infty)$ is equivalent to a discrete space.
    \end{itemize}   

All in all we have that $E$ is equivalent to a $1$-category.
\end{proof}

\begin{definition}\label{definition: definition5}
Under the assumptions of \cref{proposition: prop3} we will say that $E$ is the \textbf{cone with left inverses on $C$} and we will denote it by $E_C$.
\end{definition}

\begin{corollary} \label{corollary: cor3}
Let $I$ be a poset and denote by $E_I$ the cone with left inverses on $I$. Denote the distinct map (if it exists) from $i$ to $j$ in $I$ by $b_{i,j}$. Then $E_I$ is equivalent to the $1$-category described below:
\begin{itemize}
    \item The objects are given by:
     \[
    \Ob(E_I) = \Ob(I) \sqcup \{-\infty\}.
    \]
    \item The morphisms sets are defined as follows:
    \[
    \hom_{E_I}(a,b) = \begin{cases}
    \{q_k\}_{k \geq i}, \quad \quad \quad \quad \ a,b \ \in I, \ \ a=i, \ b=j, \ i\not<j \\
    \{q_k\}_{k \geq i}  \cup \{b_{ij}\}, \quad a,b \ \in I, \ \ a=i, \ b=j, \ i\leq j \\
    \{g_k\}_{k \geq i},\quad \quad \quad \quad \quad \  a =i  \ \in I, \ \  b = -\infty \\
    \{h_i\},\quad \quad \quad \quad \quad \quad \quad b =i  \ \in I, \ \  a = -\infty \\
    \{\Id\}, \quad \quad \quad \quad\quad \quad \quad a = b= -\infty   \\
    \end{cases}.
    \]
    where the composition law is determined by:
    \[
    q_pq_k=q_k, \quad g_pq_k=g_k, \quad q_kb_{ij}=q_k, \quad b_{ij}q_k=q_k, \quad h_ig_k=q_k.
    \]
\end{itemize}
\end{corollary}

\subsection{Showing That $L_I \in \Cat_1$}

Given a poset $I$, we have defined the cone with retracts, $L_I$, on $I$ see \cref{construction: cons1}. In this subsection we will show that $L_I$ is equivalent to a $1$-category.

\begin{lemma} \label{lemma: lem5}
For any quasi-category $\mathcal{C}$, the restriction map:
\[
\Fun(\Ret, \mathcal{C}) \to \Fun(\wRet, \mathcal{C})
\]
is a trivial fibration of simplicial sets (here $\Ret$ and $\wRet$ are defined as in \cref{definition: def4}).
\end{lemma}

\begin{proof}
This is \cite[4.4.5.7]{lurie2009higher}.
\end{proof}

\begin{lemma}\label{lemma: id like eq}
Let $\cC \in \Cat_\infty$. We denote by $\Fun^{(0\to2) \mapsto \mathrm{Eq}}(\Delta^2,\cC)$ the category of functors which send $(0\to2)$ to an equivalence and by $\Fun^{(0\to2) \mapsto \Id}(\Delta^2,\cC)$
the category of functors which send $(0\to2)$ to the identity \footnote{Both defined formally as pullbacks like in \cref{lemma : initial as mapping to prop}.}. Following the above notions, the natural map
\[
\Fun^{(0\to2) \mapsto \Id}(\Delta^2,\cC) \to \Fun^{(0\to2) \mapsto \mathrm{Eq}}(\Delta^2,\cC)
\]
is an equivalence of categories.
\end{lemma}

\begin{proof}
The map above is clearly fully-faithful. It remains to show that any $F \in  \Fun^{(0\to2) \mapsto \mathrm{Eq}}(\Delta^2,\cC)$ is equivalent to a functor which sends $0\to2$ to the identity. Indeed, let $\sigma$ be the $2$-simplex in $\cC$
\[
\begin{tikzcd}
    & Y \\
    X && Z
    \arrow["h"', from=2-1, to=2-3]
    \arrow["f", from=2-1, to=1-2]
    \arrow["g", from=1-2, to=2-3]
\end{tikzcd}
\]
which represents $F$. Choose an inverse, $h^{-1}$, of $h$, and let $\alpha$ be a $3$-simplex we get by the horn filling property: 
\[
\begin{tikzcd}
    & Y && X \\
    X && Z
    \arrow["h"', from=2-1, to=2-3]
    \arrow["f", from=2-1, to=1-2]
    \arrow["g", from=1-2, to=2-3]
    \arrow["{h^{-1}g}", from=1-2, to=1-4]
    \arrow["{h^{-1}}"', from=2-3, to=1-4]
    \arrow["\Id"{description, pos=0.3}, from=2-1, to=1-4]
\end{tikzcd}
\]
in $\cC$. Define $F' \in \Fun^{(0\to2) \mapsto \Id}(\Delta^2,\cC)$ as the composition
\[
\Delta^2 \overset{i}\to \Delta^3 \overset{\alpha}\to \cC
\]
where $i$ is the embedding of $\Delta^2$ as the face in front of $2$. One can verify that $\alpha$ induces a natural equivalence between $F$ and $F'$ via the map $\Delta^2\times \Delta^1 \to \Delta^3$ that collapses $\{0\}\times \Delta^1$ and $\{1\}\times \Delta^1$.
\end{proof}

\begin{lemma} \label{lemma: lem6}
Let $C$ be a $1$-category. Denote by $L_C$ the cone with retracts on $C$ and by $E_C$ the cone with left inverses on $C$. Then, we have an equivalence of $\infty$-category
\[
L_C \simeq E_C.
\]
In particular $L_C$ is a $1$-category.   
\end{lemma}

\begin{proof}
Let $D_C$ be the category obtained from $C$ by gluing arrows to all the objects. By the universal property of localization, for every $\mathcal{T} \in \Cat_\infty$, composing with the natural map $D_C \to E_C$ induces a fully-faithful functor
\[
\Fun(E_C,\mathcal{T}) \hookrightarrow \Fun(D_C,\mathcal{T})
\]
whose essential image is the full subcategory on functors that sends the maps $-\infty \to x'$ to  equivalences. By the definition of $D_C$ we can write:
\begin{gather*}
\Fun^{(-\infty \to x') \mapsto \mathrm{Eq}}(D_C,\mathcal{T}) \simeq \Fun^{(-\infty \to x') \mapsto \mathrm{Eq}}(C^{\lhd} \underset{\Delta^0\times C_\delta}\coprod (\Delta^1 \times C_\delta)  ,\mathcal{T}) .
\end{gather*}
Note that the pushout $C^{\lhd} \underset{\Delta^0\times C_\delta}\coprod (\Delta^1 \times C_\delta)$ is equivalent to the pushout $C^{\lhd} \underset{\Delta^1\times C_\delta}\coprod (\Delta^2 \times C_\delta)$ where the latter is defined via the map
\[
\Delta^1\times C_\delta \to C^{\lhd}
\]
that chooses all the maps from the initial object and the map
\[
 \Delta^1\times C_\delta  \to \Delta^2 \times C_\delta
\]
that chooses the maps from $0$ to $1$. \\
From the above together with \cref{lemma: id like eq} we may write
 \begin{gather*}
\Fun^{(-\infty \to x') \mapsto \mathrm{Eq}}(C^{\lhd} \underset{\Delta^0\times C_\delta}\coprod \Delta^1 \times C_\delta  ,\mathcal{T})\simeq \Fun^{(0 \to 2) \mapsto \mathrm{Eq}}(C^{\lhd} \underset{\Delta^1\times C_\delta}\coprod \Delta^2 \times C_\delta  ,\mathcal{T})  \simeq \\
\simeq \Fun(C^{\lhd},\mathcal{T}) \underset{\Fun(\Delta^1\times C_\delta,\cT)}\times \Fun^{(0 \to 2) \mapsto \mathrm{Eq}}(\Delta^2 \times C_\delta  ,\mathcal{T}) \simeq \\
\simeq \Fun(C^{\lhd},\mathcal{T}) \underset{\Fun(\Delta^1\times C_\delta,\cT)}\times(\underset{C_\delta}\prod \Fun^{(0 \to 2) \mapsto \Id}( \Delta^2   ,\mathcal{T})).
\end{gather*}
Applying \cref{lemma: id like eq} we get
\[
\Fun^{(0\to 2)\mapsto \Id}( \Delta^2 , \mathcal{T})= \Fun( \wRet , \mathcal{T}) \simeq  \Fun( \Ret , \mathcal{T}),
\]
which gives
\[
\Fun(C^{\lhd},\mathcal{T}) \underset{\Fun(\Delta^1\times C_\delta,\cT)}\times(\underset{C_\delta}\prod \Fun^{(0 \to 2) \mapsto \Id}( \Delta^2   ,\mathcal{T})) \simeq \Fun(C^{\lhd},\mathcal{T}) \underset{\Fun(\Delta^1\times C_\delta,\cT)}\times(\underset{C_\delta}\prod \Fun( \Ret   ,\mathcal{T})) \simeq \Fun(L_I,\mathcal{T})
\]
as desired.
\end{proof}

From the above and \cref{corollary: cor3} we get an explicit description of $L_I$ when $I$ is a poset.

\begin{corollary} \label{corollary: cor2}
Let $I$ be a poset and denote by $L_I$ the cone with retracts on $I$. Denote the distinct map (if it exists) from $i$ to $j$ in $I$ by $b_{i,j}$. Then $L_I$ is equivalent to the $1$-category described below:
\begin{itemize}
    \item The objects are given by:
     \[
    \Ob(L_I) = \Ob(I) \sqcup \{-\infty\}.
    \]
    \item The morphisms sets are defined as follows:
    \[
    \hom_{L_I}(a,b) = \begin{cases}
    \{q_k\}_{k \geq i}, \quad \quad \quad \quad \ a,b \ \in I, \ \ a=i, \ b=j, \ i\not<j \\
    \{q_k\}_{k \geq i}  \cup \{b_{ij}\}, \quad a,b \ \in I, \ \ a=i, \ b=j, \ i\leq j \\
    \{g_k\}_{k \geq i},\quad \quad \quad \quad \quad \  a =i  \ \in I, \ \  b = -\infty \\
    \{h_i\},\quad \quad \quad \quad \quad \quad \quad b =i  \ \in I, \ \  a = -\infty \\
    \{\Id\}, \quad \quad \quad \quad\quad \quad \quad a = b= -\infty   \\
    \end{cases}.
    \]
    where the composition law is determined by:
    \[
    q_pq_k=q_k, \quad g_pq_k=g_k, \quad q_kb_{ij}=q_k, \quad b_{ij}q_k=q_k, \quad h_ig_k=q_k.
    \]
\end{itemize}
\end{corollary}

\section{Subcategories That Are Closed Under Pure Morphisms}

In this section we prove that a subcategory of a category of presheaves closed under limits and $\kappa$-filtered colimits is closed under $\kappa$-pure morphisms provided that $\kappa$ is sufficiently large. In more formal terms, we define:      

\begin{definition} \label{definition: def6}
Let $\mathcal{C} \in \Cat_\infty$ and $\mathcal{D} \subset \mathcal{C}$ a full subcategory. We say that \emph{$\mathcal{D}$ is closed under $\kappa$-pure morphisms in $\mathcal{C}$} if for every $B \in \mathcal{D}$ and every $A \to B$ a $\kappa$-pure morphism in $\mathcal{C}$, $A$ is also in $\mathcal{D}$.
\end{definition} 

The main objective of section $4$ is then to prove:

\begin{proposition} \label{proposition: prop4}
Let $\cE$ be a small $\infty$-category and $\mathcal{D} \subset \Fun(\cE,\mathcal{S})=:\mathcal{C}$ be a full subcategory closed under limits and $\kappa$-filtered colimits. Then, there exists $\mu \geq \kappa$ such that the $\infty$-category $\mathcal{D}$ is closed under $\mu$-pure morphisms in $\mathcal{C}$.
\end{proposition}

The above proposition is the key to proving \cref{theorem: reflaction} (the reflection theorem). We will begin by proving some preliminary lemmas for the proof of \cref{proposition: prop4}, which will be presented in the next subsection.

\begin{lemma}\label{lemma: lem7}
Let $\mathcal{C}, \mathcal{D}, \mathcal{E} \in \Cat_\infty$ where $\cE$ is complete, $F:\mathcal{C} \to \mathcal{E}$ a functor and $p:\mathcal{C} \to \mathcal{D}$ a right fibration. Assume that we have a pullback square
\[
\begin{tikzcd}
    {\mathcal{C}'} & {\mathcal{C}} \\
    {\mathcal{D}'} & {\mathcal{D}}
    \arrow["v"', from=2-1, to=2-2]
    \arrow["p", from=1-2, to=2-2]
    \arrow["q"', from=1-1, to=2-1]
    \arrow["u", from=1-1, to=1-2]
\end{tikzcd}
\]  
Then
\[
v^*p_*(F)\simeq q_* u^*(F)
\]
where $(-)^*$ is post-composition and $(-)_*$ is the right Kan extension.
\end{lemma}

\begin{proof}
By  \cite[4.4.11]{Cisinski19}, a left fibration is proper. Therefore, the claim is just proper base change \cite[6.4.13]{Cisinski19} for
\[
\begin{tikzcd}
    {\mathcal{C}'^\op} & {\mathcal{C}^\op} \\
    {\mathcal{D}'^\op} & {\mathcal{D}^\op}
    \arrow["{v^\op}"', from=2-1, to=2-2]
    \arrow["{p^\op}", from=1-2, to=2-2]
    \arrow["{q^\op}"', from=1-1, to=2-1]
    \arrow["{u^\op}", from=1-1, to=1-2]
\end{tikzcd}
\]
\end{proof}

\begin{lemma}\label{lemma: lem8}
Let $A$ be a set, $I$ a category and $F_a:T_a \to \Fun(I,\mathcal{S})$ be a collection of functors from filtered posets indexed on $A$. Then, there is a canonical equivalence:
\[
\underset{a \in A}{\prod}\colim \ F_a \simeq \underset{(c_a) \in \underset{A}{\prod}T_a}{\colim}\underset{a \in A}{\prod}F_a(c_a).
\]
\end{lemma}

\begin{proof}
Since limits and colimits of functors are computed point-wise this follows from \cite[3.10]{BarthelSchlankStapleton17}.
\end{proof}

\begin{lemma} \label{lemma: lem9}
Let $S$ be a simplicial set. Assume that for every functor $F:K \to S$ with $K$ finite simplicial set, there exists a natural transformation in $\Fun(K,S)$ to another functor $F \to F'$ where $F'$ admits a cone i.e. extents to $K^{\lhd}$. Then, $S$ is contractible.  
\end{lemma}

\begin{proof}
We will use Kan's, $\Ex^{\infty}(-)$ construction for the fibrant replacement in the Kan model structure of simplicial sets.\\
Recall that $\Ex^\infty(X)$ is defined as the following colimit in $\sset$
\[
\Ex^\infty(X) = \colim\ ( X \subset \Ex(X) \subset \Ex^2(X) \subset \cdots )
\]
where each $\Ex^k:\sset\to\sset $ is the right adjoint of the $k$-th subdivision functor.\\
By definition of fibrant replacement, any homotopy class in $\pi_n(|S|)$ can be represented by a map :
\[
\partial \Delta^n \to \Ex^{\infty}(S).
\]
Since $\partial \Delta^n$ has only finitely many non-degenerate simplices, this map must factor through some $\Ex^k(S)$, which, by adjunction, means that the homotopy class we started with can be represented by a map of simplicial sets
\[
\mathrm{Sb}^k(\partial \Delta^n) \to S.
\]
This map factors through the contractible simplicial set 
\[
(\mathrm{Sb}^k(\partial \Delta^n)\times \Delta^1) \underset{\mathrm{Sb}^k(\partial \Delta^n)\times \{1\}}\coprod \mathrm{Sb}^k(\partial \Delta^n)^{\lhd}
\]
by our assumption on $S$, which implies that $S$ is contractible.
\end{proof}

\subsection{Proof of \cref{proposition: prop4}:}

\begin{proof} In this subsection we prove \cref{proposition: prop4}.
Hence, we will assume that we are in the setting of \cref{proposition: prop4}, i.e. we have $\cC:=\Fun(\cE,\mathcal{S})$ where $\cE$ is a small $\infty$-category, and $\cD \subset \cC$ a full subcategory closed under limits and $\kappa$-filtered colimits. Fixing some more notation, let $\pi$ be the cardinal from \cref{lemma: lem1} that corresponds to $\cC$ and choose $\mu \geq \max\{\pi,\kappa, \omega^+\}$. We also fix a $\mu$-pure morphism $f:A \to B$ with $B \in \mathcal{D}$, our goal is to show that we can write $A$ in terms of limits and $\mu$-filtered colimits from $B$.\\

First, we show that $f$ gives rise to a functor from the cone with retracts on a filtered poset to $\cC_{A/}$.

\begin{construction}\label{construction: Z}
From \cref{proposition: prop1} there exists a $\mu$-filtered category, $I$ together with a quasi-split cone  $Z':I^{\lhd} \to \cC$ such that $Z' (-\infty)=A $ and $\underset{i\in I}\colim \ Z'_{|I}(i)=\underset{i\in I}\colim \  B_i = B$ where $B_i=Z'(i)$. By \cite[5.3.1.18]{lurie2009higher} we may, and will, assume that $I$ is a $\mu$-filtered poset. We choose retractions $g_i'$ for all $h_i':=Z'(-\infty\to i)$ and denote the resulting functor from the cone with retracts on $I$ (see \cref{construction: cons1}) by
\[
Z'': L  \to \mathcal{C}.
\]
By \cref{corollary: cor2} $L$ is a $1$-category and we will denote its elements (throughout this section) as we did in \cref{corollary: cor2}. Note that $-\infty$ is the initial object of $L$, and $Z''$ sends it to $A$, thus, by \cref{lemma : initial as mapping to prop}, $Z''$ extends to a functor
\[
Z: L \to \mathcal{C}_{A/}.
\]
\end{construction}

%\begin{notation} \label{notation: I_i}
%For $i\in I$ we denote $I_i:=I_{\geq i}$
%\end{notation}
We will use the following notation:
\begin{notation} \label{notation: category B}
Denote by $\mathcal{B}$ the smallest full subcategory of $\cC$ that is closed under limits and $\mu$-filtered colimits and contains $B$.
\end{notation}

Our goal is to show that $A\in \mathcal{B}$. We will do so in three steps:

\begin{enumerate}
    \item First, we will build a $\Delta^+$ indexed diagram in $\cC_{A/}$ i.e. we construct a functor
    \[
    X:\Delta^+ \to \cC_{A/}.
    \]
    \item Second, we will show that $X([n])\in \mathcal{B}_{A/}$.
    \item Third, we will show that $(\Id\colon A\to A)\in \cC_{A/}$ is a retract of $\underset{\Delta^+}\lim \ X$ in $\cC_{A/}$ .
\end{enumerate}

These three steps together conclude the proof.

\subsubsection{First Step - Constructing a Diagram}

The main goal of this subsection is to build a diagram $X:\Delta^+ \to \cC_{A/}$. We will do so by constructing a functor $\phi:\Fun(L,\cC_{A/}) \to \Fun(\Delta^+,\cC_{A/})$ and applying it to $Z$ of \cref{construction: Z}. We will construct $\phi$ by first constructing a functor $\phi': \Fun(L,\cC_{A/}) \to \Fun(J\times\Delta^+,\cC_{A/})$ for a category $J$ that we will now define, and then taking a colimit over $J$.  
 
\begin{construction}\label{construction: J}
 Let $\underset{m\in \mathbb{N}}{\prod } I^{I_\delta^m}$ be the $\mu$-filtered poset whose objects are series of functions of sets
\[
(f_0,f_1,....) \ \textrm{ where } \ f_m: I^{m}_\delta \to I_\delta
\]
with the order $(f_0,...) \leq (f_0',...)$ if $f_m(i_0,....,i_{m-1}) \leq f_m'(i_0,....,i_{m-1})$ for all $ i_0,...,i_{m-1} \in I,  \ m \in \mathbb{N}$. We will denote by $J \subset \underset{m\in \mathbb{N}}{\prod } I^{I_\delta^m}$ the full subcategory on $(f_0,f_1,...)$ satisfying
\[
f_m(i_0,...i_{k-1},i_{k+1},...i_{m}) \leq f_{m+1}(i_0,...i_{k-1},i_k,i_{k+1},...,i_{m})
\]
for all $m \in \mathbb{N}$, $0 \leq k \leq m$ and for all $i_0,...,i_m \in I$. 
\end{construction}

We will also denote:

\begin{notation}
For $(f_0,...,f_n) \in \underset{m \leq n}{\prod } I^{I_\delta^m}$ we let
\[
I_{f_0,...f_n} \subseteq I^{n+1}_\delta
\]
be the subset
\[
I_{f_0,...f_n} :=
\{(i_0,\dots, i_n) \in I^{n+1}_\delta \ | \  i_{k} \geq f_k(i_0,\dots,i_{k-1}) \quad \forall\, 0\leq k\leq n\}.
\]
\end{notation}

\begin{example}
$I_{\emptyset}$ is a set which contains a single point, and $I_{f_0}$ is $I_{\geq f_0}$. %(which agrees with \cref{notation: I_i}). 
\end{example}

We will need the following lemma for the second step.

\begin{lemma}\label{lemma: prop J}
 The following are true:
 \begin{enumerate}
    \item Let $J$ be as in \cref{construction: J}. Then $J$ is $\mu$-filtered and cofinal in $\underset{m\in \mathbb{N}}{\prod } I^{I_\delta^m}$.
    \item Let $i \in I$ and $J^i \subset \underset{m\in \mathbb{N}}{\prod } I_i^{(I_i)_\delta^m}$ be the full subcategory on $(f_0,f_1,...)$ satisfying
    \[
    f_m(i_0,...i_{k-1},i_{k+1},...i_{m}) \leq f_{m+1}(i_0,...i_{k-1},i_k,i_{k+1},...,i_{m})
    \]
    for all $m \in \mathbb{N}$, $0 \leq k \leq m$ and for all $i_0,...,i_m \in I_i$ (this is \cref{construction: J} when we replace $I$ by $I_i$). Then there exists a cofinal map $\xi: J^i \to J$ such that $\xi(f_k)_{|(I_i)_\delta^k}=f_k$ for all $k \in \mathbb{N}$.
\end{enumerate} 
\end{lemma}

\begin{proof}
\begin{enumerate}
    \item We first prove that $J$ is cofinal in $\underset{m\in \mathbb{N}}{\prod } I^{I_\delta^m}$. Let $(f_0,...)\in \underset{m\in \mathbb{N}}{\prod } I^{I_\delta^m}$. Our goal is to define $(\bar{f}_0,...)\in J$ such that $(f_0,..)\leq (\bar{f}_0,...)$ and we will do so by induction on $n$. Let $\bar{f}_0=f_0$ and assume $\bar{f}_n$ is defined. For every $(i_0,...,i_{n})\in I^n_\delta$ we pick 
    \[
    p_{i_0,...,i_n}^k \geq f_{n+1}(i_0,..,i_{n}),\bar{f}_n(i_0,...,\widehat{i_k},...,i_{n})
    \]
    for $0\leq k \leq n$ and we choose $p_{i_0,...,i_n} \geq  p_{i_0,...,i_n}^0,...,p_{i_0,...,i_n}^n$. We define $\bar{f}_{n+1}(i_0,...i_n)= p_{i_0,...,i_n}$ and by construction we have that $\bar{f}_n(i_0,...,\widehat{i_k},...,i_{n})\leq \bar{f}_{n+1}(i_0,...,i_{n})$ and that $(f_0,..)\leq (\bar{f}_0,...)$.\\

    Since every cofinal subset of a $\mu$-filtered poset is itself $\mu$-filtered, the first part follows as well.
    
    %We now show that $J$ is $\mu$-filtered. Let $F:K \to J$ be a functor with $K$ a $\mu$-small simplicial set. Since $\underset{m\in \mathbb{N}}{\prod } I^{I_\delta^m}$ is $\mu$-filtered, after composing with the inclusion $J\subset \underset{m\in \mathbb{N}}{\prod } I^{I_\delta^m} $, $F$ admits a cocone, which we denote by $(f_0,...)\in \underset{m\in \mathbb{N}}{\prod } I^{I_\delta^m}$. Thus $(\bar{f}_0,...)\in J$ form the previous paragraph is a cocone on $F$.
    
    \item We will construct $\xi$ explicitly. For every $j\in I $ such that $i\not\leq j$ we pick $p_j\geq i,j$. Given $f:(I^n_i)_\delta \to (I_i)_\delta$ we define 
    \[
    \tilde{f}:I^n_\delta \to I_\delta, \quad (i_0,...i_{n-1}) \mapsto f(i_0',...i_{n-1}') \quad \textrm{where} \quad i_k'=\begin{cases}i_k, \quad  i_k \geq i \\
    p_{i_k}, \quad i_k \not \geq i \end{cases}.
    \]
    We let $\xi: J^i \to J$ be defined by $\xi(f_0,...)=(\tilde{f_0},...)$. One can verify that this map is cofinal.
\end{enumerate}
\end{proof}

Recall, that our first goal is to define a functor $\phi':\Fun(L,\cC_{A/}) \to \Fun(J\times\Delta^+,\cC_{A/})$. We will do so by defining a category $M$, functors
$F:M \to J \times \Delta^+$, $ H:M \to L$
and taking a right Kan extension:
\[
\begin{tikzcd}
M & L & \mathcal{C}_{A/} \\
{J \times \Delta^+}
\arrow["F"', from=1-1, to=2-1]
\arrow["H", from=1-1, to=1-2]
\arrow["(-)", from=1-2, to=1-3]
\arrow["F_*((-)\circ H)"', dashed, from=2-1, to=1-3]
\end{tikzcd}
\]
We begin by defining $F:M\to J \times \Delta^+$ as a right fibration: 

\begin{construction} \label{construction: F}
Let $G: (J \times \Delta^+)^{\mathrm{op}} \to \mathrm{Set}$ be defined as follows: \\
For an object $((f_0,...),[n]) \in J \times \Delta^+$ let
\[
G((f_0,...),[n])= I_{f_0,..,f_{n-1}}.
\]
For a morphism $(j,h)$ where $j:(f_0,...) \to (f_0',...) \in J$ and $h: [k] \to [n]$ is a strictly increasing function let
\[
G(j,h):I_{f_0',..f_{n-1}'}\to I_{f_0,..f_{k-1}},\quad (i_0,...i_{n-1})\mapsto(i_{h(0)},...,i_{h(k-1)}).
\]
It follows from the definition of $J$ that this defines a functor. We let $F:M \to J \times \Delta^+$ be the right fibration over $J \times \Delta^+$ classified by $G$. 
\end{construction} 

\begin{remark} \label{remark: explicit M}
Note that the objects of $M$ are $3$-tuples $((f_0,...), \ [n], \ (i_0,...,i_{n-1}))$ with $(f_0,...) \in J$, $n\in \mathbb{N}$ and $(i_0,...,i_{n-1}) \in I_{f_0,...,f_{n-1}}$, and a morphism:
\[
\varphi:((f_0,...) , \ [k], \ (i_0,...,i_{k-1})) \to ((f_0',...), \ [n], \ (i_0',...,i_{n-1}'))
\]
over $h:[k] \to [n]$ implies that $i'_{h(m)}=i_m$ if $m \not=k$.
\end{remark}

We shall now define a functor $H\colon M \to L$, as follows (see \cref{corollary: cor2} for the explicit description of $L$):\\
For  $((f_0,...),[n],(i_0,...,i_{n-1})) \in M$ let:
\[
H((f_0,...),[n],(i_0,...,i_{n-1}))=f_n(i_0,...,i_{n-1}).
\]
For $\varphi_{j,h}:((f_0,...),[k],(i_0,...,i_{k-1})) \to ((f_0',...),[n],(i_0',...,i_{n-1}'))$ over $(j:(f_0,...) \to (f_0',...), h: [k] \to [n])$ let:
\[
H(\varphi_{j,h})= \begin{cases} b_{f_k(i_0,...i_{k-1}),f_n'(i_0',...,i_{n-1}')},  \quad n= h(k)\\
q_{i_{h(k)}'},  \ \ \quad \quad \quad \quad \quad \quad \quad \ \ \quad  n \not= h(k)
\end{cases}.
\]
By the definition of $J$ this is well defined as a function of $\Hom$ sets.

\begin{lemma}\label{lemma: H is a functor}
$H\colon M \to L$ defined above is a functor.
\end{lemma}

\begin{proof}
We shall check that this is a functor using the explicit description of $L$ as in \cref{corollary: cor2}. Let $[k] \overset{h}\to [m] \overset{g}\to [n]$ be $2$ composable arrows in $\Delta^+$. There are $2$ possibilities:
\begin{itemize}
    \item If $n= g \circ h(k)$, then
    \begin{gather*}
    H(\varphi_{j',g}) \circ H(\varphi_{j,h})=b_{f_k(i_0,...i_{k-1}),f_m'(i_0',...,i_{m-1}')}b_{f_m'(i_0',...,i_{m-1}'),f_n''(i_0'',...i_{n-1}'')}=\\= b_{f_k(i_0,...i_{k-1}),f_n''(i_0'',...,i_{n-1}'')}= H(\varphi_{j' \circ j,g\circ h}).   
    \end{gather*}
        
    \item If $n \not= g \circ h(k)$. Then, either $m \not=h(k)$ which gives:
    \[
    H(\varphi_{j',g}) \circ H(\varphi_{j,h})=s\circ q_{i'_{h(k)}}=q_{i'_{h(k)}}=q_{i''_{gh(k)}}= H(\varphi_{j' \circ j,g\circ h})
    \]
    where $s$ is some morphism. Or $m=h(k)$ and $n \not= g(m)$ which gives  
    \[
    H(\varphi_{j',g}) \circ H(\varphi_{j,h})=q_{i''_{g(m)}}b_{f_m'(i_0',...,i_{m-1}'),f_n''(i_0'',...i_{n-1}'')}=q_{i''_{g(m)}}=q_{i''_{gh(k)}}= H(\varphi_{j' \circ j,g\circ h}).
    \]
    \end{itemize}
\end{proof} 

We can now define $\phi$ and $\phi'$:

\begin{definition} \label{definition: phi}
Let $F$ be as in \cref{construction: F} and $H$ be as in \cref{lemma: H is a functor}. Denote 
\[
\phi':= F_*H^*:\Fun(L,\cC_{A/})\to \Fun(J\times \Delta^+,\cC_{A/}).
\] 
By adjunction, we can view $\phi'$ as a functor $\Fun(L,\cC_{A})\to \Fun(J,\Fun(\Delta^+,\cC_{A/}))$, so we can define \[
\phi:=\underset{J}\colim \ \phi':\Fun(L,\cC_{A/})\to \Fun(\Delta^+,\cC_{A/}) .
\] 
\end{definition} 

We now apply our functors to $Z$ of \cref{construction: Z}:
 
\begin{notation} \label{notation: X Y}
Let $Z$ be as in \cref{construction: Z} and $\phi$, $\phi'$ be as in \cref{definition: phi}. We denote $Y:=\phi'(Z)$ i.e. $Y$ is the right Kan extension:
\[
\begin{tikzcd}
M & L & \mathcal{C}_{A/} \\
{J \times \Delta^+}
\arrow["F"', from=1-1, to=2-1]
\arrow["H", from=1-1, to=1-2]
\arrow["Z", from=1-2, to=1-3]
\arrow["Y"', dashed, from=2-1, to=1-3]
\end{tikzcd}
\]
and we let $X:=\phi(Z)$ i.e.
\[
X=\underset{J}{\colim} \ Y: \Delta^+\to \mathcal{C}_{A/}.
\]

\end{notation}

It will be useful for the second step to describe $Y$ more explicitly: 

\begin{remark} \label{remark: expli Y }
Let $Y$ be as in \cref{notation: X Y}. By \cref{lemma: lem7}, we can ``describe" $Y$ more explicitly: For an object $((f_0,..),[n])\in J \times \Delta^+$ we have:
\[
Y((f_0,..),[n]) \simeq \underset{(i_0,...,i_{n-1})\in I_{f_0,...f_{n-1}}}{\prod}Z(f_{n}(i_0,...,i_{n-1}))=\underset{(i_0,...,i_{n-1})\in I_{f_0,...f_{n-1}}}{\prod}(g_{f_{n}(i_0,...,i_{n-1})}':A \to B_{f_{n}(i_0,...,i_{n-1})})
\]
and for a morphism $(j,h)\in \Mor(J \times \Delta^+)$ we have that $Y(j,h)$ is the product of the maps over it. 
\end{remark}

This concludes the first step.

\subsubsection{Second Step - Finding a Diagram We Can Write in Terms of Limits and $\mu$-Filtered Colimits From $B$}

In this subsection our goal is to show that $X([n]) \in \mathcal{B}_{A/}$. We begin by defining a functor

\begin{definition}\label{defenition: *}
Let $I$ be the $\mu$-filtered poset in \cref{construction: Z}. Define
\[ 
(-)^*:\mathcal{C} \to \mathcal{C} \quad C \mapsto \underset{i\in I}{\colim}\underset{(I_{i})_{\delta}}{\prod}C
\]
where the maps $\underset{(I_{i})_{\delta}}{\prod}C \to \underset{(I_{j})_{\delta}}{\prod}C$ for $i \leq j$ are projections.
\end{definition}

\begin{remark}
Informally, for $C\in \mathcal{C}_{A/}$, $C^*$ is the space of $I$-indexed series in $C$ when we identify two series that "agree at some point".
\end{remark}

\begin{definition}
For an object $C \in \cC$ we let $C^{(0)}=C$ and define recursively $C^{(n)} =(C^{(n-1)})^*$. 
\end{definition}

\begin{lemma}\label{lemma: X([n])}
Let $X:\Delta^+ \to \cC_{A/}$ be as in \cref{notation: X Y} i.e. $X= \phi(Z)$ where $Z:L \to \cC_{A/}$ is a functor that preserves the initial object and $\phi$ is as in \cref{definition: phi}. Recall that $\underset{I}\colim \ Z_{|I}=f:A\to B$. Then,
\[
X([n])\simeq (h:A \to B^{(n)})\in \cC_{A/}.
\]
where $h:A \to B^{(n)}$ is some morphism.
\end{lemma}

\begin{proof} Before proving the lemma we fix some notation. Let $p_\cC: \cC_{A/} \to \cC$ be the forgetful functor and denote $X'=p_\cC\circ X$, $Y'=p_\cC\circ Y$. In our revised notation we need to show that $X'([n])=B^{(n)}$. It should be noted, that because $\mu$-filtered colimits in $\mathcal{C}_{A/}$ are computed in $\mathcal{C}$, we have that $X'=\underset{J}{\colim} \ Y'$.\\

We shall now prove the lemma by induction on $n$:\\

For $n=0$: By \cref{remark: expli Y }, $ Y'((f_0,...),[0])=B_{f_0}$ and
\[
Y'(j,[0])=Z'(b_{f_0.f_0'})=b_{f_0,f_0'}':B_{f_0} \to B_{f_0'}
\]
for a morphism $j:(f_0,...)  \to  (f_0',...)$ in $J$.
Therefore, as the map
\[
J \to I, \quad (f_0,...) \mapsto f_0
\]
is cofinal, we have that $X'([0])\simeq B$.\\

We now prove for $n>0$: Let $i \in I$ and  $J^i \subset \underset{m\in \mathbb{N}}{\prod } I_i^{(I_i)_\delta^m}$ be the full subcategory on $(f_0,f_1,...)$ satisfying
\[
f_m(i_0,...i_{k-1},i_{k+1},...i_{m}) \leq f_{m+1}(i_0,...i_{k-1},i_k,i_{k+1},...,i_{m})
\]
for all $m \in \mathbb{N}$, $0 \leq k \leq m$ and for all $i_0,...,i_m \in I_i$ (this is \cref{construction: J} when we replace $I$ by $I_i$). Note, that by \cref{lemma: prop J} we can choose a cofinal map $\xi: J^i \to J$ such that $\xi(f_k)_{|(I_i)_\delta^k}=f_k$ for all $k \in \mathbb{N}$ and denote by $J_n^i$, the image of $J^i$ under the projection
\[
\underset{m\in \mathbb{N}}{\prod } I_i^{(I_i)_\delta^m} \to \underset{m \leq n}{\prod } I_i^{(I_i)_\delta^m}.
\]
Based on the above and the fact that $Y'(-,[n])$ only depends on the first $n$ coordinates of $J$, we can write: 
\begin{gather*}
B^{(n+1)}=\underset{f_0\in I}{\colim}\underset{(I_{f_0})_{\delta}}{\prod} B^{(n)}\simeq\underset{f_0\in I}{\colim}\underset{(I_{f_0})_{\delta}}{\prod} X'([n])=
\\
=\underset{f_0\in I}{\colim}\underset{(I_{f_0})_{\delta}}{\prod} \underset{J}{\colim} \ Y'(-,[n])\simeq
\underset{f_0\in I}{\colim}\underset{(I_{f_0})_{\delta}}{\prod} \underset{J}{\colim}\underset{I_{f_0,...,f_{n-1}}}{\prod}B_{f_n(i_0,...,i_{n-1})} \simeq  
\\
\simeq \underset{f_0\in I}{\colim}\underset{(I_{f_0})_{\delta}}{\prod} \underset{J^{f_0}_n}{\colim}\underset{I_{f_0,...,f_{n-1}}}{\prod}B_{f_n(i_0,...,i_{n-1})}
\end{gather*}
where all the equalities are by definition, the first equivalence follows from our induction hypothesis, the second equivalence follows from \cref{remark: expli Y }, and the last equivalence follows from our last remark (cofinality of $J_n^i$ in $J_n$ and the fact that $Y'(-,[n])$ only depends on the $n$ first coordinates).\\

Furthermore, by \cref{lemma: lem8} we can write:
\begin{gather*}
\underset{f_0\in I}{\colim}\underset{(I_{f_0})_{\delta}}{\prod} \underset{J^{f_0}_n}{\colim}\underset{I_{f_0,...,f_{n-1}}}{\prod}B_{f_n(i_0,...,i_{n-1})} \simeq \\
\simeq \underset{f_0\in I}{\colim} \underset{(J^{f_0}_n)^{(I_{f_0})_{ \delta}}}{\colim}\underset{i_0 \in (I_{f_0})_{\delta}}{\prod} \underset{i_1 \in (I_{f_1'(i_0)})_{\delta}}{\prod}...\underset{i_{n} \in (I_{f_{n}'(i_0,...,i_{n-1})})_{\delta}}{\prod}B_{f_{n+1}'(i_0,...,i_{n})}
\end{gather*}
where we think of $(J^{f_0}_n)^{I_{f_0}}$ as a subset of $\underset{m=1}{\overset{n+1}{\prod }}I_{f_0}^{(I_{f_0})_\delta^m}$. \\

Finally, let $J^{f_0}_{n+1,\geq 1}$ be the image of the projection $J^{f_0}_{n+1}\to\underset{m=1}{\overset{n+1}{\prod }} I_{f_0}^{(I_{f_0})_\delta^m}$ and note that, by a similar argument to the one in the proof of $1$ in \cref{lemma: prop J}, it is cofinal in $(J^{f_0}_n)^{I_{f_0}}$. Therefore, we can write:
\begin{gather*}
\underset{f_0\in I}{\colim} \underset{(J^{f_0}_n)^{(I_{f_0})_{\delta}}}{\colim}\underset{i_0 \in (I_{f_0})_{\delta}}{\prod} \underset{i_1 \in (I_{f_1'(i_0)})_{\delta}}{\prod}...\underset{i_{n} \in (I_{f_{n}'(i_0,...,i_{n-1})})_{\delta}}{\prod} B_{f_{n+1}'(i_0,...,i_{n})} \simeq
\\
\simeq \underset{f_0\in I}{\colim} \underset{J^{f_0}_{n+1,\geq 1}}{\colim}\underset{i_0 \in (I_{f_0})_{\delta}}{\prod} \underset{i_1 \in (I_{f_1(i_0)})_{\delta}}{\prod}...\underset{i_{n} \in (I_{f_{n}(i_0,...,i_{n-1})})_{\delta}}{\prod} B_{f_{n+1}(i_0,...,i_n)}=
\\
= \underset{J_{n+1}}{\colim}\underset{I_{f_0,...f_n}} \prod B_{f_{n+1}(i_0,...,i_n)}=X'([n+1])
\end{gather*}
where the second to last equality follows from the definition of $J_{n+1}$.
\end{proof}

\begin{remark}
We note that all the equivalences in the above lemma are in fact equivalences of diagrams.  
That is, the diagram
\[
(f_0, \dots, f_n) \mapsto \prod_{I_{f_0, \dots, f_n}} B_{f_{n+1}(i_0, \dots, i_n)}
\]
obtained through the interchange steps in the lemma is (up to the difference between $J_{n+1}$ and $J$) exactly the one defining $X'([n+1])$.

To see that this is an equivalence of diagrams, observe that each step in the interchange arises by pulling back along a functor between products of posets.  
The question therefore reduces to a $1$-categorical one and can be checked arrow-by-arrow, where it becomes a straightforward verification.
\end{remark}

This concludes the second step.

\subsubsection{Third Step - Showing That $(\Id\colon A\to A)$ Is a Retract of $\lim_{\Delta^+} \ X$}

This subsection concludes the proof of \cref{proposition: prop4}. We begin by explaining why showing that $(\Id\colon A\to A)$ is a retract of $\underset{\Delta^+}\lim \ X$ proves the proposition.

\begin{lemma}\label{lemma: why enough ret}
Let $X:\Delta^+ \to \cC_{A/}$ be as in \cref{notation: X Y} i.e. $X=\phi(Z)$ where $Z:L \to \cC_{A/}$ is as in \cref{construction: Z} and $\phi$ is as in \cref{definition: phi}. If $(\Id\colon A\to A)$ is a retract of $\underset{\Delta^+}\lim \ X$, then  \cref{proposition: prop4} follows.
\end{lemma}

\begin{proof}
Assume that $(\Id\colon A\to A)$ is a retract of $\underset{\Delta^+}\lim \ X$. As limits in $\cC_{A/}$ are computed in $\cC$, it follows that $A$ is a retract of $\underset{\Delta^+}\lim \ p_{\cC} \circ X$ where $p_{\cC}:\cC_{A/} \to \cC$ is the forgetful functor. By \cref{lemma: X([n])} we get that $A$ is a retract of an object in $\mathcal{B}$ where $\mathcal{B}$ is as in \cref{notation: category B}. Thus, since $\mathcal{B}$ is complete and therefore by \cite[4.4.5.14]{lurie2009higher} it is idempotent complete, it follows that $A\in \mathcal{B}$.
\end{proof}

We will need the following definitions

\begin{definition} \label{definition: psi}
Let $Z \colon L \to \cC_{A/}$ be a functor that preserves the initial object. We define $\psi(Z) \in \cC_{A/}$ to be 
\[
\psi(Z):= (F_1)_*(F_*)H^*(Z)
\]
For the functors on the diagram:
\[
\begin{tikzcd}
	M & L & {\cC_{A/}} \\
	{J \times \Delta^+} \\
	{J}
	\arrow["H", from=1-1, to=1-2]
	\arrow["Z", from=1-2, to=1-3]
	\arrow["F"', from=1-1, to=2-1]
	\arrow["{F_1}"', from=2-1, to=3-1]
\end{tikzcd}
\]
where $H$ is as in \cref{lemma: H is a functor}, $F$ is as in \cref{construction: F} and $F_1$ is the projection.
\end{definition}

\begin{notation} \label{notation: Z pord}
Let $i:* \to L$ be the functor that chooses $-\infty$ and $Z:L \to \cC_{A/}$ be as in \cref{construction: Z}. We denote $i_*i^*Z=:Z_{\prod A}$.
\end{notation}

Our proof that $(\Id\colon A\to A)$ is a retract of $\underset{\Delta^+}\lim \ X$ will be divided into four parts: First, we will write $ \psi(Z_{\prod A})(j)$ as a constant limit on $(\Id:A \to A)$ in $\cC_{A/}$ indexed on some diagram which we denote by $M^j_{/-\infty}$. Second, we will prove that $M^j_{/-\infty}$ is contractible and thus $\underset{J}\colim \ \psi(Z_{\prod A})=(\Id:A \to A)$. Third, we will show $\underset{J}\colim \ \psi(Z_{\prod A})=\underset{\Delta^+}\lim \ \phi(Z_{\prod})$ where $\phi$ is as in \cref{definition: phi}. Lastly, we will observe that there exists a map $\phi(Z) \to \phi(Z_{\prod})$ and that $(\Id:A \to A)$ is the initial object in $\cC_{A/}$, which together with the previous parts proves the claim.\\

We first want to show that for a fixed $j\in J$, $\psi(Z_{\prod})(j)$ is equivalent to the limit of the constant diagram on $(\Id\colon A\to A)$, indexed on some category defined in \cref{lemma: conastant limit}.
 
\begin{lemma}  \label{lemma: conastant limit}
Let $j\in J$ and denote the pullback
\[
\begin{tikzcd}
{M^j} & M \\
{\{j\} \times \Delta^+ } & {J \times \Delta^+}
\arrow["p"', from=2-1, to=2-2]
\arrow["F" '', from=1-2, to=2-2]
\arrow["{F'}"', from=1-1, to=2-1]
\arrow["{p'}", from=1-1, to=1-2]
\end{tikzcd}
\]
by $M^j$. Denote the comma category over $-\infty\in L$, defined by $H:M \to L$ (see \cref{lemma: H is a functor} for the definition of $H$) by $M^j_{/-\infty}$. Let $Z_{\prod}:L \to \cC_{A/}$ be as in \cref{notation: Z pord}. Then $\psi(Z_{\prod})(j)$ is equivalent to the limit of the constant $M^j_{/-\infty}$ indexed diagram on $(\Id\colon A\to A)$ where $\psi$ is as in \cref{definition: psi}.
\end{lemma}

\begin{proof}
Applying \cref{lemma: lem7} to the following pullback diagram
\[
\begin{tikzcd}
{M^j} & M \\
{\{j\} \times \Delta^+ } & {J \times \Delta^+}
\arrow["p"', from=2-1, to=2-2]
\arrow["F" '', from=1-2, to=2-2]
\arrow["{\widetilde{F}}"', from=1-1, to=2-1]
\arrow["{\widetilde{p}}", from=1-1, to=1-2]
\end{tikzcd}
\]
yields an equivalence $p^*F_*\simeq \widetilde{F}_*\widetilde{p}^*$. And similarly, denoting $q=H\circ \widetilde{p}$, and applying \cite[5.3.9]{riehl2017kan} to the following lax commutative square
\[
\begin{tikzcd}
{M^j_{/-\infty}} & {*} \\
{M^j} & L
\arrow["{q}"', from=2-1, to=2-2]
\arrow["i" '', from=1-2, to=2-2]
\arrow["\widetilde{i}" ', from=1-1, to=2-1]
\arrow["\widetilde{q}" '', from=1-1, to=1-2]
\end{tikzcd}
\]
yields an equivalence $q^*i_*\simeq \widetilde{i}_* \widetilde{q}^*$. \\

We can arrange both of the above squares in the following commutative diagram 
\[\begin{tikzcd}
	{M^j_{/ -\infty}} && {*} & {\cC_{A/}} \\
	{M^j} & M & L & {} \\
	{j\times\Delta^+} & {J \times \Delta^+} \\
	& {}
	\arrow["H", from=2-2, to=2-3]
	\arrow["F"', from=2-2, to=3-2]
	\arrow["p"', from=3-1, to=3-2]
	\arrow["{i}", from=1-3, to=2-3]
	\arrow["{p'}", from=2-1, to=2-2]
	\arrow["{\tilde{F}}"', from=2-1, to=3-1]
	\arrow["{\tilde{q}}", from=1-1, to=1-3]
	\arrow["{\tilde{i}}"', from=1-1, to=2-1]
	\arrow["{c_{(\Id\colon A\to A)}}", from=1-3, to=1-4]
\end{tikzcd}\]
where $c_{(\Id\colon A\to A)}$ is the functor $* \to \cC_{A/}$ that chooses $(\Id\colon A\to A)$. Recall that $\phi'(Z_{\prod})=F_*H^*i_*(c_{(\Id\colon A\to A)})$. Hence, by the previous paragraph: 
\[
\phi'(Z_{\prod})(j,-)=p^*F_*H^*i_*(c_{(\Id\colon A\to A)})\simeq\widetilde{F}_*\widetilde{p}^*H^*i_*(c_{(\Id\colon A\to A)})\simeq\widetilde{F}_*\widetilde{i}_* \widetilde{q}^*(c_{(\Id\colon A\to A)}).
\]
From the above, and the fact that a limit is the right Kan extension to a point, we conclude that $\psi(z_{\prod})(j)=\underset{\Delta^+}\lim  \ \phi'(Z_{\prod})(j,-)$ is equivalent to the limit of the constant $M^j_{/-\infty}$ indexed diagram on $(\Id\colon A\to A)$.
\end{proof}

According to our program we now want to show that $M^j_{/-\infty}$ is contractible. In order to do so we will first need to describe $M^j_{/-\infty}$ more explicitly.

\begin{remark} \label{remark: explicit M^j}
Let us write an explicit description of $M^j_{/-\infty}$. First, as $j\in J$, by definition $j=(f_0,...)$, i.e. $j$ is a tuple of functions $f_m:I_\delta^m\to I_\delta$ satisfying
\[
f_m(i_0,...i_{k-1},i_{k+1},...i_{m}) \leq f_{m+1}(i_0,...i_{k-1},i_k,i_{k+1},...,i_{m})
\]
where $m \in \mathbb{N}$, $0 \leq k \leq m$ and $i_0,...,i_m \in I$ (see \cref{construction: J}). Recall that $M^j$ is the fiber over $j\in J$ of the map $M \to J \times \Delta^+$, thus, from \cref{remark: explicit M}, we can describe $M^j$ as follows: \\
An object in $M^j$ is a $2$-tuple $([n], \ (i_0,...,i_{n-1}))$ with $n\in \mathbb{N}$ and $(i_0,...,i_{n-1}) \in I_{f_0,...,f_{n-1}}$ (i.e. $i_k \geq f_k(i_0,...,i_{k-1})$).\\
A morphism $\varphi:([k], \ (i_0,...,i_{k-1})) \to ([n], \ (i_0',...,i_{n-1}'))$ in $M^j$ over $h:[k] \to [n]$ exists (and is the unique morphism over $h$, with the same source and target as $\varphi$) if and only if $i'_{h(m)}=i_m$ for all $m\not=k$.\\
Lastly, recall how the functor which defines the comma category, $H:M^j\to L$ \footnote{Here, and in what follows, we abuse notation by identifying $H$ with its restriction to $M^j$.}, is defined (see \cref{corollary: cor2} for the explicit description of $L$) - \\
For objects:
\[
H([n], \ (i_0,...,i_{n-1}))=f_n(i_0,...,i_{n-1})
\]
and for morphisms:
\[
H(\varphi)=\begin{cases} b_{f_k(i_0,...,i_{k-1}),f_n(i_0',...,i_{n-1}')}, \quad n=h(k) \\
        q_{i'_{h(k)}}, \quad\quad\quad\quad\quad\quad\quad\quad\quad \ n \not= h(k) \end{cases}
\]
where $\varphi$ is over $h:[k] \to [n]$ as before (this is a functor by \cref{lemma: H is a functor}).\\
Adding everything together, we have the following description of $M^j_{/ -\infty}$:\\
The objects of $M^j_{/ -\infty}$ consist of $3$-tuples $([n],\ (i_0,...,i_{n-1}), \ g_{i_n}:f_n(i_0,...,i_{n-1})\to -\infty)$ with $([n], \ (i_0,...,i_{n-1})) \in M^j$ and $g_{i_{n}}:f_n(i_0,...,i_{n-1})\to -\infty$ a morphism in $L$ as in \cref{corollary: cor2}. Since a morphism $g_{i_{n}}:f_n(i_0,...,i_{n-1})\to -\infty$ is determined by an element $i_{n} \geq f_n(i_0,...,i_{n-1})$ we get that we can view objects of $M^j_{/ -\infty}$ as $2$-tuples $([n],\ (i_0,...,i_{n}))$ with $[n]\in \Delta^+$ and $(i_0,...,i_{n})) \in I_{f_0,...f_n}$.   \\
The morphisms of $M^j_{/ -\infty}$ are defined as morphism in the comma category i.e. a morphism:
\begin{gather*}
\varphi:([k],\ (i_0,...,i_{k}), )  \to ([n],\ (i_0',...,i_{n}'))
\end{gather*}
is a morphism $\widetilde{\varphi}:([k],\ (i_0,...,i_{k-1}))\to ([n],\ (i_0',...,i_{n-1}'))$ in $M^j$ such that the following diagram commutes:
\[
\begin{tikzcd}
{ f_k(i_0,...,i_{k-1})} && {f_n(i_0',...,i_{n-1}')} \\
 & {-\infty}
\arrow["{H(\widetilde{\varphi})}", from=1-1, to=1-3]
\arrow["{g_{i_k}}"', from=1-1, to=2-2]
\arrow["{g_{i_n}}", from=1-3, to=2-2]
 \end{tikzcd}
 \]
 In other words a morphism $\varphi:([k],\ (i_0,...,i_{k})) \to ([n],\ (i_0',...,i_{n}'))$ over $h: [k] \to [n]$ exists (and is the unique morphism over $h$, with the same source and target as $\varphi$) if and only if $i_m=i_{h(m)}'$ for all $0\leq m \leq k$.
\end{remark}

Using the above description we will show:
        
\begin{lemma} \label{lemma: Mj contra}
For every $j\in J$ the category $M^j_{/-\infty}$ is contractible, where $M^j_{/-\infty}$ is as in \cref{lemma: conastant limit}.
\end{lemma}
         
\begin{proof}
We will show that $M^j_{/-\infty}$ satisfies the conditions of \cref{lemma: lem9}. Indeed, let $E:K \to M^j_{/-\infty}$ be a finite diagram. Because $M^j_{/-\infty}$ is a $1$-category we may assume that $K$ is also a $1$-category. \\
For any $k \in K_\delta$ denote
\[
E(k)=([n_k], \ (i^k_0,...,i^k_{n_k})).
\]
Since $I$ is $\mu$-filtered, we can pick $p \in I$ such that:
\[
p \geq f_{n_k+1}(i_0^k,....i_{n_k}^k) \quad \forall \ k \in K_\delta.
\]

We proceed to define $E':K \to  M^j_{/-\infty}$ (which will play the role of $F'$ in \cref{lemma: lem9}) as follows:\\
For objects:
\[
E'(k)=([n_k+1], \ (i_0^k,....i_{n_k}^k,p))
\]
For morphisms: Given, $s:k \to k' \in \Mor(K)$ such that $E(s)$ is over
$h:[n_k] \to [n_{k'}]  \in \Mor(\Delta^+)$, define
\[
h':[n_k+1] \to  [n_{k'}+1], \quad m \mapsto \begin{cases} h(m), \quad  \ \ m \not= n_k+1 \\
n_{k'}+1,  \quad m = n_k+1 \end{cases}.
\]
and let 
\begin{gather*}
E'(s):([n_k+1], \ (i_0^k,....i_{n_k}^k,p)) 
\to ([n_{k'}+1], \ (i_0^{k'},....i_{n_{k'}}^{k'},p))
\end{gather*}
be the unique morphism over $h'$, with the above source and target. Note that by \cref{remark: explicit M^j} such a morphism exists (uniqueness is evident from the discussion in \cref{remark: explicit M^j}) in $M^j_{/-\infty}$ if and only if $i^{k'}_{h'(m)}=  i^k_m$ for all $0 \leq m\leq n_k$ and $p=p$ i.e. if and only if $i^{k'}_{h(m)}=  i^k_m$ for all $0 \leq m\leq n_k$ which follows since $E(s)$ is a morphism in $M^j_{/-\infty}$. Furthermore, since the maps in $M^j_{/-\infty} $ are completely determined by their source, target and image in $\Delta^+$ it is obvious that the definition of $E'$ assembles to a functor.\\

Now, it remains to show that we have a natural transformation $\theta:E \to E'$ and that $E'$ admits a cone.\\

We turn to define $\theta$. We define
\[
\theta_k: ([n_k],(i_0^k,...,i_{n_k}^k)) \to ([n_k+1],(i_0^k,...,i_{n_k}^k,p))
\]
as the unique morphism in $M^j_{/-\infty}$ over $d^{n_k}:[n_k]\to [n_k+1]$ (the map in $\Delta^+$ which skips $n_k+1$). Note that such a morphism exists in $M^j_{/-\infty}$ because $i_m^k=i_{d^n(m)}^k$ for $0\leq m\leq n_k$. Furthermore, by the commutativity of the following diagram
\[
\begin{tikzcd}
{[n+1]} \arrow[r, "h'"]                       & {[n'+1]}\\
{[n]} \arrow[r, "h"'] \arrow[u, "d^n"] & {[n']} \arrow[u, "d^n"']
\end{tikzcd}
\]
the different $\theta_k$ assemble to a natural transformation.\\

Finally, we need to provide a natural transformation from a constant functor to $E'$. We define $\psi_k:([0],p)\to([n_k+1],(i_0^k,...i_{n_k}^k,p))$ as the unique map in $M^j_{/-\infty}$ over the map $\{0\} \mapsto \{n+1\} \mapsto [n+1]$ in $\Delta^+$. The different $\psi_k$ obviously assemble to a natural transformation and thus we conclude that $M^j_{/-\infty}$ satisfies the conditions of \cref{lemma: lem9} and therefore contractible. 
\end{proof}

We deduce:

\begin{corollary} \label{corollary: lim A}
Let $Z_{\prod}:L \to \cC_{A/}$ be as in \cref{notation: Z pord}. Then $\underset{J}\lim \ \psi(Z_{\prod})=(\Id\colon A \to A)$, where $\psi$ is as in \cref{definition: psi}.
\end{corollary}

\begin{proof}
By \cref{lemma: conastant limit} and \cref{lemma: Mj contra}  $\psi(Z_{\prod})(j)=(\Id\colon A\to A)$. Since $J$ is contractible and $(\Id\colon A\to A)$ is the initial object of $\cC_{A/}$ the corollary follows.
\end{proof}

We now use the fact that $\mu$-small limits commute with $\mu$-filtered colimits in $\mathcal{S}$ in order to show that $\underset{J}\colim \ \psi(Z_{\prod A})=\underset{\Delta^+}\lim \ \phi(Z_{\prod}) $.\\

\begin{lemma} \label{lemma: colim lim}
Let $Z_{\prod}:L \to \cC_{A/}$ be as in \cref{notation: Z pord}. Then $\underset{J}\colim \ \psi(Z_{\prod A})=\underset{\Delta^+}\lim \ \phi(Z_{\prod})$, where $\psi$ is as in \cref{definition: psi} and $\phi$ is as in \cref{definition: phi}.
\end{lemma}

\begin{proof}
Since $\mu$-filtered colimits commute with $\mu$-small limits in $\mathcal{S}$ and since limits and colimits in a category of presheaves are computed point-wise we get that $\mu$-filtered colimits commute with $\mu$-small limits in $\cC$. Furthermore, as limits and $\mu$-filtered colimits in $\cC_{A/}$ are computed in $\cC$, we get that $\mu$-filtered colimits commute with $\mu$-small limits in $\cC_{A/}$. Now, since $J$ is $\mu$-filtered and $\mu > \omega$, we have
\[
\underset{\Delta^+}\lim \ \phi(Z_{\prod})=\underset{\Delta^+}\lim \ \underset{J}\colim \ F_*H^*(Z_{\prod}) \simeq\underset{J}\colim \ \underset{\Delta^+}\lim \ F_*H^*(Z_{\prod})=\underset{J}\colim \ \psi(Z_{\prod}).
\]
\end{proof}

We can finally prove the claim.

\begin{lemma}\label{lemma: A retract}
Let $X:\Delta^+ \to \cC_{A/}$ be as in \cref{notation: X Y} i.e. $X= \phi(Z)$ where $Z:L \to \cC_{A/}$ is a functor that preserves the initial object and $\phi$ is as in \cref{definition: phi}. Then $(\Id\colon A\to A)$ is a retract of $\underset{\Delta^+}\lim \ X $.
\end{lemma}

\begin{proof}
Let $i:* \to L$ be the map that chooses $-\infty$ and
\[
\eta:Z \to i_*i^*Z=Z_{\prod A}.
\]
be the unit map. Applying $\phi$ from \cref{definition: phi} to the above map, yields a morphism   
\[
\phi(\eta): X \to \phi(Z_{\prod})
\]
in $\Fun( \Delta^+,\cC_{/A})$. By taking limits, and using  \cref{lemma: colim lim} and \cref{corollary: lim A} we get a map $\underset{\Delta^+}\lim \ X \to (\Id\colon A\to A)$. Since $(\Id\colon A\to A)$ is the initial object in $\cC_{A/}$ it follows that it is a retract of $\underset{\Delta^+}\lim \ X$.
\end{proof}

The proposition now follows from \cref{lemma: why enough ret} and \cref{lemma: A retract}.
\end{proof}

\section{Factoring Through Pure Morphisms}

In this section we show that every map in a category of presheaves can be factored through a $\mu$-pure morphism for $\mu$ large enough, and that we can bound the size \footnote{In a presentable category, we say that an object $C \in \cC$ is of size $\kappa$, where $\kappa$ is a cardinal, if $\kappa$ is the minimal cardinal such that $C$ is $\kappa$-compact.} of the source of the factorization, in terms of the size of the source of the original map (for precise formulation, see \cref{proposition: prop5}). \\
We will first need some preliminary lemmas.

\begin{notation}
For any two cardinals $\kappa$ and $\mu$, we will denote $\kappa^{<\mu}:= \sum_{\alpha < \mu} \kappa^\alpha$.
\end{notation}

\begin{lemma} \label{lemma: lem11}
Let $K$ be a $\kappa$-small simplicial set. Then, there exists a $\max(\kappa,\omega)$-small quasi-category $\mathcal{C}$ and a categorical equivalence $K \to \mathcal{C}$.  
\end{lemma}

\begin{proof}
We shall build a fibrant replacement for $K$ using the small object argument.\\
Let $S$ be the set of inner horn inclusions:
\[
S:= \{\Lambda^n_i \hookrightarrow \Delta^n \ | \ 0<i<n \ n \in \mathbb{N} \}
\]
Define $E^0(K):=K$, and let $\theta_S^0$ be the set of maps from inner horns to $E^0(K)$. Define $E^1(K)$ as the pushout in $\sset$
\[
\begin{tikzcd}
\underset{\theta_S^0}\coprod \Lambda^n_i \arrow[d] \arrow[r] & E^0(K) \arrow[d] \\
\underset{\theta_S^0}\coprod \Delta^n \arrow[r]              & E^1(K)          
\end{tikzcd}
\]
and repeat the process to get a filtered diagram:
\[
\begin{tikzcd}
K \arrow[r] & E^1(K) \arrow[r] & E^2(K) \arrow[r] & \cdots 
\end{tikzcd}.
\]
We denote the colimit of the above diagram in $\sset$ by $\mathcal{C}$. By the small object argument, $\mathcal{C}$ is a quasi-category and the natural map $K \to \mathcal{C}$ is a categorical equivalence.\\

It remains to analyse the cardinality of each object in the above (filtered) diagram. First, note that there are $\omega$ horns and for a specific horn, $\Lambda^n_i$, the number of maps $ \Lambda^n_i \to K$ is less than $\kappa^{<\omega}=\max\{\kappa,\omega\}$. So, the size of $\theta_S^0$ is strictly less than $\kappa^{<\omega}=\max\{\kappa,\omega\}$. Now, as each map in $\theta_S^0$ adds less than $\omega$ non degenerate simplices, $E^1(K)$ is $\kappa + \omega\max\{\kappa,\omega\}=\max\{\kappa,\omega\}$-small. Repeating the argument we conclude that $E^i(K)$ is $\max\{\kappa,\omega\}$-small. All in all we conclude that $\mathcal{C}$ is $\max\{\kappa,\omega\}$-small.
\end{proof}

\begin{lemma} \label{lemma : Un compact}
Let $\kappa> \omega$ be a cardinal and $\cC$ be $\kappa$-compact object of $ \Cat_\infty$. Let $F: \cC \to \mathcal{S}$ be a functor that factors through $\mathcal{S}^\kappa$. Then, $\mathrm{Un}(F)$ is also a $\kappa$-compact object of $ \Cat_\infty$ where $\mathrm{Un}(F)$ is the unstraightening of $F$.  
\end{lemma}

\begin{proof}
By \cite[1.1]{GepnerHaugsengNikolaus15} we have the following chain of equivalences:
\[
\mathrm{Un}(F) \simeq  \mathrm{oplax-colim} \ F \simeq \underset{\mathrm{Tw}(\mathcal{C})}{\colim} \  F(-) \times \mathcal{C}_{-/}.
\]
We claim that the colimit above is a $\kappa$-small colimit of $\kappa$-compact objects. Indeed from the definition of the over category the $n$ simplices of $\cC_{c/}$ are $n+1$ simplices in $\cC$ that begins in $c$ and thus by \cite[5.4.1.2]{lurie2009higher} $\cC_{c/}$ is $\kappa$-compact. Invoking \cite[5.4.1.2]{lurie2009higher} again, we get that $F(c) \times \mathcal{C}_{c/}$ is $\kappa$-compact for all $c$. Hence, recalling that 
\[
\mathrm{Tw}(\cC)_n=\underline{\Hom_{\seset}} ((\Delta^n)^{\op}\ast(\Delta^n),\cC)
\]
where $\underline{\Hom_{\seset}}$ is the $\Hom$ set and invoking \cite[5.4.1.2]{lurie2009higher} again, we can write $\mathrm{Un}(F)$ as a $\kappa$-small colimit of $\kappa$-compact objects. Since $\kappa$-compact objects are closed under $\kappa$-small colimits, it follows that $\mathrm{Un}(F)$ is $\kappa$-compact.
\end{proof}

\begin{remark} \label{remark: gamma^<mu}
Note that $(\kappa^{<\mu})^{< \mu}=\kappa^{< \mu}$ for $\kappa \geq \mu$ and $\mu$ regular.
\end{remark}

\begin{lemma} \label{lemma: lem12}
Let $I$ be $\mu$-small simplicial set, where $\mu$ is a regular cardinal and let $F:I \to \mathcal{S}$ be a functor that factors through $\mathcal{S}^\kappa$. Assume further that $\mu, \kappa > \omega$. Then, $\lim F$ is $\gamma$-compact, where $\gamma:=\max\{\mu^{<\mu}, \ \kappa^{<\mu}\}$.   
\end{lemma}

\begin{proof}
From \cite[3.3.3.4]{lurie2009higher} $\lim F$ is the fiber over $\Id$ of the map $\Map(I,\mathrm{Un}(F)) \to \Map(I,I)$
where again, $\mathrm{Un}(F)$ denotes the unstraightening of $F$. Thus, by \cite[5.4.1.5]{lurie2009higher} and the long exact sequence in homotopy groups, it suffices to show that $\Map(I,\mathrm{Un}(F))$ and $\Map(I,I)$ are $\gamma$-compact.\\

We will show that $\Map(I,\mathrm{Un}(F))$ and $\Map(I,I)$ are $\gamma$-compact by finding simplicial models with less than $\gamma$ simplices in each simplicial degree. First, by \cref{lemma : Un compact} $\mathrm{Un}(F)$ is $\max\{\kappa,\mu \}$-compact object of $\Cat_\infty$. Therefore, by \cite[5.4.1.2]{lurie2009higher} there exists a simplicial model $\mathrm{Un}(F)' \in \sset$ for $\mathrm{Un}(F)$, with less than  $\max\{\kappa,\mu \}$ simplices in each simplicial degree. Invoking \cref{lemma: lem11}, we get quasi-category models $I', \ \mathrm{Un}(F)'' \in \sset$, for $I$ and $\mathrm{Un}(F)$, that have less than $\mu^\omega, \ \max\{\kappa^\omega,\mu^\omega \}$ simplices in each simplicial degree respectively. Second, by \cite[1.2.7.3]{lurie2009higher}
\[
\Hom_{\sset}(I,I') \quad \textrm{and} \quad \Hom_{\sset}(I,\mathrm{Un}(F)'')
\]
are simplicial models for
\[
\Fun(I,I) \quad \textrm{and} \quad  \Fun(I,\mathrm{Un}(F))
\]
respectively (here, $\Hom_{\sset}$, denotes the inner-$\Hom$ in $\sset$). Lastly, as the $n$-simplices of the inner-$\Hom$ in simplicial sets are given by $\underline{\Hom_{\sset}}(\Delta^n \times (-), (-))$
(where $\underline{\Hom_{\sset}}(-,-)$ is the $\Hom$ set) and we have the following identity
 \[
 \max\{\mu^{<\mu}, \ \kappa^{<\mu}\}=\max\{(\mu^\omega)^{<\mu}, \ (\kappa^\omega)^{<\mu}\},
 \]
we have found models for $\Fun(I,\mathrm{Un}(F))$ and $\Fun(I,I)$ with less than $\gamma$ simplices in each simplicial degree. We conclude that both $\Map(I,\mathrm{Un}(F))$ and $\Map(I,I)$ are $\gamma$-compact.
\end{proof}

\begin{lemma}\label{lem: lim ordinal non empty}

Let $\kappa$ be an arbitrary ordinal and let $X : \kappa^\mathrm{op} \to \mathcal{S}$ be a space-valued diagram. Suppose each “transition map”
\[
X_\beta \to \lim_{\gamma < \beta} X_\gamma
\]
is $\pi_0$-surjective. Then $\lim_\alpha X$ is non-empty.
\end{lemma}

\begin{proof}
    This is \cite[Theorem 1.1]{OrdinalsRamzi}.
\end{proof}

\begin{construction}\label{cons: ordianl diag}
Let $\cE$ be a small category and denote $\cC := \Fun(\cE, \mathcal{S})$. Let $\kappa$ be a regular cardinal for which $\cC$ is $\kappa$-presentable and $\cE$ is $\kappa$-small. Let $\pi$ be the cardinal from \cref{lemma: lem1}. For any cardinal $\mu \geq \pi$, let $\gamma'$ be such that $\cC^\mu$ is $\gamma'$-small, and define
\[
\gamma := \left( \max\left\{ \gamma'^{< \gamma'},\, \mu^{< \mu} \right\} \right)^+.
\]
We also fix a set of representatives of $\pi_0\left( \cC^{\gamma}_{/B} \right)^\simeq$. Given this data, we construct the following:

\begin{enumerate}
    \item For each $f: A \to B \in \pi_0(\cC^{\gamma}_{/B})^\simeq$, a map $g: A_f \to B$ in $\cC^\gamma_{/B}$ factoring $f$. 
    
    Let $D_f$ be the set of representatives of spans
    \[
    \begin{tikzcd}
    & A' \arrow[ld, "u"'] \arrow[d, "g'"] \\
    A & B'                                 
    \end{tikzcd}
    \]
    where $A'$ and $B'$ are $\mu$-compact and there exists a factorization
    \[
    \begin{tikzcd}
                       & A' \arrow[ld, "u"'] \arrow[d, "g'"] \\
    A \arrow[rd, "f"'] & B' \arrow[d, "v"]                   \\
                       & B                                  
    \end{tikzcd}.
    \]
    We denote elements of $D_f$ by
    \[
    j \ = \  \begin{tikzcd}
    {A^j} & {} \\
    {A} & {B^j}
    \arrow[from=1-1, to=2-1]
    \arrow[from=1-1, to=2-2]
    \end{tikzcd}
    \]
    Given $j \in D_f$, define $A_j := A \underset{A^j}{\coprod} B^j$, viewed as an object of $(\cC_{A/})_{/B}$.

    Let
    \[
    A_f := \coprod_{j \in D_f} A_j
    \]
    where the coproduct is taken in $(\cC_{A/})_{/B}$. By construction, there is a canonical map $i: A \to A_f$, and we may choose a map $g: A_f \to B$ fitting into a $2$-simplex:
    \[
    \begin{tikzcd}
	A \arrow[r, "i"] \arrow[dr, "f"'] & A_f \arrow[d, "g"] \\
	& B
    \end{tikzcd}
    \]
    Therefore, this defines a functor
    \[
    T: \pi_0(\cC^{\gamma}_{/B})^\simeq \to (\cC_{/B})^{\Delta^1}.
    \]

    We now show that $T$ factors through $(\cC^\gamma_{/B})^{\Delta^1}$. It suffices to show that $A_f \in \cC^\gamma$.

    First, fix $A' \in \cC^\mu$. Then $\Map_\cC(A', A)$ is $\gamma$-compact. Indeed, writing $A'$ as a $\mu$-small colimit of representables and applying the Yoneda lemma yields:
    \[
    \Map_\cC(A', A) \simeq \Map_\cC\left( \colim \Map_\cE(e, -), A \right) \simeq \lim A(e),
    \]
    which is a $\mu$-small limit of $\gamma$-compact objects, hence $\gamma$-compact by \cref{lemma: lem12} and \cref{lemma: lem1}.

    Thus, for each fixed $A' \xrightarrow{g'} B'$, the number of spans
    \[
    \begin{tikzcd}
    & A' \arrow[ld, "u"'] \arrow[d, "g'"] \\
    A & B'                                 
    \end{tikzcd}
    \]
    is less than $\gamma$. Moreover, since $\cC^\mu$ is $\gamma$-small, the number of such $A' \xrightarrow{g'} B'$ is also $< \gamma$. Hence, $|D_f| < \gamma$, and $A_f$ is a $\gamma$-small colimit of $\gamma$-compact objects. Therefore, $A_f \in \cC^\gamma$.

    \item For every $f \in \pi_0(\cC^{\gamma}_{/B})^\simeq$, a diagram
    \[
    F_f \in \Fun(\mu, \cC^\gamma_{/B}).
    \]

    For each $\beta \leq \mu$, consider the pullback
    \[
    \begin{tikzcd}
	P \arrow[r] \arrow[d] & \Fun(\beta_\delta, \pi_0(\cC^{\gamma}_{/B})^\simeq) \arrow[d] \\
	\Fun(\beta, \cC^{\gamma}_{/B}) \arrow[r] & \prod_{\alpha\to \alpha+1\in \beta} (\cC^{\gamma}_{/B})^{\Delta^1}
    \end{tikzcd}
    \]
    where the right vertical functor evaluates on objects and applies $T$, and the bottom map evaluates on morphisms of $\beta$.

    Let $X_\beta \subseteq P^{\simeq}$ denote the full subspace spanned by those points whose underlying map $\varphi: \beta \to \cC^{\gamma}_{/B}$ sends limit ordinals to colimits and for which $\varphi(0)$ is isomorphic to $f$.

    We aim to show
    \[
    \lim_{\beta < \mu} X_\beta \neq \emptyset.
    \]

    We verify that the system satisfies the hypotheses of \cref{lem: lim ordinal non empty}.

    \begin{itemize}
        \item If $\beta$ is a limit ordinal, the condition follows because the relevant comparison map is an equivalence.
        \item If $\beta = \beta_0 + 1$:
        \begin{itemize}
            \item If $\beta_0$ is a successor, then by \cref{proposition: prop2}, we have a pushout square
            \[
            \begin{tikzcd}
        	\Delta^0 \arrow[r, "\mathrm{source}"] \arrow[d, "\beta_0 - 1"'] & \Delta^1 \arrow[d] \\
        	\beta_0 \arrow[r] & \beta_0 + 1
            \end{tikzcd}
            \]
            Given $\varphi \in X_{\beta_0}$, we apply $T$ to an object isomorphic to $\varphi(\beta_0 - 1)$. Using the pushout above, this defines a section to the canonical map
            \[
            \pi_0(X_\beta) \to \pi_0(X_{\beta_0}).
            \]
            \item If $\beta_0$ is a limit ordinal and $\varphi \in X_{\beta_0}$, then applying $T$ to an object isomorphic to $\colim_{\beta_0} \varphi$ yields a diagram indexed by $\beta_0 * \Delta^0$. This defines a section the canonical map
            \[
            \pi_0(X_\beta) \to \pi_0\left( \lim_{\alpha < \beta_0} X_\alpha \right).
            \]
        \end{itemize}
    \end{itemize}
\end{enumerate}

\end{construction}

\begin{proposition} \label{proposition: prop5}
Let $\cE$ be a small category and denote $\mathcal{C}:=\Fun(\cE,\mathcal{S})$. Let $\kappa$ be a regular cardinal for which $\cC$ is $\kappa$-presentable and $\cE$ $\kappa$-small. Let $\pi$ be the cardinal from \cref{lemma: lem1}.
Then, for every regular cardinal $\mu \geq \pi$ there exists $\gamma \geq \mu$ such that for every map $f:A \to B$, where $A$ is $\gamma$-compact in $\cC$, there exists a $2$-simplex in $\cC$
\[
\begin{tikzcd}
    A & B \\
    {A'}
    \arrow["{f'}"', from=2-1, to=1-2]
    \arrow["f", from=1-1, to=1-2]
    \arrow[from=1-1, to=2-1]
\end{tikzcd}
\]
such that $f'$ is $\mu$-pure and $A'$ is $\gamma$-compact in $\cC$.
\end{proposition}

\begin{proof}
Let $\gamma'$ be such that $\mathcal{C}^\mu$ is $\gamma'$-small, denote $\gamma:=(\max\{ \gamma'^{< \gamma'},\mu^{< \mu} \})^+$. \\
Let $F_f\in \Fun(\mu,\cC^\gamma_{/B})$ be the diagram of \cref{cons: ordianl diag} (b). For each $i\leq \mu$ we denote $\colim_{i}F_{f_{|i}}=f_i$ and the set of corresponding representatives of spans from \cref{cons: ordianl diag} (a) by $D_i$ .\\
By construction we get a factorization
\[
\begin{tikzcd}
    A & B \\
    {A_\mu}
    \arrow[from=1-1, to=2-1]
    \arrow["f", from=1-1, to=1-2]
    \arrow["{f_\mu}"', from=2-1, to=1-2]
\end{tikzcd}
\]
where $A_\mu$ is $\gamma$-compact. Hence, it remains to show that $f_\mu: A_\mu \to B$ is $\mu$-pure. Indeed, assume that we are given a commutative diagram
\[
\begin{tikzcd}
    {A'} & {B'} \\
    {A_\mu} & B
    \arrow["{u}"', from=1-1, to=2-1]
    \arrow[from=2-1, to=2-2]
    \arrow[from=1-2, to=2-2]
    \arrow[from=1-1, to=1-2]
\end{tikzcd}
\]
where $A',\ B'$ are $\mu$-compact. Because $A_\mu$ is a $\mu$-filtered colimit there exists an $i<\mu$ such that  
\[
\begin{tikzcd}
    & {A'} \\
    {A_i} & {B'}
    \arrow[from=1-2, to=2-1]
    \arrow[from=1-2, to=2-2]
\end{tikzcd} \in D_i
\]
and a $2$-simplex:
\[
\begin{tikzcd}
    {A'} \\
    {A_{i+1}} & {A_\mu}
    \arrow["{u'}"', from=1-1, to=2-1]
    \arrow["{u}", from=1-1, to=2-2]
    \arrow[ from=2-1, to=2-2]
\end{tikzcd}
\]
but, by construction we also have a $2$-simplex:
\[
\begin{tikzcd}
    {A'} \\
    {B'} & {A_{i+1}}
    \arrow[from=1-1, to=2-1]
    \arrow["{u'}", from=1-1, to=2-2]
    \arrow[from=2-1, to=2-2]
\end{tikzcd}
\]
Therefore, we get a horn $\Lambda^3_2$ in $\mathcal{C}$:
\[
\begin{tikzcd}
    {A_\mu} & {B'} \\
    {A'} & {A_{i+1}}
    \arrow[from=2-1, to=1-2]
    \arrow["{u'}"', from=2-1, to=2-2]
    \arrow[from=1-2, to=2-2]
    \arrow[from=2-1, to=1-1]
    \arrow[from=2-2, to=1-1]
    \arrow[from=1-2, to=1-1]
\end{tikzcd}
\]
which yields a $2$-simplex
\[
\begin{tikzcd}
    {A'} & {B'} \\
    {A_\mu}
    \arrow[from=1-1, to=1-2]
    \arrow[from=1-2, to=2-1]
    \arrow[from=1-1, to=2-1]
\end{tikzcd}
\]
in the quasi-category $\cC$, as desired.
 \end{proof}
 
\section{Proving the Reflection Theorem} 

In this section we will finally prove our main theorem, i.e. that a category closed under limits and sufficiently filtered colimits in a presentable category is presentable. Before doing so, we will need the following lemma:

\begin{lemma} \label{lemma: lem13}
Let $\cC$ be a $\kappa$-filtered category and $\cD\subset \cC$ a full subcategory such that every $c \in \cC$ admits a map to an object of $\cD$. Then, the inclusion $i:\cD \hookrightarrow \cC$ is cofinal.   
\end{lemma}

\begin{proof}
By Quillen’s Theorem A it suffices to show that for every $c \in \cC$, the comma category $\cD_{c/}$ is contractible. We shall do so by showing that $\cD_{c/}$ is $\kappa$-filtered. Let $F:K \to \cD_{c/}$ be a diagram, where $K$ is $\kappa$-small. Since $\cC_{c/}$ is $\kappa$-filtered, after composing with the inclusion $\cD_{c/} \hookrightarrow \cC_{c/}$, $F$ admits a cocone, which we denote by $c'$. Composing with a map $c' \to d$ for $d \in \cD$ we get a cocone on $F$, which implies that $\cD_{c/}$ is $\kappa$-filtered and therefore contractible.    
\end{proof}
 
 \begin{theorem} \label{theorem: reflaction}
Let $\mathcal{C}$ be a presentable category and let $\mathcal{D} \subset \mathcal{C}$ be a full subcategory which is closed under limits and $\kappa$-filtered colimits for some regular cardinal $\kappa$. Then, $\mathcal{D}$ is presentable.
\end{theorem}
 
 \begin{proof}
We start by reducing to the case in which $\mathcal{C}=\Fun(\mathcal{A},\mathcal{S})$ where $\mathcal{A}$ is a small category. \\
Recall that every presentable category is a reflective accessible localization of $\Fun(\mathcal{A},\mathcal{S})$ where $\mathcal{A}$ is some small category \cite[5.5.1.1]{lurie2009higher}, and as such, $\mathcal{C}$ is equivalent to a full subcategory of $\Fun(\mathcal{A},S)$ which is closed under limits and $\mu$-filtered colimits (for some regular cardinal $\mu$). As a consequence, any full subcategory of $\mathcal{C}$ which is closed under limits and $\kappa$-filtered colimits in $\cC$, is equivalent to a full subcategory of $\Fun(\mathcal{A},S)$ closed under limits and $\mu+ \kappa$-filtered colimits. Therefore, we may assume that $\mathcal{C}=\Fun(\mathcal{A},\mathcal{S})$.\\

We proceed to prove the claim, under said assumption. \\
Let $\mu$ be larger than the cardinal from \cref{lemma: lem1} (i.e. the one which corresponds to $\cC$) by enlarging $\mu$ we may also assume that:
\begin{itemize}
    \item $\mathcal{D}$ is closed under $\mu$-pure morphisms (this follows from \cref{proposition: prop4}).
    \item $\mathcal{C}$ is $\mu$-presentable (this follows from the definition of presentable category).
    \item $\mathcal{D}$ is closed under $\mu$-filtered colimits in $\cC$ (this follows from our assumption on $\cD$).
\end{itemize}
We shall denote by $\Pure_\mu (\mathcal{C})$ the wide-subcategory of $\cC$ spanned by all $\mu$-pure morphisms.\\
Let $\gamma$ be a regular cardinal bigger than the cardinal that corresponds to $\mu$ in the setting of \cref{proposition: prop5}, i.e. for $f:A \to B$ a morphism in $\cC$
where $A$ is $\gamma$-compact there exists a factorization
\[
\begin{tikzcd}
    A & B \\
    {A'}
    \arrow["{f'}"', from=2-1, to=1-2]
    \arrow["f", from=1-1, to=1-2]
    \arrow[from=1-1, to=2-1]
\end{tikzcd}
\]
with $f'$ $\mu$-pure and $A'$ $\gamma$-compact. \\

The rest of the proof of the theorem will be divided into two parts: First, we will show that for $B \in \mathcal{D}$, the category $\mathcal{P}:=\Pure_\mu (\mathcal{C})_{/B} \cap (\mathcal{C}^\gamma)_{/B}$ is a full subcategory of $(\mathcal{C}^\gamma)_{/B}$. Then, we will conclude that $\cD$ is presentable.\\

\textbf{First part - $\mathcal{P}$ is a full subcategory of $(\mathcal{C}^\gamma)_{/B}$:} Observe, that if we have a commutative diagram:
\[
\begin{tikzcd}
A \arrow[rd, "f'"'] \arrow[r, "h"] & A' \arrow[d, "f"] \\
                                   & B                
\end{tikzcd}
\]
with $f$ and $f'$ $\gamma$-pure then $h$ is also $\gamma$-pure (in fact one only needs $f'$ to be $\gamma$-pure), this implies that $\Pure_\mu (\mathcal{C})_{/B}$ is the pullback of the following diagram:
\[
\begin{tikzcd}
 & {\Fun^{\Pure}(\Delta^1,\mathcal{C})} \arrow[d, "\mathrm{target}"] \\
\Delta^0 \arrow[r, "B"']                           & \mathcal{C}                                           
\end{tikzcd}
\]
where $\Fun^{\Pure}(\Delta^1,\mathcal{C}) \subset \Fun(\Delta^1,\mathcal{C})$ is the full subcategory on pure morphisms. Since there is a natural map between the following two diagrams:
\[
\begin{tikzcd}
    &&&&& {\Fun(\Delta^1,\cC)} \\
    &&&&& {} \\
    &&& {\Fun^{\Pure}(\Delta^1,\cC)} \\
    && {\Delta^0} &&& \cC \\
    \\
    {\Delta^0} &&& \cC
    \arrow[from=6-1, to=6-4]
    \arrow[from=6-1, to=4-3]
    \arrow[from=6-4, to=4-6]
    \arrow[from=4-3, to=4-6]
    \arrow[from=3-4, to=6-4]
    \arrow[from=3-4, to=1-6]
    \arrow[from=1-6, to=4-6]
\end{tikzcd}
\]
and the induced map on limits is the natural map $\Pure_\mu (\mathcal{C})_{/B} \to (\mathcal{C})_{/B}$, it follows that $\mathcal{P}$ is a full subcategory of $(\mathcal{C}^\gamma)_{/B}$.\\

\textbf{Second Part - $\cD$ is presentable:} First, note that by definition $(\mathcal{C}^\gamma)_{/B}$ admits $\gamma$-small colimits and thus it is $\gamma$-filtered. Therefore, invoking \cref{proposition: prop5} and \cref{lemma: lem13} we have that the inclusion $\mathcal{P} \hookrightarrow (\mathcal{C}^\gamma)_{/B}$ is cofinal. Second, as $\mathcal{D}$ is closed under $\mu$-pure morphisms we have that $\mathcal{P} \subset (\mathcal{D}^\gamma)_{/B}$ (note that since $\cD$ is closed under $\gamma$-filtered colimits, a $\gamma$-compact object in $\cC$ that lies in $\cD$ is also $\gamma$-compact in $\cD$), which implies, by the above, that $B$ is $\gamma$-filtered colimit of objects in $\mathcal{D}^\gamma$. Since $B$ was arbitrary we get that every object of $\cD$ is a $\gamma$-filtered colimit of $\gamma$-compact objects of $\cD$ i.e. $\mathcal{D}$ is accessible. Finally, using our assumptions on $\mathcal{D}$, it follows that the functor $i:\mathcal{D} \hookrightarrow \mathcal{C}$ is accessible and thus satisfies the solution set condition. The adjoint functor theorem \cite[3.2.5]{NguyenRaptisScharadn18} then asserts that $i$ has a left adjoint, which implies that $\mathcal{D}$ is also co-complete.
\end{proof}

\begin{corollary}
For a category $\cC$ the following are equivalent:
\begin{enumerate}
    \item $\cC$ is presentable.
    \item $\cC$ is complete and there exist a set of $\kappa$-compact objects, $S$, for some cardinal $\kappa$, such that restricted Yoneda embedding to the full subcategory on the elements of $S$ is fully-faithful.\footnote{Some sources call such set a `dense' set of $\cC$.}  
\end{enumerate}
\end{corollary}

\begin{proof}
\textbf{$1 \implies 2$:} By definition $\cC$ is complete.  Choose a cardinal $\kappa$ such that $\cC$ is $\kappa$-presentable. Using the model for $\Ind_\kappa$ as the full subcategory of $\Fun((\cC^\kappa)^\op,\mathcal{S})$ on the functors that commutes $\kappa$-small limits given in \cite[5.3.5.4]{lurie2009higher} we see that $\cC \simeq \Ind_\kappa(\cC^\kappa)\overset{i}\hookrightarrow \Fun((\cC^\kappa)^{\op},\mathcal{S})$ can be identified with the restricted Yondea embedding to $\cC^\kappa$.\\

\textbf{$2 \implies 1$:} Choose a set $S$ as in $2$ and denote the full subcategory on $S$ by $\cC_0$. By assumption the restricted Yoneda embedding $\cC \hookrightarrow \Fun(\cC_0^\op,\mathcal{S})=:\cD$ is fully-faithful. Thus, by \cref{theorem: reflaction}, it suffices to show that $\cC$ is closed under limits and sufficiently filtered colimits in $\cD$. As the Yoneda embedding commutes with limits, $\cC$ is closed under limits in $\cD$, and since all the objects of $S$ are $\kappa$-compact, $\cC$ is also closed under $\kappa$-filtered colimits in $\cD$.
\end{proof}

\section{Applications - Recognizing Smashing Localization}

In this section we discuss the use of the \cref{theorem: reflaction} in recognizing subcategories of the ambient symmetric monoidal category $\cC$ for which inclusion admits a smashing left adjoint, as defined in \cref{defenition: smashing}. We first deal with the case where $\cC \in \CAlg(\Pr^L)$ (see \cref{theorem: smashing in C}) and then with the case $\cC=\Pr^L$ (see \cref{theorem: smashing in PrL}). 

\subsection{Recognition Result for Smashing Localizations of  $\cC \in \CAlg(\Pr^L)$}

In this subsection, we prove a necessary and sufficient  condition for when a full-subcategory of $\cC  \in \CAlg(\Pr^L)$ is classified by an idempotent algebra. Recall that we say that a morphism $u:\mathbbm{1} \to X$ in $\cC$ exhibits $X$ as an idempotent object of $\cC$, if
\[
X \simeq X \otimes \mathbbm{1} \overset{1 \otimes u}\to X \otimes X
\]
is an equivalence. By \cite[4.8.2.9]{Lurie11}, an idempotent object $u:\mathbbm{1} \to X$ admits a unique commutative algebra structure for which $u$ is the unit. Conversely, the unit $u:\mathbbm{1} \to R$ of a commutative algebra $R$ exhibits it as an idempotent object if and only if the multiplication map $R \otimes R \to R$ is an isomorphism. We call such $R$-s idempotent algebras. More precisely, the functor
$\CAlg(\cC) \to \cC_{\mathbbm{1}/}$ which forgets the algebra structure and remembers only the unit map, induces an equivalence of  categories from the full-subcategory of idempotent algebras $\CAlg^{\mathrm{idem}}(\cC)$ to the full-subcategory of idempotent objects \cite[4.8.2.9]{Lurie11}.\\

The fundamental feature of an idempotent algebra $R$ is that the forgetful functor $\Mod_R(\cC) \to \cC$ is fully faithful. Thus, it is a property of an object in $\cC$ to have the structure of an $R$-module. We shall say that $R$ classifies the property of being an $R$-module. So, if we are interested in a property of objects of $\cC$, we can ask whether this property is classified by an idempotent algebra. When this is the case, we get a universal object in $\cC$ that satisfies this property. Therefore, one can consider this subsection's results as classification results of properties classified by idempotent algebras.\\

Let us recall some definitions and known results.

\begin{definition}(\cite[2.2.1.6]{Lurie11})
Let $\cC$ be a symmetric monoidal category. Let $\iota\colon\cD \hookrightarrow \cC$ be a reflective full-subcategory and denote the left adjoint of $\iota$ by $L$. We say that $L$ is compatible with the symmetric monoidal structure if for any $L$-equivalence $f:X\to Y$ and an object $Z$, the morphism
\[
X\otimes Z \overset{f \otimes \Id}\to Y\otimes Z
\]
is an $L$-equivalence.
\end{definition}

\begin{definition} \label{defenition: smashing}
Let $\cC \in \CAlg(\Cat_\infty)$ and $\iota: \cD \hookrightarrow \cC$ a reflective subcategory. Denote the left adjoint of $\iota$ by $L$. we say that 
$L$ is a \emph{smashing localization} if there exists an object $X$ and an equivalence of functors $L\simeq L_X$ where 
\[
L_X(Y) = X \otimes Y.
\]
\end{definition}

\begin{remark}
    The definition of smashing localization used here is equivalent to the standard one. In the standard formulation, a smashing localization is given by an algebra object $X$ and the unit of and the localization functor is given by 
    \[
    Y \overset{e\otimes \Id}\to X \otimes Y
    \]
    In our setting, the assumption that $L$ is smashing localization means that $L$ is compatible with the symmetric monoidal structure on $\cC$ and therefore by \cite[2.2.1.9]{Lurie11} admits a natural structure of a symmetric monoidal functor . It follows that the unit map evaluated at $\mathbbm{1}_\cC$
    \[
     \mathbbm{1}_\cC \to L(\mathbbm{1}_\cC)=X 
    \]
    exhibits $X$ as an idempotent algebra and that the unit map is of the desired form.
\end{remark}

There is a bijective correspondence between smashing localizations and idempotent algebras.

\begin{proposition} \label{proposition: smashing is idem}
For $R \in \CAlg(\cC)$ the following are equivalent:
\begin{enumerate}
\item $R \in \CAlg^{\mathrm{idem}}(\cC)$.
\item The functor $L_R$ is a smashing localization.
\end{enumerate}
\end{proposition}
 
\begin{proof}
This is \cite[4.8.2.4]{Lurie11}.
\end{proof}
 
We will need the following lemma:

\begin{lemma} \label{lemma: compatible with mon iff}
Let $\cC$ be a closed symmetric monoidal category. Let $\iota:\cD \hookrightarrow \cC$ be a reflective full-subcategory. Then the following are equivalent:
\begin{enumerate}
\item The left adjoint of the inclusion, $L$, is compatible with the symmetric monoidal structure.

\item For all $d \in \cD$ and $c \in \cC$,  $\Hom^{\cC}(c,d) \in  \cD$ where $\Hom^{\cC}(-,-)$ is the inner-$\Hom$.
\end{enumerate}
\end{lemma}

\begin{proof}
\textbf{$1 \implies 2$:} Assume that $L$ is compatible with the symmetric monoidal structure. We will show that the natural map $\Hom^\cC(c,d) \to L\Hom^\cC(c,d) $ is an equivalence. By \cite[2.2.1.9]{Lurie11} $\cD$ is symmetric monoidal and the natural map $L(c \otimes c') \to L(L(c) \otimes L(c'))$ is an equivalence, where we can view the latter as a ``formula" for the tensor product in $\cD$. Therefore, for all $c,c' \in \cC$ and $d \in \cD$ the unit map induces an equivalence:
\[
\Map_\cC(Lc,\Hom^{\cC}(c',d)) \to  \Map_\cC(c,\Hom^{\cC}(c',d)).
\]
Indeed, by adjunction we have
\begin{gather*}
\Map_\cC(c,\Hom^{\cC}(c',d))\simeq \Map_\cC(c \otimes c',d) \simeq  \Map_\cC(L(Lc \otimes Lc'),d)\simeq \\ \simeq \Map_\cC(L(Lc\otimes c'),d) \simeq \Map_\cC(Lc\otimes c',d)\simeq  \Map_\cC(Lc,\Hom^{\cC}(c',d)).
\end{gather*}
Choosing $c=\Hom^{\cC}(c',d)$ we get that the unit of the adjunction
\[
\eta_{\Hom^{\cC}(c',d)}:\Hom^{\cC}(c',d) \to L\Hom^{\cC}(c',d)
\]
has a right inverse, which will be denoted by $v$. Since both $ \eta_{\Hom^{\cC}(c',d)} \circ v$ and $\Id_{L\Hom^{\cC}(c',d)}$ fit in the dotted line  \[\begin{tikzcd}
{\Hom^{\cC}(c',d)} \arrow[d, "{\eta_{\Hom^{\cC}(c',d)}}"'] \arrow[r, "{\eta_{\Hom^{\cC}(c',d)}}"] & {L\Hom^{\cC}(c',d)} \\
{L\Hom^{\cC}(c',d)} \arrow[ru, dotted]                                                                          &                           
\end{tikzcd}\]
we get that $v$ is also a left inverse by the universal property of $\eta$. \\

\textbf{$2 \implies 1$:} Assume that for all $d \in \cD$ and $c \in \cC$,  $\Hom^{\cC}(c,d) \in  \cD$. We will show that $L(c \otimes c') \simeq L(Lc \otimes Lc')$. Fix $c, \ c' \in \cC$ and $d \in \cD$ and note that we have the following chain of equivalences
\begin{gather*}
 \Map(L(c \otimes c'),d) \simeq  \Map(c ,\Hom^\cC(c',d)) \simeq \\ \simeq \Map(Lc ,\Hom^\cC(c',d)) \simeq  \Map(L(Lc \otimes c'),d).
\end{gather*}
So, by the Yoneda lemma, $L(c \otimes c') \simeq L(Lc \otimes c')$ and by repeating the argument we get that $L(c \otimes c') \simeq L(Lc \otimes Lc')$.
\end{proof}

\begin{corollary}\label{corollary: localization of hom}
Let $\cC$ be a closed symmetric monoidal category and let $\iota:\cD \hookrightarrow \cC$ be a reflective full-subcategory. Assume that the left adjoint of the inclusion, $L$, is compatible with the symmetric monoidal structure. Then the unit map $c \to L(c)$ induces an equivalence
\[
\Hom^\cC(Lc,d) \to \Hom^\cC(c,d)  \quad \forall \ d \in \cD
\]
where $\Hom^{\cC}$ is the inner-$\Hom$.
\end{corollary}

\begin{proof}
By \cref{lemma: compatible with mon iff} $\Hom^\cC(L(c),d), \ \Hom^\cC(c,d) \in \cD$. Hence, by the Yoneda lemma, it suffices to show that the unit map induces an equivalence:
\[
\Map(d',\Hom^\cC(Lc,d)) \simeq \Map(d',\Hom^\cC(c,d))
\]
But by adjunction we have 
\[
\Map(d',\Hom^\cC(c,d)) \simeq \Map(d'\otimes c,d) \simeq \Map(L(Ld'\otimes Lc),d) \simeq  \Map(d'\otimes Lc,d) \simeq \Map(d',\Hom^\cC(Lc,d))   
\]
\end{proof}

We now prove the main theorem of this subsection.

\begin{theorem} \label{theorem: smashing in C}
Let $\cC \in \CAlg(\Pr^L)$, and let $\iota \colon \cD \hookrightarrow \cC$ be the inclusion of a full subcategory. The following are equivalent:

    \begin{enumerate}
        \item There exists an object $R\in\CAlg^{\mathrm{idem}}(\cC)$  such that $\cD\simeq \Mod_R(\cC)$, and under this equivalence, $\iota$ identifies with the forgetful functor $\Mod_R(\cC) \to \cC$.

        \item The inclusion $\iota$ admits a left adjoint $L$, and $L$ is a smashing localization.
    
        \item The subcategory $\cD$ is closed under limits and colimits in $\cC$, and for all $d\in \cD$, $c \in \cC$, both $d \otimes c$ and $\Hom^{\cC}(c,d)$ lie in $\cD$.
    \end{enumerate}
\end{theorem}

\begin{proof}

The equivalence between $1$ and $2$ is \cref{proposition: smashing is idem} (which is just \cite[4.8.2.4]{Lurie11}) and \cite[4.8.2.10]{Lurie11}.\\

\textbf{$2 \implies 3$:} Assume that $\iota$ admits a left adjoint, $L$, and that $L$ is smashing. First, it is immediate that $\cD$ is closed under limits and $c \otimes d \in \cD$ for all $c \in \cC$ and $d \in \cD$. Second, since the tensor product in $\cC$ commutes with colimits, $\cD$ is also closed under colimits. Finally, since the localization is compatible with the symmetric monoidal structure, $\Hom^{\cC}(c,d)\in \cD$ for all $c \in \cC$ and $d \in \cD$ by \cref{lemma: compatible with mon iff}.\\

\textbf{$3 \implies 2$:} Assume that the conditions of $3$ are satisfied. \cref{theorem: reflaction} then implies that $i$ admits a left adjoint, $L$. We need to show that $L$ is smashing i.e. that there exists an object $X$ such that $L$ is given by $c \mapsto X \otimes c$. We will show that $X= L(\mathbbm{1})$ satisfies this condition. By assumption $L(\mathbbm{1}) \otimes c \in \cD $ and $\Hom^{\cC}(c,d)\in \cD$ for all $c \in \cC$ and $d \in \cD$. Therefore, it follows from \cref{corollary: localization of hom} that
\[
\Map(Lc,d) \simeq \Map(c,d) \simeq \Map(c,\Hom^\cC(L\mathbbm{1},d)) \simeq \Map(c\otimes L\mathbbm{1},d).
\]
Hence, from the Yoneda lemma $Lc \simeq L\mathbbm{1} \otimes c$ as desired.
\end{proof}

 The following situation is often of interest to us. Given a map $f:X \to Y$ between objects in a symmetric monoidal presentable category $\cC$, can we characterize the category spanned by the objects $Z$ such that $Z\otimes X \overset{\Id \otimes f}\to Z\otimes Y$ is an equivalence? By \cite[4.3.17]{Lurie18} if $X$ and $Y$ are invertible with respect to the tensor product then this category is equivalent to a category of modules over an idempotent algebra. It follows from \cref{theorem: smashing in C} that the same is true when $X$ and $Y$ are only dualizable.

\begin{corollary} \label{corollary: Inverting map between dualizable objects}
Let $\cC \in \CAlg(\Pr^L)$. Let $f:D_1 \to D_2 $ be a map between dualizable objects, and let $\cD$ be the full subcategory on objects, $X$, such that $X\otimes D_1 \overset{\Id \otimes f}\to X\otimes D_2$ is an equivalence. Then, $\cD$ is equivalent to a category of modules over an idempotent algebra in $\cC$.
\end{corollary}

\begin{proof}
We will verify that $\cD$ satisfies the equivalent conditions of \cref{theorem: smashing in C}. Since the tensor product in $\cC$ commutes with colimits in each variable, it is immediate that $\cD$ is closed under colimits. Moreover, because $D_1$ and $D_2$ are dualizable, it follows immediately that $\cD$ is also closed under limits in $\cC$. It is also evident from the definition of $\cD$ that if $X \in \cD$ and $Y \in \cC$, then $X \otimes Y \in \cD$.

It remains to show that $\Hom^\cC(Y, X) \in \cD$ for all $X \in \cD$ and $Y \in \cC$. To that end, we first show that
\[
\Hom^\cC(Y, X) \otimes D_i \simeq \Hom^\cC(Y, X \otimes D_i).
\]
Indeed, for any $T \in \cC$, we have:
\begin{gather*}
\Map(T, \Hom^\cC(Y, X) \otimes D_i)\simeq \Map(T \otimes D_i^\vee \otimes Y, X) \simeq \Map(T, \Hom^\cC(Y \otimes D_i^\vee, X)) \\
\simeq \Map(T, \Hom^\cC(Y, X \otimes D_i)),
\end{gather*}
which gives the desired equivalence.

Under this identification, the map
\[
\Hom^\cC(Y, X) \otimes D_1 \xrightarrow{\Id \otimes f} \Hom^\cC(Y, X) \otimes D_2
\]
corresponds to the map
\[
\Hom^\cC(Y, X \otimes D_1) \xrightarrow{(\Id_X \otimes f) \circ (-)} \Hom^\cC(Y, X \otimes D_2),
\]
which is an equivalence by assumption. This completes the argument.       
\end{proof}

 \subsection{Recognition Result for Smashing Localizations of  $\Pr^L$}

In this subsection, we prove a necessary and sufficient  condition for when a full-subcategory of $\Pr^L$ is classified by an idempotent algebra. In \cite{carmeli2021ambidexterity} idempotent algerbas in $\Pr^{L}$ are studied and are called ``Modes''. The result we obtain is similar to \Cref{theorem: smashing in C}, however since $\Pr^{L}$ is not itself presentable there are some additional set theoretic assumptions to consider - see \cref{theorem: smashing in PrL}.\\

We first recall some definitions and known results:

\begin{proposition}
$\Pr^L$ has a closed symmetric monoidal structure where the inner-$\Hom$ between two categories, $\cC$ and $\cD$, is the category of colimits preserving functors denoted by $\Fun^L(\cC,\cD)$. Furthermore, in this symmetric monoidal structure we have the formula:
\[
\cC \otimes \cD \simeq \Fun^R(\cC^{\op},\cD)
\]
where $\Fun^R(\cC,\cD)$ is the category of functors that admits a left adjoint.  
\end{proposition}

\begin{proof}
This is \cite[4.8.1.17]{Lurie11}
\end{proof}

\begin{definition}\cite[5.5.7.7]{lurie2009higher}
Let $\kappa$ be a regular cardinal. We denote by $\Cat_\infty^{\mathrm{Rex}(\kappa)}$ the category whose objects are categories with $\kappa$-small colimits and morphisms are functors which preserve $\kappa$-small colimits.
\end{definition}

\begin{definition} \cite[5.5.7.7]{lurie2009higher}
Let $\kappa$ be a regular cardinal. We denote by $\Pr^L_\kappa$ the category whose objects are $\kappa$-presentable categories and morphisms are functors that preserve colimits and send $\kappa$-compact objects to $\kappa$-compact objects. For $\cP \subset \Pr^L$ we denote $\cP_\kappa := \cP \cap \Pr^L_\kappa$.
\end{definition}

Though $\Pr^L$ is not presentable, $\Pr^L_\kappa$ is presentable for all $\kappa$ and
\[
\underset{\kappa}\colim \ \mathrm{Pr}^L_\kappa = \mathrm{Pr}^L
\]
where the colimit is taken in the category of huge categories. Furthermore, by \cite[5.3.2.9]{Lurie11} the natural maps
\[
\mathrm{Pr}^L_\kappa \to \mathrm{Pr}^L
\]
commutes with colimits for all $\kappa$.\\

Since $\kappa$-compact objects in a presentable category are closed under $\kappa$-small colimits we have a functor $(-)^\kappa:\Pr^L_\kappa \to \Cat_\infty^{\mathrm{Rex}(\kappa)}$ which sends a $\kappa$-accessible presentable category to its $\kappa$-compact objects. For $\kappa > \omega$ this functor is an equivalence and its left adjoint is given by freely adding $\kappa$-filtered colimits \cite[5.5.7.10]{lurie2009higher}. \\

We can summarize the above discussion as follows.  For every $\omega<\kappa <\pi$, there exists a commutative diagram
\[
\begin{tikzcd}
    {\mathrm{Pr}^L_\pi} & {\Cat_\infty^{\mathrm{Rex(\pi)}}} \\
    {\mathrm{Pr}^L_\kappa} & {\Cat_\infty^{\mathrm{Rex(\kappa)}}}
    \arrow["{(-)^\kappa}"', from=2-1, to=2-2]
    \arrow["i", from=2-1, to=1-1]
    \arrow["{\Ind^\pi_\kappa}"', from=2-2, to=1-2]
    \arrow["{(-)^\pi}", from=1-1, to=1-2]
\end{tikzcd}
\]
where the horizontal maps are equivalences and the vertical maps are left adjoints. Note that the right adjoint of $i$ is $\Ind_\kappa((-)^\pi)$ and that the right adjoint of $\Ind^\pi_\kappa$ is the forgetful functor.
In particular we get:

\begin{lemma}\label{lemma: limits in PrLk}
Let $I$ be a small category and
\[
p: I \to \mathrm{Pr}^L_\kappa.
\]
Then
\[
\lim p = \Ind_\kappa(\lim p(i)^\kappa)
\]
where the limit on the left hand side is taken in $\Pr^L_\kappa$, and the limit on the right hand side is taken in $\Cat_\infty$.
\end{lemma}

Before proving our main theorem we need to understand how compact objects behave with respect to the tensor product in $\Pr^L$.

\begin{lemma} \label{lemma: compact in tensor}
For $\cC,\cD  \in \Pr^L_\kappa$ we have
\[
(\cC  \otimes \cD )^\kappa \simeq \cC^\kappa \otimes^\kappa \cD^\kappa
\]
where $\otimes^\kappa$ is the idempotent completion of tensor product of the symmetric monoidal structure on $\Cat_{\infty}^{\mathrm{Rex}(\kappa)}$ from \cite[4.8.1.4]{Lurie11}.
\end{lemma}

\begin{proof}
By our assumption on $\cC$ and $\cD$ we may assume that $\cD =\Ind_\kappa(\cD^\kappa),\ \mathcal{C}=\Ind_\kappa(\mathcal{C}^\kappa)$. From the universal property of the tensor product in $\Pr^L$, we have a natural equivalence
\[
\Fun^L(\Ind_\kappa(\cD^\kappa)\otimes \Ind_\kappa(\cC^\kappa),\cT)\simeq \Fun^{L,L}(\Ind_\kappa(\cD^\kappa)\times \Ind_\kappa(\cC^\kappa),\cT)
\]
for all $\cT \in \Pr^L$, where $\Fun^{L,L}((-)\times (-),(-))$ is the subcategory of $\Fun((-)\times (-),(-))$ spanned by functors that commute with colimits in each variable. Using the universal property of $\Ind_\kappa$ twice we get that
\[
\Fun^{L,L}(\Ind_\kappa(\cD^\kappa)\times \Ind_\kappa(\cC^\kappa),\cT)\simeq \Fun^{\kappa,\kappa}(\cD^\kappa \times \cC^\kappa,\cT)
\]
where $\Fun^{\kappa,\kappa}((-)\times (-),(-))$ is the subcategory of $\Fun((-)\times (-),(-))$ spanned by functors that commute with $\kappa$-small colimits in each variable. But, from the universal property of the tensor product in $\Cat_{\infty}^{\mathrm{Rex}(\kappa)}$ we get that
\[
\Fun^{\kappa,\kappa}(\cD^\kappa\times \cC^\kappa,\cT)\simeq \Fun^{\kappa}(\cD^\kappa \otimes^\kappa \cC^\kappa,\cT).
\]
Finally using the universal property of $\Ind_\kappa$ again we get
\[
\Fun^{\kappa}(\cD^\kappa\otimes^\kappa \cC^\kappa,\cT) \simeq \Fun^L(\Ind_\kappa(\cD^\kappa \otimes^\kappa \cC^\kappa),\cT) \]
as desired.
\end{proof}

We will also need the following claim.

\begin{lemma} \label{lemma: genrate PrL}
$\Pr^L$ is generated under colimits by $\mathcal{S}^{\Delta^1}$.
\end{lemma}

\begin{proof}
Since the functors $\mathrm{Pr}^L_\kappa \to \mathrm{Pr}^L$ commute with colimits and every presentable category lies in $\mathrm{Pr}^L_\kappa$ for some $\kappa$, it suffices to show that $\mathcal{S}^{\Delta^1}$ generates $\mathrm{Pr}^L_\kappa$ under colimits. This is equivalent to showing that $(\mathcal{S}^{\Delta^1})^\kappa$ generates $\Cat_\infty^{\Rex(\kappa)}$ under colimits. As $\mathcal{S}$ is a retract of $\mathcal{S}^{\Delta^1}$, it suffices to show that $\Cat_\infty^{\Rex(\kappa)}$ is generated under colimits by $(\mathcal{S}^{\Delta^1})^\kappa$ and $\mathcal{S}^\kappa$.  Noting that the functors
\[
 \Map_{\Cat_\infty^{\Rex(\kappa)}}(\mathcal{S}^\kappa, -) \quad \text{and} \quad \Map_{\Cat_\infty^{\Rex(\kappa)}}((\mathcal{S}^{\Delta^1})^\kappa, -)
\]
are equivalent to
\[
((-)^\kappa)^\simeq \quad \text{and} \quad (((-)^{\Delta^1})^\kappa)^\simeq,
\]
we see that they are jointly conservative. Therefore, the claim follows by \cite[2.5]{Yanovski21}.

\end{proof}

We now prove the main theorem of this subsection

\begin{theorem} \label{theorem: smashing in PrL}
For $\iota:\cP \hookrightarrow  \Pr^L$ a full-subcategory the following are equivalent:
\begin{enumerate}
\item $\iota$ admits a left adjoint $L$ which is a smashing localization.
\item \begin{enumerate}
    \item $\cP $ is closed under colimits in $\Pr^L$.
    \item If $\cD \in \cP $, then $\cD^{\Delta^1}:=\Fun(\Delta^1,\cD) \in \cP $.
    \item There exists a regular cardinal $\kappa$ such that for all $\kappa \leq \mu \leq \pi$ if $\cD \in \cP_\pi$  
    then $\Ind_\mu(\cD^\pi) \in \cP$.
    \item  There exists a regular cardinal $\kappa$ such that for all $\kappa \leq \mu$ and $p:I \to \cP_\mu$, $\Ind_\mu(\lim p(i)^\mu) \in \cP$.
\end{enumerate}
\end{enumerate}
\end{theorem}

\begin{proof}
\textbf{$2 \implies 1$:} Assume that $\cP$ satisfies the conditions in $2$.\\
We first construct a left adjoint for $\iota$. Let $\mu \geq \kappa$. By condition (d) and  \cref{lemma: limits in PrLk} the embedding $\cP_\mu \overset{\iota_\mu}\hookrightarrow \Pr^L_\mu$ commutes with limits. Furthermore, since $\Pr^L_\mu$ is closed under colimits in $\Pr^L$, $\iota_\mu$ also commutes with colimits. Thus, since $\Pr^L_\mu$ is presentable, by \cref{theorem: reflaction} $\iota_\mu$ has a left adjoint $L_\mu$. We will show that the functors $L_\mu$ assemble into a left adjoint to $\iota$. Note that the natural map $i_\mu^\pi: \cP_\mu \to \cP_\pi $ is the restriction of $\overline{i_\mu^\pi}: \Pr^L_\mu \to \Pr^L_\pi$ to $\cP_\mu$ in the source and corestriction to $\cP_\pi$ in the target. Thus, for $\cD\in \cP_\mu$ and $\cC\in \cP_\pi$
\[
\Map_{\cP_\pi}(i_\mu^\pi(\cD),\cC)\simeq \Map_{\Pr^L_\pi}(\overline{i_\mu^\pi}(\cD),\cC) \simeq \Map_{\Pr^L_\mu}(\cD,\Ind_\mu(\cC^\pi)) \simeq \Map_{\cP_\mu}(\cD,\Ind_\mu(\cC^\pi)).
\]
 Hence, by $(c)$ we have that the right adjoint $\cP_\mu \to \cP_\pi$ is given $\Ind_\mu((-)^\pi)$. It follows that $\iota_\mu$ assemble to a natural transformation between the functors
\[
\mathrm{Pr}^L_\bullet: \mathrm{Crd}^\op_{\geq \kappa} \to \mathrm{Pr}^R, \quad\quad \mathrm{Pr}^L_\bullet(\pi \geq \mu):\mathrm{Pr}^L_\pi \overset{\Ind_\mu((-)^\pi)}\to \mathrm{Pr}^L_\mu
\]
and
\[
\cP_\bullet : \mathrm{Crd}^\op_{\geq \kappa} \to \mathrm{Pr}^R, \quad\quad \  \cP_\bullet(\pi \geq \mu):\cP_\pi \overset{\Ind_\mu((-)^\pi)}\to \cP_\mu
\]
where $\mathrm{Crd}$ is the poset of small cardinals and $\mathrm{Pr}^R$ is the category of presentable categories and accessible right adjoints. Moving to left adjoints we conclude that the $L_\mu$-s assemble to a natural transformation between the functors
\[
\mathrm{Pr}^L_\bullet, \ \cP_\bullet : \mathrm{Crd}_{\geq \kappa} \to \mathrm{Pr}^L
\]
where the maps $\Pr^L_\mu \to \Pr^L_\pi$ and $\cP_\mu \to \cP_\pi$ are $\overline{i_\mu^\pi}$ and $i_\mu^\pi$ respectively. Taking a colimit and using the description of mapping spaces in a filtered colimit of categories given by \cite[0.2.1]{rozenblyum2012filtered}, we get that the embedding $\cP \overset{i}\to \Pr^L$ has a left adjoint $L:=\underset{\mu}{\colim} \ L_\mu$.\\

We now show that $\cP$ is an ideal in $\Pr^L$ i.e. if $\cD \in \cP$ and $\cC \in \Pr^L$, then $\cC \otimes \cD \in \cP$. Fix $\cD \in \cP$ and let $A$ be the full-subcategory of $\Pr^L$ spanned by the categories $\cC$ for which $\cC \otimes \cD \in \cP$. Recall that the tensor product in $\Pr^L$ commutes with colimits in each variable, and so $A$ is closed under colimits.
Furthermore, by $(b)$:
\[
 \mathcal{S}^{\Delta^1} \otimes \cC \simeq \Fun^R((\mathcal{S}^{\Delta^1})^{\op} ,\cC ) \simeq \Fun^L(\mathcal{S}^{\Delta^1} ,\cC^{\op} )^{\op} \simeq \Fun(\Delta^1,\cC)=\cC^{\Delta^1} \in \cP \implies \mathcal{S}^{\Delta^1} \in A.
 \]
Hence, the claim follows from \cref{lemma: genrate PrL}.\\

We finally show that $L$ is smashing. By the previous paragraph it suffices to show that for all $T \in \cP$ the map  
\[
\Fun^L(X \otimes L\mathcal{S},T) \to \Fun^L(X ,T)
\]
coming from pre-composing with the unit of the adjunction, is an equivalence. Note that $\Fun^L(X \otimes L\mathcal{S},T) \simeq \Fun^L(X ,\Fun^L(L\mathcal{S},T))$ and so it suffices to prove the above for the case $X=\mathcal{S}$ i.e. to prove that the map $\Fun^L(L\mathcal{S},T) \to \Fun^L(\mathcal{S},T)$ is an equivalence. Since $L$ is a left adjoint, this map is an equivalence on the space of objects and since $T^{\Delta^1} \in \cP$ it is also an equivalence on the space of arrows. It follows that it is an equivalence. \\

\textbf{$1 \implies 2$:} Assume the inclusion $\cP \hookrightarrow \Pr^L$ admits a left adjoint, $L$, which is a smashing localization. We immediately get that $\cP$ is closed under colimits in $\Pr^L$. Furthermore $L$ is obviously compatible with the symmetric monoidal structure. Hence by \cref{lemma: compatible with mon iff}, if $\cD \in \cP $ then
\[
\Fun^L(\Fun((\Delta^1)^\op,\mathcal{S}),\mathcal{D}) \simeq \Fun(\Delta^1,\cD) \in \cP .
\]

We now show that there exists a regular cardinal $\kappa > \omega$ such that for all $\kappa \leq \mu \leq \pi$ and $\cC \in \cP_\pi$ we have that $\Ind_\mu(\cC^\pi) \in \cP$. Let $\kappa$ be a regular cardinal such that $L\mathcal{S}$ is $\kappa$-compactly generated and the unit of the adjunction $\eta_\mathcal{S}:\mathcal{S} \to L\mathcal{S}$ sends $\kappa$-compact objects to $\kappa$-compact objects. By \cref{lemma: compact in tensor} we get that $L\mathcal{S}^\mu$ is an idempotent algebra in $\Cat^{\Rex(\mu)}$ for all $\mu \geq \kappa$. Therefore, the inclusion $\iota_\mu:\cP_\mu \hookrightarrow \mathrm{Pr}^L_\mu$, has a left adjoint, $L_\mu$, given by tensoring with $L\mathcal{S}$ \footnote{Note that being an $L\mathcal{S}$-module is a property and $\mathrm{Pr}^L_\mu\to \mathrm{Pr}^L$ is strong symmetric monoidal.}. We conclude that we have a commutative diagram
\[
\begin{tikzcd}
	{\Pr^L_\mu} & {\Pr^L_\pi} \\
	{\cP_\mu} & {\cP_\pi}
	\arrow["{L_\mu}"', from=1-1, to=2-1]
	\arrow["{L_\pi}", from=1-2, to=2-2]
	\arrow["{i_\mu^\pi}"', from=2-1, to=2-2]
	\arrow["{\overline{i_\mu^\pi}}", from=1-1, to=1-2]
\end{tikzcd}
\]
where all the maps are left adjoints. By passing to right adjoints we get that the right adjoint of $i_\mu^\pi$ is given by $\Ind_\mu((-)^\pi)$ and the claim follows.\\

Let $\mu \geq \kappa$ be cardinals and $p: I \to \cP_\mu$ be a functor, we need to show that $\Ind_\mu(\lim p(i)^\mu) \in \cP$.
In the previous paragraph we saw that the inclusion $\cP_\mu \to  \mathrm{Pr}^L_\mu$ admits a left adjoint and hence $\cP_\mu$ is closed under limits in $\mathrm{Pr}^L_\mu$. Thus, by \cref{lemma: limits in PrLk} we get
\[
\lim p \simeq  \Ind_\mu(\lim p(i)^\mu) \in \cP
\]    
as desired.
\end{proof}

\begin{remark}
We note that one can substitute conditions $(c)$ and $(d)$ for condition $(c')$: 
\begin{center}
 (c')  \ \ There exists a regular cardinal $\kappa$ such that for all $\kappa \leq \pi$ if $p:I \to \cP_\pi$, then for all $\kappa \leq \mu \leq \pi$,  $\Ind_\mu(\lim p(i)^\pi) \in \cP$. 
\end{center}
By splitting condition $(c')$, we make transparent the fact that it plays different roles in the proof. Condition $(d)$ guarantees that the map $\cP_\mu \to \Pr^L_\mu$ admits a left adjoint while condition $(c)$ guarantees that the different adjoints glue.   
\end{remark}

\bibliographystyle{alpha}
\phantomsection\addcontentsline{toc}{section}{\refname}
\bibliography{main}

\end{document}